\documentclass[11pt]{amsart}

\DeclareMathOperator{\Lie}{Lie}

\DeclareMathOperator{\soc}{soc}

\usepackage{placeins}
\usepackage{amsmath,amssymb,amsthm,graphics,amscd}
\usepackage{longtable,enumitem}
\usepackage{cite,xy}
\usepackage{graphicx,geometry} 
\usepackage{epsf}
\usepackage{fancyhdr,hyperref}
\usepackage{lscape}
\usepackage[usenames,dvipsnames]{color}
\usepackage{supertabular,multicol}

\hypersetup{colorlinks=true,citecolor=Purple}

\newsavebox\ltmcbox

\newcounter{entryno}
\def\tabline{Test & \the\value{entryno} & Description\addtocounter{entryno}{1}\\}

  \renewenvironment{thebibliography}[1]{
    \begin{oldthebibliography}{#1}
      \setlength{\parskip}{0ex}
      \setlength{\itemsep}{0ex}
  }
  {
  \end{oldthebibliography}
  }
  
\setlongtables
\begin{document}

\newcounter{rownum}
\setcounter{rownum}{0}
\newcommand{\ab}{\addtocounter{rownum}{1}\arabic{rownum}}

\newcommand{\x}{$\times$}
\newcommand{\bb}{\mathbf}

\newcommand{\Ind}{\mathrm{Ind}}
\newcommand{\Char}{\mathrm{char}}
\newcommand{\hra}{\hookrightarrow}
\newtheorem{lemma}{Lemma}[section]
\newtheorem{theorem}[lemma]{Theorem}
\newtheorem*{TA}{Theorem A}
\newtheorem*{TB}{Theorem B}
\newtheorem*{TC}{Theorem C}
\newtheorem*{CorC}{Corollary C}
\newtheorem*{TD}{Theorem D}
\newtheorem*{TE}{Theorem E}
\newtheorem*{PF}{Proposition E}
\newtheorem*{C3}{Corollary 3}
\newtheorem*{T4}{Theorem 4}
\newtheorem*{C5}{Corollary 5}
\newtheorem*{C6}{Corollary 6}
\newtheorem*{C7}{Corollary 7}
\newtheorem*{C8}{Corollary 8}
\newtheorem*{claim}{Claim}
\newtheorem{cor}[lemma]{Corollary}
\newtheorem{conjecture}[lemma]{Conjecture}
\newtheorem{prop}[lemma]{Proposition}
\newtheorem{question}[lemma]{Question}
\theoremstyle{definition}
\newtheorem{example}[lemma]{Example}
\newtheorem{examples}[lemma]{Examples}
\theoremstyle{remark}
\newtheorem{remark}[lemma]{Remark}
\newtheorem{remarks}[lemma]{Remarks}
\newtheorem{obs}[lemma]{Observation}
\theoremstyle{definition}
\newtheorem{defn}[lemma]{Definition}

  \def\hal{\unskip\nobreak\hfil\penalty50\hskip10pt\hbox{}\nobreak
  \hfill\vrule height 5pt width 6pt depth 1pt\par\vskip 2mm}

\newcount\n
\n=0
\def\tablebody{}
\makeatletter
\loop\ifnum\n<100
        \advance\n by1
        \protected@edef\tablebody{\tablebody
                \textbf{\number\n.}& shortText
                \tabularnewline
        }
\repeat

\makeatletter
\let\mcnewpage=\newpage
\newcommand{\TrickSupertabularIntoMulticols}{%
  \renewcommand\newpage{%
    \if@firstcolumn
      \hrule width\linewidth height0pt
      \columnbreak
    \else
      \mcnewpage
    \fi
  }%
}
\makeatother

\renewcommand{\labelenumi}{(\roman{enumi})}
\newcommand{\Hom}{\mathrm{Hom}}
\newcommand{\Vm}{V_\mathrm{min}}
\newcommand{\Int}{\mathrm{int}}
\newcommand{\Ext}{\mathrm{Ext}}
\newcommand{\opH}{\mathrm{H}}
\newcommand{\D}{\mathcal{D}}
\newcommand{\SO}{\mathrm{SO}}
\newcommand{\Sp}{\mathrm{Sp}}
\newcommand{\SL}{\mathrm{SL}}
\newcommand{\GL}{\mathrm{GL}}
\newcommand{\OO}{\mathcal{O}}
\newcommand{\diag}{\mathrm{diag}}
\newcommand{\End}{\mathrm{End}}
\newcommand{\tr}{\mathrm{tr}}
\newcommand{\Stab}{\mathrm{Stab}}
\newcommand{\red}{\mathrm{red}}
\newcommand{\Aut}{\mathrm{Aut}}
\renewcommand{\H}{\mathcal{H}}
\renewcommand{\u}{\mathfrak{u}}
\newcommand{\Ad}{\mathrm{Ad}}
\newcommand{\N}{\mathcal{N}}
\newcommand{\C}{\mathbb{C}}
\newcommand{\Z}{\mathbb{Z}}
\newcommand{\la}{\langle}\newcommand{\ra}{\rangle}
\newcommand{\gl}{\mathfrak{gl}}
\newcommand{\g}{\mathfrak{g}}
\newcommand{\F}{\mathbb{F}}
\newcommand{\m}{\mathfrak{m}}
\renewcommand{\b}{\mathfrak{Borho}}
\newcommand{\p}{\mathfrak{p}}
\newcommand{\q}{\mathfrak{q}}
\renewcommand{\l}{\mathfrak{l}}
\newcommand{\del}{\partial}
\newcommand{\h}{\mathfrak{h}}
\renewcommand{\t}{\mathfrak{t}}
\renewcommand{\k}{\mathfrak{k}}
\newcommand{\Gm}{\mathbb{G}_m}
\renewcommand{\c}{\mathfrak{c}}
\renewcommand{\r}{\mathfrak{r}}
\newcommand{\nn}{\mathfrak{n}}
\newcommand{\s}{\mathfrak{s}}
\newcommand{\Q}{\mathbb{Q}}
\newcommand{\z}{\mathfrak{z}}
\newcommand{\pso}{\mathfrak{pso}}
\newcommand{\so}{\mathfrak{so}}
\renewcommand{\sl}{\mathfrak{sl}}
\newcommand{\psl}{\mathfrak{psl}}
\renewcommand{\sp}{\mathfrak{sp}}
\newcommand{\Ga}{\mathbb{G}_a}

\newenvironment{changemargin}[1]{%
  \begin{list}{}{%
    \setlength{\topsep}{0pt}%
    \setlength{\topmargin}{#1}%
    \setlength{\listparindent}{\parindent}%
    \setlength{\itemindent}{\parindent}%
    \setlength{\parsep}{\parskip}%
  }%
  \item[]}{\end{list}}

\parindent=0pt
\addtolength{\parskip}{0.5\baselineskip}
\newgeometry{margin=2cm}
\title{On the minimal modules for exceptional Lie algebras: Jordan blocks and stabilisers}

\author{David I. Stewart}
\address{University of Manchester, UK} \email{dis20@cantab.net {\text{\rm(Stewart)}}}
\pagestyle{plain}

\begin{abstract}Let $G$ be a simple simply-connected exceptional algebraic group of type $G_2$, $F_4$, $E_6$ or $E_7$ over an algebraically closed field $k$ of characteristic $p>0$ with $\g=\Lie(G)$.
For each nilpotent orbit $G\cdot e$ of $\g$, we list the Jordan blocks of the action of $e$ on the minimal induced module $\Vm$ of $\g$. We also establish when the centralisers $G_v$ of vectors $v\in \Vm$ and stabilisers $\Stab_G\la v\ra$ of $1$-spaces $\la v\ra\subset\Vm$ are smooth; that is, when $\dim G_v=\dim\g_v$ or $\dim \Stab_G\la v\ra=\dim\Stab_\g\la v\ra$.
\end{abstract}
\maketitle
\section{Introduction}

Let $G$ be a simply-connected exceptional algebraic group over an algebraically closed field $k$ of characteristic $p\geq 0$ with $\g=\Lie(G)$. It is a basic fact of the theory of algebraic groups that $G$ is defined over $\Z$; that is, there is a group $G_\Z$ such that after extension of scalars to $k$, one gets the group $G$. If $G$ is not of type $E_8$, then $G_\Z$ admits a non-trivial module $(\Vm)_\Z$ of smaller dimension. After reduction modulo $p$, one then gets a module $\Vm$ for $G$. In case $G=G_2,F_4,E_6$ or $E_7$ such a  module has dimension $7,26,27$ or $56$, respectively. Recall also that $G$ acts via the adjoint action on its Lie algebra;  the associated representation is called the adjoint module. With a classification of unipotent elements in hand, the Jordan block sizes of the action of unipotent elements of $G$ on the adjoint module $\g$ and minimal module $\Vm$ were computed in \cite{Law95} (see also \cite{Law98}) and have been used extensively by the mathematical community. Recall that the characteristic $p$ is good for the exceptional group $G$ if $p>3$ and if $G$ is of type $E_8$, $p>5$. In good characteristic one has a Springer morphism: a $G$-equivariant bijective map between the variety of unipotent elements of $G$ and the nilpotent cone $\N(\g)$ of $\g$. Thus the classification of orbits of nilpotent and unipotent elements is the same. At the beginning of \cite[\S3]{UGA05}, a reference to a private communication with Lawther indicates that he has checked that the Jordan blocks of unipotent elements and associated nilpotent elements on the adjoint and minimal modules are always the same in good characteristic, with a single exception: on the minimal $56$-dimensional module for $E_7$ when $p=5$, the regular nilpotent element has blocks of size $23^2,10$ whereas the regular unipotent element has blocks  $24,22,10$. A consequence of these calculations is that together with the remaining Jordan block sizes for nilpotent elements in bad characteristic found in \cite{UGA05}, the block sizes on the adjoint module are therefore known in all characteristics.  One aim of this note is to compute the Jordan block sizes of nilpotent elements on $\Vm$, which are new in bad characteristic. For completeness and ease of use, we have included the block sizes in good characteristic also.

\begin{theorem}\label{T1}The Jordan blocks of nilpotent elements $e$ on $\Vm$ are listed in Tables \ref{t2} and \ref{t3}.
\end{theorem}

As explained above, comparison with the tables in \cite{Law95} yields the following:
\begin{cor}[Lawther]\label{C1}In good characteristic, the Jordan block sizes on $\Vm$ of nilpotent and unipotent elements of the same label are the same, unless $p=5$ and $G=E_7$, where only the Jordan block sizes of the regular unipotent and nilpotent elements disagree.\end{cor}

Using the calculations in \cite{Law95}, one sees by inspection that the number of Jordan blocks of unipotent elements on the adjoint module is independent of good characteristic; this is reflecting the fact that the centralisers of unipotent elements $G_u$ are smooth. Another way of stating this is that the orbit $G\cdot u$ of $u$ is separable, or that $\Lie(C_G(u)(k))=\c_\g(u)$. The phenomenon that centralisers are usually smooth holds much more generally; see \cite{BMRT07} and more recently, \cite{Her13}. It was also noted in \cite{Law95} when the number of Jordan blocks of a unipotent element on $\Vm$ was the same as in characteristic zero. (It turned out that this held in good characteristic.) Thus in good characteristic the scheme of fixed points $(\Vm)^u$ is smooth.

We discuss the complementary question in our context, which is possibly more natural. Beforehand we must first be a little more precise about $\Vm$. Most of the time $\Vm$ is irreducible and the theory of high weights identifies $\Vm$ uniquely up to isomorphism (possibly after twisting with a graph automorphism in the case of $E_6$). The two exceptions are when $(G,p)=(F_4,3)$ or $(G_2,2)$. In this case there are essentially two ways in which one may construct a lattice in $(\Vm)_\Z$. For one of these, the resulting module  after reduction modulo $p$, henceforth $\Vm$, has a $1$-dimensional trivial module in its head with an irreducible $(\dim \Vm-1)$-dimensional socle; this is an \emph{induced}, \emph{co-standard} or \emph{dual-Weyl} module for $G$, and $\Vm^*$ is the corresponding \emph{standard} or \emph{Weyl} module for $G$. (Note that for the purposes of computing ranks of powers of matrices, hence Jordan blocks, it matters not whether one works with a module $V$ or its dual.)

With this clarification in hand, let $v\in \Vm$ and $\la v\ra=kv$ be  the $1$-space it spans over $k$. Then we establish when the stabilisers $G_v$ and $\Stab_G\la v\ra$ are smooth; by \cite[I.7.18(5)]{Jan03} this occurs precisely when $\g_v=\Lie(G_v(k))$ and $\Stab_\g\la v\ra=\Lie(\Stab_G\la v\ra(k))$, respectively. This amounts again to establishing when the orbit $G\cdot v$ (or $G\cdot \la v\ra$) is separable. 

One consequence is the following:

\begin{theorem}The stabilisers $G_v$ of vectors $v\in \Vm$ and $\Stab_G\la v\ra$ of $1$-spaces of $\Vm$ are smooth whenever $p$ is a good prime for $G$.\end{theorem}
We were unable to find any general results in this direction. Thus we ask the following (probably too general) question:
\begin{question}Let $G$ be a reductive algebraic group over an algebraically closed field $k$ and $V$ be a restricted $G$-module. Under what circumstances are the centralisers $G_U$ or stabilisers $\Stab_G(U)$ of subschemes $U\subseteq V$ smooth algebraic groups?\end{question}
Let us reiterate that the paper \cite{Her13} gives an answer to the question about centralisers when $V$ is the adjoint module for $G$; the correct condition on the characteristic $p$ of $k$ is that it be \emph{pretty good} for $G$. Also, Cartier's theorem tells us that all algebraic groups are smooth over fields of characteristic zero, so there should be some hope that there are answers which involve a bound on $p$. (Possibly $p>\dim V$ might suffice for the answer to be yes.) Note that $V$ must be restricted for the answer to the question to be interesting: if $V$ is a Frobenius twist $W^{[1]}$ of another non-trivial $G$-module $W$, then $\g$ acts trivially on $V$, hence centralises every vector, but $G$ certainly does not. Futhermore, let us underline that already the question is interesting for the case that $U$ consists of a single $k$-point of $V$. The part of the question dealing with stabilisers is much less likely to have a nice answer: see \cite{HS14} which basically gives the answer for the corresponding question about normalisers of subspaces on the adjoint module.

\subsection{Comparison with UGA VIGRE results:}\label{ugasec} During the process of making these computations, it was discovered that there are a number of errors in the tables of representatives of nilpotent orbits in \cite{UGA05}. 
In the appendix we have provided a complete list of corrected representatives derived from \cite{LS12} and \cite{Sp84} that can be used to fix the UGA VIGRE tables. 

The first set of (minor) errors involve the transcription of representatives from the MAGMA code used into LaTeX, yet the stated Jordan block sizes remain correct. These orbit representatives are 
\begin{itemize} 
\item $F_{4}(a_{3})$ in $F_{4}$;
\item $A_6^{(2)}$ in $E_8$;
\item $E_7(a_1)$ in $E_8$;
\item $E_8(a_2)$, $E_{8}(a_{3})$, $E_8(a_4)$ in $E_{8}$.
\end{itemize} 
With these corrections the UGA VIGRE representatives are correct for $p=0$ and $p\geq 5$. 

The second set of errors orbits involve the following orbits in characteristics $2$ or $3$: 
\begin{itemize} 
\item $D_{r}(a_{1})$,  $D_{r}(a_{1})+A_{s}$, $p=2$; 
\item $D_{r}(a_{2})$, $D_{r}(a_{2})+A_{s}$, $p=3$; 
\item $E_8(b_6)$, $p=2$ and $3$.
\end{itemize} 
The problem is that the aforementioned listed representatives in \cite{UGA05} break down in either characteristic $2$ or $3$, meaning that the Jordan block sizes are 
not correct in these cases. One can use the tables in the appendix for these orbits and their Jordan blocks to correct the tables in \cite{UGA05}.  
In fact, we give all adjoint Jordan block sizes in all cases for completeness, even where they are known to coincide with those in \cite{Law95} for good characteristic.

\section{Jordan blocks}

We describe the method of computation of the tables of Jordan blocks. All this was done in GAP\footnote{The code is available at \url{github.com/davistem/nilpotent_orbits_GAP}}. For convenience, we locate the relevant modules $\Vm$ in subquotients of nilradicals of parabolic subalgebras of $\g$. The general theory on the structure of nilradicals of parabolic subalgebras of reductive Lie algebras $\g=\Lie(G)$ (and the group-theoretic analogues for $G$) can be found in \cite{ABS90}. With the notation of  \cite{Bourb82} we will give a description in terms of roots. In $\g$ of type $E_8$ with Cartan subalgebra $\h$, we locate a Levi subalgebra $\l$ of type $E_7$ containing $\h$ and corresponding roots of the form \[\def\arraystretch{0.5} \arraycolsep=0pt\begin{array}{c c c c c c c}*&*&*&*&*&*&0\\&&*\end{array},\]
where each value of $*$ is taken arbitrarily so that the result is a root. 
Then the derived subalgebra $\l'$ is simple of type $E_7$ and acts by derivations on the space spanned by the $56$ positive roots of the form 
\[\def\arraystretch{0.5} \arraycolsep=0pt\begin{array}{c c c c c c c}*&*&*&*&*&*&1\\&&*\end{array}\]
An analysis of the highest weight shows that this is indeed the $56$-dimensional module $\Vm=V_{56}$. Similarly, one locates $E_6$ in $E_7$ correponding to roots $\def\arraystretch{0.5} \arraycolsep=0pt\begin{array}{c c c c c c c}*&*&*&*&*&0\\&&*\end{array}$ acting on a $27$-dimensional minimal module $\Vm=V_{27}$ corresponding to roots $\def\arraystretch{0.5} \arraycolsep=0pt\begin{array}{c c c c c c c}*&*&*&*&*&1\\&&*\end{array}$. Nilpotent orbit representatives for $E_6$ and $E_7$ in all characteristics in terms of a sum of simple root spaces are available from \cite{LS12} and \cite{Sp84}.

Now one realises all these elements in GAP. The package {\tt LieAlgebras} in the standard distribution of GAP4 will construct a simple Lie algebra $\g$ of rank $r$ and number of positive roots $n=|R^+|$ over the rationals which comes with a `canonical' Chevalley basis $B$. This has $\{B[1],\dots,B[r]\}$ being a basis for the simple root spaces, $\{B[1],\dots,B[n]\}$ being a basis for the positive root spaces; for $1\leq i\leq n$ we have  $B[n+i]$ spans the root space corresponding to the negative root to which $B[i]$ corresponds, and $\{B[2n+1],\dots,B[2n+r]\}$ spans a Cartan subalgebra. Further, if $a,b\in \g$, then the operation $\tt a*b$ returns the commutator $[a,b]$ expressed via the $B[i]$. It is a simple matter to write any nilpotent representative  $e$ as a sum of a subset of the $B[i]$ and also identify those $B[i]$ which span a basis $B_V$ of $\Vm$ as described in the previous paragraph. One may then ask GAP to compute the action of $e$ on $\Vm$ as a matrix over the basis $B_V$ by hitting each vector of $B_V$ with $e$ and re-expressing this as a linear combination of elements of $B_V$. This associates $e$ to a $(\dim\Vm\times\dim\Vm)$-matrix $M_e$ whose entries are integers by virtue of the fact that $B$ was a Chevalley basis.

The Jordan blocks of $e$ are then determined by the ranks of successive powers of $M_e$. We first run a routine which bounds the number of exceptional primes whose Jordan block structure differs from the generic Jordan block structure as seen over $\Q$. This works simply by taking the union over all primes dividing the elementary divisors of each power of $M_e$. Then the Jordan block structure is output over $\Q$, together with the Jordan block structure over $\F_p$ for any exceptional $p$. The result for $E_6$ and $E_7$ is then output in the tables below.

The situation for the $7$- and $26$-dimensional minimal  modules for the Lie algebras of type $G_2$ and $F_4$ is done similarly, but one must work just a little harder. One has that $F_4$ is a subalgebra of $E_6$, such that the $E_6$-module $V_{27}$ has restriction to $F_4$ which is $\Vm\oplus k$ unless $p=3$ and $V_{27}|F_4$ is uniserial with successive factors $k$, $\soc \Vm$, and $k$. (It is in fact a tilting module.) If $p\neq 3$ then the module $\Vm$ is irreducible and self-dual, so that for all $p$ there is a quotient isomorphic to $\Vm$ of dimension $26$. 
Now $F_4$ is located in $E_6$ by sending the simple root vector $e_{\alpha_1}$ to $e'_{\alpha_2}$, $e_{\alpha_2}$ to $e'_{\alpha_4}$, $e_{\alpha_3}$ to $e'_{\alpha_3}+e'_{\alpha_5}$ and $e_{\alpha_4}$ to $e'_{\alpha_1}+e'_{\alpha_6}$. The remaining positive root elements of $F_4$ expresed in terms of those of $E_6$ can be generated from this. If $V_{27}$ has basis $B_V$ it is straightforward to locate a vector stabilised by $F_4$ and then form the matrix of the action of each element $e$ on the quotient; this gives a $26\times 26$-matrix $M_e$. One then repeats the procedure as described above for computing the Jordan blocks.

Finally, a similar procedure locates $G_2$ in a $D_4$-Levi subalgebra of $E_6$ sending $e_{\alpha_1}$ to $e'_{\alpha_2}+e'_{\alpha_3}+e'_{\alpha_5}$ in $E_6$ and $e_{\alpha_2}$ to $e'_{\alpha_4}$ in $E_6$. Then the $8$-dimensional natural module $V_8$ for $D_4$ corresponding to roots $\def\arraystretch{0.5} \arraycolsep=0pt\begin{array}{c c c c c c}0&*&*&*&0\\&&*\end{array}$ is obtained via the roots $\def\arraystretch{0.5} \arraycolsep=0pt\begin{array}{c c c c c c}1&*&*&*&0\\&&*\end{array}$. The restriction of $V_8$ to the $G_2$-subalgebra contains a trivial submodule such that the quotient is isomorphic to $\Vm$ as before. 

\section{Smoothness of centralisers and stabilisers}\label{sec:stabs}
In this section we consider simple, simply-connected algebraic groups of type $G:=G_2$, $F_4$, $E_6$ and $E_7$ over algebraically closed fields $k$ of arbitrary characteristic acting on their minimal modules $\Vm:=V_{7}$, $V_{26}$, $V_{27}$ and $V_{56}$, of dimensions $7$, $26$, $27$ and $56$ respectively. The stabilisers and centralisers of $1$-spaces of $\Vm$ for the corresponding finite groups $G_2(\F_q)$, $F_4(\F_q)$, $E_6(\F_q)$ and $E_7(\F_q)$ are well-known to group theorists and we record the extension which gives the reduced part $(\Stab_G\la v\ra)_\red$ of the scheme-theoretic stabilisers $\Stab_G\la v\ra$ in $G$ of the $1$-space $\la v\ra\in \Vm$ in the next lemma. (Recall that for an algebraic group $K$, over an algebraically closed field $k$, not necessarily smooth, one may associate a smooth algebraic group $K_\red\subseteq K$ such that $K_\red(k)=K(k)$.) We prove, using computational methods, that $\Stab_G\la v\ra$ and $G_v$ are smooth provided $p$ is a good prime, so that for these primes, $(\Stab_G\la v\ra)_\red=\Stab_G\la v\ra$ and the Lie theoretic stabiliser is recovered as $\Lie((\Stab_G\la v\ra)_\red)$. 


\begin{lemma}\label{stabilisers}\begin{enumerate}\item If $G$ is of type $E_7$ then $G(k)$ has four orbits on the $1$-spaces of $V:=V_{56}$. If $0\neq \la v\ra\subseteq V$ then the stabiliser $(\Stab_G\la v\ra)_\red$ is a closed subgroup of $G$ isomorphic to one of:
\begin{enumerate}\item an $E_6$ parabolic subgroup of $G$; 
\item the semidirect product of an $E_6$ Levi subgroup of $G$ with an involution inducing a graph automorphism on $E_6$; 
\item a subgroup of a $D_6$-parabolic subgroup $P$ of $G$ isomorphic to $B_5T_1R_u(P)$; 
\item a subgroup of an $E_6$-parabolic $P$ of $G$ equal to the semidirect product $HR_u(P)$, where $H$ is a subgroup of the Levi subgroup $L$ of $P$ of type $F_4T_1$.\end{enumerate}

\item If $G$ is of type $E_6$ then $G(k)$ has three orbits on the $1$-spaces of $V:=V_{27}$. If $0\neq \la v\ra\subseteq V$ then the stabiliser $(\Stab_G\la v\ra)_\red$ is a closed subgroup of $G$ isomorphic to one of: \begin{enumerate}\item a $D_5$-parabolic subgroup $P$ of $G$; 
\item a subgroup of a $D_5$-parabolic subgroup $P$ equal to the semidirect product $HR_u(P)$, where $H$ is a subgroup of the Levi subgroup $L$ of $P$ of type $B_4T_1$; 
\item a subgroup of $G$ of type $F_4$.\end{enumerate}

\item If $G$ is of type $F_4$ then $G(k)$ has infinitely many orbits on the $1$-spaces of $V:=V_{26}$. If $0\neq \la v\ra\subseteq V$ and $p\neq 3$ then the stabiliser $(\Stab_G\la v\ra)_\red$ is isomorphic to one of \begin{enumerate}\item $B_4$ (one orbit);
\item a $B_3$-parabolic subgroup $P$ of $G$ (one orbit);
\item a subgroup of $P$ isomorphic to $G_2T_1\ltimes k^{14}$ (one orbit);
\item a subgroup of $P$ isomorphic to $B_3\ltimes k^{7}$ (one orbit);
\item a $28$-dimensional subgroup such that the connected component $((\Stab_G\la v\ra)_\red)^\circ$ has type $D_4.$ In this case, each orbit contains a unique element in one of the $k\setminus\{0\}$ generic $1$-spaces of the $0$-weight space of $F_4$ on $V_{26}$. 
\end{enumerate}

\item If $G$ is of type $G_2$ then $G(k)$ has two orbits on the set of $1$-spaces of $V:=V_7$. If $0\neq \la v\ra\subseteq V$ then the stabiliser $(\Stab_G\la v\ra)_\red$ is isomorphic to one of \begin{enumerate}\item a long $A_1$-parabolic subgroup of $G$;
\item $A_2.2$ for a subsystem subgroup of type $A_2$ consisting of long roots.\end{enumerate}
\end{enumerate}
\end{lemma}
\begin{proof}
Let $H:=(\Stab_G\la v\ra)_\red$. First of all, observe that $H$, as a subgroup of the $\Z$-defined embedding of $G$ into $\GL(V)$ for $V=V_{27}$ or $V_{56}$ is defined over $\Z$, hence certainly over $\bar\F_p$. We have $H_{\bar{\F}_p}(\bar{\F}_p)=\bigcup_{r\geq 0} H_{\bar{\F}_p}(\F_{p^r})$, for instance by intersecting with $\GL(V)(\F_{p^r})$. Since $H$ is smooth, the $\bar{\F}_p$-points of $H$ are dense in $H$ and so we have that the union $\bigcup_{r\geq 0} H_{\bar{\F}_p}(\F_{p^r})$ is dense in $H$.

Assume we are not in case (iii)(e) or (iv) in characteristic $2$. Then the structure as stated, in view of \cite[I.2.7]{SpSt70}, follows directly from \cite[p467 \& Table 2]{CC88} (for $E_6$ and $F_4$), \cite[Lemma 4.3]{LS87} (for $E_7$) and \cite[Prop.~2.2]{Kle88} (for $G_2$). In reading those references, note that twisted subgroups such as ${}^2A_2(q)$ occur in the presence of a stabliser $\Stab_G\la v\ra$ of the the form $H\cdot \la\tau\ra$ where $\tau$ is a graph automorphism of $H$. Then the orbit $G\cdot v$ splits into $|\opH^1(F,G_v/G_v^\circ)|$ orbits under $G(\F_{p^r})$, by \cite[I.2.7]{SpSt70}.

A little more work is necessary to understand (iii)(e). It is shown in \cite{CC88} that each $1$-space not conjugate to any previously considered is conjugate to a subspace of a `special plane' $\pi=\la e_1,e_2,e_3\ra$ (in \cite{CC88} a `plane' is a plane of $\mathbb P(V_{27})$, hence a $3$-space of $V_{27}$).  Moreover, all such special planes are conjugate under the action of $E_6$ and one may choose the $e_i$ to be weight vectors for $E_6$. Since $F_4$ is in fact the stabiliser of an element $e=e_1+e_2+e_3$ of $\pi$, it is not hard to check that the $0$-weight space for a maximal torus $F_4$ in $V_{27}/\la e\ra$ is $\pi/\la e\ra$. Then a generic $1$-space in $\pi/\la e\ra$ is $\la e_1+t\cdot e_2\ra+\la e\ra$ for $t\neq 0,1$. In light of \cite{CC88}, the stabilisers in $G$ of all generic $1$-spaces are then seen to satisfy the conditions in (iii)(e) as stated.

For (iv) in characteristic $2$, note that the stabilisers of the $1$-spaces given are both maximal smooth subgroups $H_1$ and $H_2$ of $G$ which are $\Z$-defined. After reduction modulo $2$, we will therefore have containments $H_1\subseteq (\Stab_G\la v_1\ra)_\red$ and $H_2\subseteq (\Stab_G\la v_2\ra)_\red$. It cannot be the case that either $\Stab_G\la v_1\ra$ or $\Stab_G\la v_2\ra$ is the whole of $G$, since $G$ has no fixed $1$-spaces on $\Vm$, so the isomorphism types of $(\Stab_G\la v\ra)_\red$ must be as given. To see that the number of orbits is still the same, one counts the number of elements in orbits of $1$-spaces for $G_2(q)$. If $P$ is a long root parabolic of $G_2(q)$ it is an easy check that \[|G_2(q)|\cdot\left(\frac{1}{|P(q)|}+\frac{1}{|A_2(q).2|}+\frac{1}{|{}^2A_2(q).2|}\right)=\frac{q^7-1}{q-1}\] as required.
\end{proof}

Using GAP we find that the group-theoretic and infinitesimal stabilisers of $1$-spaces correspond.
\begin{theorem}\label{inParabolic}Let $G$ be simple and simply-connected of type $E_7$ (resp.~$E_6$, $F_4$, $G_2$) and let $V=V_{56}$ (resp.~$V=V_{27}$, $V_{26}$, $V_{7}$) be a minimal-dimensional non-trivial induced module for $G$.

Then the stabilisers in $G$ of vectors and $1$-spaces of $V$ are smooth---that is, $\dim G_v=\dim\g_v$ and $\dim \Stab_G\la v\ra=\dim \Stab_\g\la v\ra$---with the following exceptions:
\begin{enumerate}\item If $(G,p)=(E_7,2)$, then $G_v$ and $\Stab_G\la v\ra$ are not smooth if $\la v\ra$ has stabiliser of type $F_4T_1\ltimes k^{26}$.
\item If $(G,p)=(E_7,2)$, then $\Stab_G\la v\ra$ is not smooth if $\la v\ra$ has stabiliser of type $E_6.2$.
\item If $(G,p)=(E_6,3)$, then $\Stab_G\la v\ra$ is not smooth if  $\la v\ra$ has stabiliser of type $F_4$.
\item If $(G,p)=(F_4,2)$, then $\Stab_G\la v\ra$ is not smooth if $\la v\ra$ has stabiliser of type $B_3\ltimes k^{7}$.
\item If $(G,p)=(F_4,3)$, then $G_v$ and $\Stab_G\la v\ra$ are not smooth if $\la v\ra$ has stabiliser of type $G_2T_1\ltimes k^{14}$.
\item If $(G,p)=(G_2,2)$, then $G_v$ and $\Stab_G\la v\ra$ are not smooth if $\la v\ra$ has stabiliser which is a long $A_1$-parabolic.
\end{enumerate}
\end{theorem}
\begin{proof}
Let $v\in V=\Vm$ and set $K:=(\Stab_G\la v\ra)_\red$. Then certainly we have a containment $\Lie(K)=\Lie(\Stab_G\la v\ra)_\red\subseteq\Stab_\g\la v\ra=:\k$. We wish to show that equality holds. For this, it suffices to show that $\dim \Stab_\g\la v\ra=\dim (\Stab_G\la v\ra)_\red$. The values of the right-hand side are provided by Lemma \ref{stabilisers}.

To prove the equality of dimensions, we work with GAP in the following way:

\begin{enumerate}\item Construct $\Vm$ in GAP as in the previous section with $\g$ contained in a Levi subalgebra $\l$ of a parabolic $\p=\l+\q\subseteq \h$ for $\h=E_6$, $E_7$ or $E_8$, such that $\Vm$ is contained as a quotient of the unipotent radical $\q$ of $\p$.
\item Search for representatives for the $G$-orbits on $V$. Since one knows the finite number of orbits in case $G$ is of type $E_6$, $E_7$ and $G_2$, one simply seeks this number of non-isomorphic stabilisers in $G$. (The papers \cite[p467]{CC88} and \cite[Lemma 4.3]{LS87} provided guidance.) For $F_4$, except for stabilisers of type $D_4$ a similar procedure works, and the stabilisers of type $D_4$ are described explicitly in Lemma \ref{stabilisers}.  Representatives are given in Table \ref{reps}.\end{enumerate}
\emph{Assume for the moment that $G$ is not of type $F_4$ or $\la v \ra$ does not have stabiliser of type $D_4$.}
\begin{enumerate}[resume]\item For each element $b$ in a Chevalley basis $B_1$ of $\g$, calculate the coefficients of $[b, v]$ re-expressed in terms of the Chevalley basis $B$ of $\h$.
\item Form the matrix $M$ of these coefficients (which is integral, by our choice of representatives in Table \ref{reps}) and calculate its elementary divisors. It turns out that unless we are in one of the exceptional cases, all elementary divisors are either $1$ or $0$, hence the rank $r$ of this matrix will not change after reduction modulo $p$.
\item We have $\dim \g-r=\dim\g_v$. If $v\in\g\cdot v$ then $\dim \g_v= \dim\Stab_\g\la v\ra-1$. Otherwise, $\dim \g_v= \dim\Stab_\g\la v\ra$. To establish which, we simply add a new line to the matrix $M$ containing the coefficients of $v$ in terms of $B$ and take its elementary divisors again.
\item It turns out that apart from the exceptional cases $\dim\g_v=\dim G_v$ and $\dim\Stab_\g\la v\ra=\dim \Stab_G\la v\ra$.
\end{enumerate}
To deal with the case where $G=F_4$ and $\la v\ra$ has stabiliser of type $D_4$, we perform a similar calculation with $\g[t]$ and working with matrices over $\Z[t]$. Let $v=e_1+te_2$ be a generic element in the $0$-weight space of $\Vm$. Form $M$ as before to get a $26\times 52$ matrix. It turns out that $28$ of these rows are identically zero and so the rank will not change after they are removed. It also turns out that for $t\neq 0,1$, the resulting $26\times 24$-matrix has, for any choice of $t$, a single non-zero entry in each row, with no two non-zero entries in a common column. Thus the rank of the matrix is $24$ in all characteristics for all choices of $t\neq 0,1$, and we are done.
\end{proof}
\begin{table}\begin{tabular}{|c|c|c|c|}\hline
$G$ & $(\Stab_G\la v\ra)_\red$ & Representative in $\q$ & $\dim \Stab_g{\la v\ra}-\dim\g_v$\\
\hline
$G_2$ & $A_2$ & $e_{\alpha_1+\alpha_3+\alpha_4}$ & 0\\
& $A_1$-parabolic & $e_{\alpha_1}$ & $1$\\
\hline
$F_4$ & $B_4$ & $\def\arraystretch{0.5} \arraycolsep=0pt
e_{\tiny\begin{array}{c c c c c c c}1&1&2&2&1&1\\&&1\end{array}}$ & $0$\\
 & $B_3T_1\ltimes k^{14}$ & $e_{\alpha_7}$ & $1$\\
  & $G_2T_1\ltimes k^{14}$ & $\def\arraystretch{0.5} \arraycolsep=0pt
e_{\alpha_7}+e_{\tiny\begin{array}{c c c c c c c}1&3&4&3&2&1\\&&2\end{array}}$& $1$\\
 & $B_3\ltimes k^7$ & $\def\arraystretch{0.5} \arraycolsep=0pt
e_{\alpha_7}+e_{\tiny\begin{array}{c c c c c c c}1&2&2&1&1&1\\&&1\end{array}}+e_{\tiny\begin{array}{c c c c c c c}1&3&4&3&2&1\\&&2\end{array}}$ & $0$ ($1$ if $p=2$)\\
  & $D_4$ & $\def\arraystretch{0.5} \arraycolsep=0pt e_{\tiny\begin{array}{c c c c c c c}1&2&2&1&1&1\\&&1\end{array}}+
t\cdot e_{\tiny\begin{array}{c c c c c c c}1&1&2&2&1&1\\&&1\end{array}}$, $t\neq 0,1$ & $0$
\\
\hline
$E_6$ & $D_5$-parabolic & $e_{\alpha_7}$ & $1$\\
& $B_4T_1\ltimes k^{16}$ & $e_{\alpha_7}+e_{\tilde\alpha}$ & $1$\\
& $F_4$ &$\def\arraystretch{0.5} \arraycolsep=0pt e_{\tiny\begin{array}{c c c c c c c}1&2&2&1&1&1\\&&1\end{array}}+
 e_{\tiny\begin{array}{c c c c c c c}1&1&2&2&1&1\\&&1\end{array}}+e_{\tiny\begin{array}{cccccc}0&1&2&2&2&1\\&&1\end{array}}$&$0$ ($1$ if $p=3$)\\\hline
$E_7$ & $E_6$-parabolic & $e_{\alpha_8}$ & $1$\\
 & $F_4T_1\ltimes k^{26}$ & \def\arraystretch{0.5} \arraycolsep=0pt $e_{\tiny\begin{array}{ccccccc}2&3&4&3&2&2&1\\&&2\end{array}}+e_{\tiny\begin{array}{ccccccc}1&3&4&3&3&2&1\\&&2\end{array}}+e_{\tiny\begin{array}{ccccccc}1&2&4&4&3&2&1\\&&2\end{array}}$ & $1$\\
& $B_5T_1\ltimes k^{1+32}$ & \def\arraystretch{0.5} \arraycolsep=0pt $e_{\tiny \begin{array}{ccccccc}2&3&5&4&3&2&1\\&&3\end{array}}+e_{\tiny\begin{array}{ccccccc}2&4&5&4&3&2&1\\&&2\end{array}}$ & $1$\\
& $E_6.2$ & $e_{\alpha_8}+e_{\tilde\alpha-\alpha_8}$ & $0$ ($1$ if $p=2$)\\
\hline
\end{tabular}\vspace{7pt}\caption{Representatives of the orbits of $E_6$ and $E_7$ on minimal modules}\label{reps}\end{table}

\begin{remarks}Except when $(G,p)=(G_2,2)$, for each exceptional case from Lemma \ref{inParabolic} we have that the group-theoretic stabiliser is one fewer dimension than that in the Lie algebra. In each case, there is an extra toral element which stabilises $v$ or $\la v\ra$ which is not in the Lie algebra of the reduced part of the group-theoretic stabiliser.

For $(G,p)=(G_2,2)$, the non-smoothness of the stabiliser of the $1$-space whose reduced part is an $A_1$-parabolic (which is a maximal smooth subgroup) implies  that there is a maximal rank subalgebra $\h=\g_v$ containing an $A_1$-parabolic subalgebra of $\g$ which is not the Lie algebra of any smooth subgroup of $G$, hence is not obtained using the Borel--de-Siebenthal algorithm. For our representative $v$, \[\Stab_\g\la v\ra=\la e_{\alpha_2},e_{\alpha_4},e_{-\alpha_1},e_{-\alpha_2},e_{-\alpha_3},e_{-\alpha_4},e_{-\alpha_5},e_{-\alpha_6},h_1,h_2\ra,\] where $\la h_1,h_2\ra$ is a Cartan subalgebra of $\g$. This is a long $A_1$-parabolic after throwing in the outstanding root subspace $\la e_{\alpha_4}\ra$ of the short $\tilde A_1$ subalgebra which commutes with the $A_1$ Levi $\la e_{\pm\alpha_2},h_1,h_2\ra$.

For $G_2$ in characteristic $2$, there are many more such maximal rank subalgebras, including one of dimension $11$. These are discussed in some generality in \cite{Leh}. One reason for the explosion in possibilities is the fact that in characteristic $2$, one has, remarkably, an isomorphism of Lie algebras $\g\cong \psl_4$.

Finally, let us remark also that when $p=2$, the Lie algebras $G_2T_1\ltimes k^{14}$ and $B_3T_1\ltimes k^7$ are isomorphic; thus, while the stabilisers of the $1$-spaces in the relevant orbits in $V_{26}$ are not isomorphic algebraic groups, their Lie algebra stabilisers are. 
\end{remarks}

\subsection*{Acknowledgements} We thank A.~Thomas, A.~Premet and D.~Nakano for helpful discussions. We thank the referee for checking the results of this paper carefully. In particular, we are extremely grateful that a discrepancy was noted in characteristic $2$, which we eventually traced back to errors in the representatives of orbits $D_r(a_i)$ in \cite{UGA05}---see \S\ref{ugasec}.
\newpage\newgeometry{left=0.5cm,right=0.5cm,top=1.1cm,bottom=1.1cm}
 \section*{Appendix A: Jordan block structures on the minimal modules }
\ 
\def\arraystretch{1.2} \arraycolsep=0pt 
\begin{multicols}{3}\setbox\ltmcbox\vbox{\makeatletter\col@number\@ne\tiny\begin{longtable}{l l l}$E_{7}$ & $2$ & $14^{4}$\\
 & $3$ & $19^{2},18$\\
 & $5$ & $23^{2},10$\\
 & $7$ & $28,14^{2}$\\
 & $19$ & $19^{2},18$\\
 & $23$ & $23^{2},10$\\
 & others & $28,18,10$
\\\hline
$E_{7}(a_{1})$ & $2$ & $14^{4}$\\
 & $3$ & $19^{2},9^{2}$\\
 & $7$ & $22,14^{2},6$\\
 & $11$ & $22,12,11^{2}$\\
 & $17$ & $17^{2},16,6$\\
 & $19$ & $19^{2},12,6$\\
 & others & $22,16,12,6$
\\\hline
$E_{7}(a_{2})$ & $2$ & $14^{2},12^{2},2^{2}$\\
 & $3$ & $18,10^{3},8$\\
 & $13$ & $13^{4},4$\\
 & $17$ & $17^{2},10,8,4$\\
 & others & $18,16,10,8,4$
\\\hline
$E_{7}(a_{3})$ & $2$ & $14^{2},8^{2},6^{2}$\\
 & $11$ & $11^{4},10,2$\\
 & $13$ & $13^{2},12,10,6,2$\\
 & others & $16,12,10^{2},6,2$
\\\hline
$E_{6}$ & $2$ & $13^{4},1^{4}$\\
 & $3$ & $9^{6},1^{2}$\\
 & $13$ & $13^{4},1^{4}$\\
 & others & $17^{2},9^{2},1^{4}$
\\\hline
$E_{6}(a_{1})$ & $3$ & $9^{6},1^{2}$\\
 & $7$ & $13^{2},7^{4},1^{2}$\\
 & $11$ & $11^{4},5^{2},1^{2}$\\
 & others & $13^{2},9^{2},5^{2},1^{2}$
\\\hline
$D_{6}$ & $2$ & $8^{4},6^{4}$\\
 & $11$ & $11^{4},10,1^{2}$\\
 & $13$ & $13^{2},11^{2},6,1^{2}$\\
 & others & $16,11^{2},10,6,1^{2}$
\\\hline
$E_{7}(a_{4})$ & $2$ & $8^{4},6^{4}$\\
 & $3$ & $9^{4},7^{2},3^{2}$\\
 & $5$ & $12,10^{2},8,5^{2},4,2$\\
 & $11$ & $11^{2},10,8,6,4^{2},2$\\
 & others & $12,10^{2},8,6,4^{2},2$
\\\hline
$D_{6}(a_{1})$ & $2$ & $8^{4},6^{4}$\\
 & $3$ & $9^{4},7^{2},3^{2}$\\
 & $5$ & $12,10,9^{2},5^{2},3^{2}$\\
 & $11$ & $11^{2},9^{2},6,4,3^{2}$\\
 & others & $12,10,9^{2},6,4,3^{2}$
\\\hline
$D_{5}+A_{1}$ & $2$ & $8^{4},6^{2},4^{2},2^{2}$\\
 & $3$ & $9^{2},8^{4},2^{3}$\\
 & others & $11^{2},10,8,5^{2},2^{3}$
\\\hline
$(A_{6})^{(2)}$ & 2 & $8^{4},6^{4}$
\\\hline
$A_{6}$ & $2$ & $7^{8}$\\
 & $3$ & $9^{4},7^{2},3^{2}$\\
 & $7$ & $7^{8}$\\
 & others & $11^{2},7^{4},3^{2}$
\\\hline
$E_{7}(a_{5})$ & $2$ & $8^{4},6^{2},4^{2},2^{2}$\\
 & $3$ & $9^{2},8,6^{3},4^{3}$\\
 & $5$ & $10,8^{2},6,5^{4},4$\\
 & $7$ & $7^{6},6,4^{2}$\\
 & others & $10,8^{2},6^{3},4^{3}$
\\\hline
$D_{5}$ & $2$ & $8^{4},5^{4},1^{4}$\\
 & $3$ & $9^{2},8^{4},1^{6}$\\
 & others & $11^{2},9^{2},5^{2},1^{6}$
\\\hline
$E_{6}(a_{3})$ & $2$ & $8^{4},5^{4},1^{4}$\\
 & $3$ & $9^{2},6^{4},3^{4},1^{2}$\\
 & $7$ & $7^{6},5^{2},1^{4}$\\
 & others & $9^{2},7^{2},5^{4},1^{4}$
\\\hline
$D_{6}(a_{2})$ & $2$ & $8^{2},6^{6},2^{2}$\\
 & $3$ & $9^{2},7^{2},6,5^{2},4^{2}$\\
 & $5$ & $10,8,7^{2},5^{4},4$\\
 & $7$ & $7^{6},5^{2},4$\\
 & others & $10,8,7^{2},6,5^{2},4^{2}$
\\\hline
$D_{5}(a_{1})+A_{1}$ & $2$ & $8^{2},6^{2},4^{6},2^{2}$\\
 & $3$ & $8^{3},6^{3},3^{2},2^{4}$\\
 & $7$ & $7^{6},4,2^{5}$\\
 & others & $8^{3},6^{3},4,2^{5}$
\\\hline
$A_{5}+A_{1}$ & $2$ & $8^{2},6^{6},2^{2}$\\
 & $3$ & $7^{2},6^{7}$\\
 & $7$ & $7^{4},6^{3},5^{2}$\\
 & others & $10,7^{2},6^{3},5^{2},4$
\\\hline
$(A_{5})'$ & $2$ & $7^{4},6^{4},1^{4}$\\
 & $3$ & $9^{2},6^{4},3^{4},1^{2}$\\
 & $7$ & $7^{4},6^{4},1^{4}$\\
 & others & $9^{2},6^{4},5^{2},1^{4}$
\\\hline
$A_{4}+A_{2}$ & $2$ & $7^{4},4^{4},3^{4}$\\
 & $3$ & $7^{2},6^{4},3^{6}$\\
 & $5$ & $5^{10},3^{2}$\\
 & others & $7^{4},5^{2},3^{6}$
\\\hline
$D_{5}(a_{1})$ & $2$ & $8^{2},5^{4},4^{4},1^{4}$\\
 & $7$ & $7^{6},3^{2},2^{2},1^{4}$\\
 & others & $8^{2},7^{2},6^{2},3^{2},2^{2},1^{4}$
\\\hline
$A_{4}+A_{1}$ & $5$ & $5^{10},2^{2},1^{2}$\\
 & others & $7^{2},6^{2},5^{2},4^{2},3^{2},2^{2},1^{2}$
\\\hline
$D_{4}+A_{1}$ & $2$ & $4^{12},2^{4}$\\
 & $7$ & $7^{6},2^{5},1^{4}$\\
 & others & $8,7^{4},6,2^{5},1^{4}$
\\\hline
$(A_{5})''$ & $2$ & $8^{2},6^{6},2^{2}$\\
 & $3$ & $7^{2},6^{7}$\\
 & $7$ & $7^{2},6^{7}$\\
 & others & $10,6^{7},4$
\\\hline
$A_{3}+A_{2}+A_{1}$ & $2$ & $4^{12},2^{4}$\\
 & $3$ & $6^{3},4^{3},3^{8},2$\\
 & $5$ & $5^{6},4^{4},2^{5}$\\
 & others & $6^{3},4^{7},2^{5}$
\\\hline
$A_{4}$ & $5$ & $5^{10},1^{6}$\\
 & others & $7^{2},5^{6},3^{2},1^{6}$
\\\hline
$(A_{3}+A_{2})^{(2)}$ & 2 & $4^{12},2^{4}$
\\\hline
$A_{3}+A_{2}$ & $2$ & $4^{8},3^{8}$\\
 & $3$ & $6^{2},5^{2},4^{2},3^{8},1^{2}$\\
 & $5$ & $5^{6},4^{2},3^{4},2^{2},1^{2}$\\
 & others & $6^{2},5^{2},4^{4},3^{4},2^{2},1^{2}$
\\\hline
$D_{4}(a_{1})+A_{1}$ & $2$ & $4^{12},2^{4}$\\
 & $3$ & $6,5^{4},4,3^{6},2^{4}$\\
 & $5$ & $5^{6},4,3^{4},2^{5}$\\
 & others & $6,5^{4},4^{2},3^{4},2^{5}$
\\\hline
$D_{4}$ & $2$ & $4^{12},1^{8}$\\
 & others & $7^{6},1^{14}$
\\\hline
$A_{3}+2A_{1}$ & $2$ & $4^{10},2^{8}$\\
 & $5$ & $5^{4},4^{4},3^{4},2^{3},1^{2}$\\
 & others & $6,5^{2},4^{5},3^{4},2^{3},1^{2}$
\\\hline
$D_{4}(a_{1})$ & $2$ & $4^{12},1^{8}$\\
 & others & $5^{6},3^{6},1^{8}$
\\\hline
$(A_{3}+A_{1})'$ & $2$ & $4^{8},3^{4},2^{4},1^{4}$\\
 & others & $5^{4},4^{4},3^{2},2^{4},1^{6}$
\\\hline
$2A_{2}+A_{1}$ & $2$ & $4^{8},3^{4},2^{4},1^{4}$\\
 & $3$ & $3^{18},1^{2}$\\
 & others & $5^{2},4^{4},3^{6},2^{4},1^{4}$
\\\hline
$(A_{3}+A_{1})''$ & $2$ & $4^{10},2^{8}$\\
 & $5$ & $5^{2},4^{8},2^{7}$\\
 & others & $6,4^{9},2^{7}$
\\\hline
$A_{2}+3A_{1}$ & $2$ & $4^{6},2^{16}$\\
 & $3$ & $3^{14},2^{7}$\\
 & others & $4^{7},2^{14}$
\\\hline
$2A_{2}$ & $2$ & $4^{4},3^{12},1^{4}$\\
 & $3$ & $3^{18},1^{2}$\\
 & others & $5^{2},3^{14},1^{4}$
\\\hline
$A_{3}$ & $2$ & $4^{8},3^{4},1^{12}$\\
 & others & $5^{2},4^{8},1^{14}$
\\\hline
$A_{2}+2A_{1}$ & $2$ & $4^{4},3^{4},2^{12},1^{4}$\\
 & $3$ & $3^{14},2^{4},1^{6}$\\
 & others & $4^{4},3^{6},2^{8},1^{6}$
\\\hline
$A_{2}+A_{1}$ & $3$ & $3^{12},2^{6},1^{8}$\\
 & others & $4^{2},3^{8},2^{8},1^{8}$
\\\hline
$4A_{1}$ & $2$ & $2^{28}$\\
 & $3$ & $3^{8},2^{13},1^{6}$\\
 & others & $4,3^{6},2^{14},1^{6}$
\\\hline
$A_{2}$ & others & $3^{12},1^{20}$
\\\hline
$(3A_{1})'$ & $2$ & $2^{24},1^{8}$\\
 & others & $3^{6},2^{12},1^{14}$
\\\hline
$(3A_{1})''$ & $2$ & $2^{28}$\\
 & $3$ & $3^{2},2^{25}$\\
 & others & $4,2^{26}$
\\\hline
$2A_{1}$ & $2$ & $2^{20},1^{16}$\\
 & others & $3^{2},2^{16},1^{18}$
\\\hline
$A_{1}$ & others & $2^{12},1^{32}$
\\\hline
\end{longtable}\unskip\unpenalty\unpenalty}\unvbox\ltmcbox\end{multicols}
\smallskip\begin{table}[!h]\caption{Jordan Blocks of nilpotent elements on the module $V_{56}$ for $E_7$\label{t2}}\end{table}
\ \newpage\ \\
\begin{multicols}{3}\tiny\begin{tabular}{l l l}$E_{6}$ & $2$ & $13^{2},1$\\
 & $3$ & $9^{3}$\\
 & $13$ & $13^{2},1$\\
 & others & $17,9,1$
\\\hline
$E_{6}(a_{1})$ & $3$ & $9^{3}$\\
 & $7$ & $13,7^{2}$\\
 & $11$ & $11^{2},5$\\
 & others & $13,9,5$
\\\hline
$D_{5}$ & $2$ & $8^{2},5^{2},1$\\
 & $3$ & $9,8^{2},1^{2}$\\
 & others & $11,9,5,1^{2}$
\\\hline
$E_{6}(a_{3})$ & $2$ & $8^{2},5^{2},1$\\
 & $3$ & $9,6^{2},3^{2}$\\
 & $7$ & $7^{3},5,1$\\
 & others & $9,7,5^{2},1$
\\\hline
$D_{5}(a_{1})$ & $2$ & $8,5^{2},4^{2},1$\\
 & $7$ & $7^{3},3,2,1$\\
 & others & $8,7,6,3,2,1$
\\\hline
$A_{5}$ & $2$ & $7^{2},6^{2},1$\\
 & $3$ & $9,6^{2},3^{2}$\\
 & $7$ & $7^{2},6^{2},1$\\
 & others & $9,6^{2},5,1$
\\\hline
$A_{4}+A_{1}$ & $5$ & $5^{5},2$\\
 & others & $7,6,5,4,3,2$
\\\hline
$D_{4}$ & $2$ & $4^{6},1^{3}$\\
 & others & $7^{3},1^{6}$
\\\hline
$A_{4}$ & $5$ & $5^{5},1^{2}$\\
 & others & $7,5^{3},3,1^{2}$
\\\hline
$D_{4}(a_{1})$ & $2$ & $4^{6},1^{3}$\\
 & others & $5^{3},3^{3},1^{3}$
\\\hline
$A_{3}+A_{1}$ & $2$ & $4^{4},3^{2},2^{2},1$\\
 & others & $5^{2},4^{2},3,2^{2},1^{2}$
\\\hline
$2A_{2}+A_{1}$ & $2$ & $4^{4},3^{2},2^{2},1$\\
 & $3$ & $3^{9}$\\
 & others & $5,4^{2},3^{3},2^{2},1$
\\\hline
$A_{3}$ & $2$ & $4^{4},3^{2},1^{5}$\\
 & others & $5,4^{4},1^{6}$
\\\hline
$A_{2}+2A_{1}$ & $2$ & $4^{2},3^{2},2^{6},1$\\
 & $3$ & $3^{7},2^{2},1^{2}$\\
 & others & $4^{2},3^{3},2^{4},1^{2}$
\\\hline
$2A_{2}$ & $2$ & $4^{2},3^{6},1$\\
 & $3$ & $3^{9}$\\
 & others & $5,3^{7},1$
\\\hline
$A_{2}+A_{1}$ & $3$ & $3^{6},2^{3},1^{3}$\\
 & others & $4,3^{4},2^{4},1^{3}$
\\\hline
$A_{2}$ & others & $3^{6},1^{9}$
\\\hline
$3A_{1}$ & $2$ & $2^{12},1^{3}$\\
 & others & $3^{3},2^{6},1^{6}$
\\\hline
$2A_{1}$ & $2$ & $2^{10},1^{7}$\\
 & others & $3,2^{8},1^{8}$
\\\hline
$A_{1}$ & others & $2^{6},1^{15}$
\\\hline
\end{tabular}\\\newcolumn\\
\begin{tabular}{l l l}
$F_{4}$ & $2$ & $13^{2}$\\
 & $3$ & $9^{2},8$\\
 & $13$ & $13^{2}$\\
 & others & $17,9$
\\\hline
$F_{4}(a_{1})$ & $2$ & $8^{2},5^{2}$\\
 & $3$ & $9,8^{2},1$\\
 & others & $11,9,5,1$
\\\hline
$F_{4}(a_{2})$ & $2$ & $8^{2},5^{2}$\\
 & $3$ & $9,6^{2},3,2$\\
 & $7$ & $7^{3},5$\\
 & others & $9,7,5^{2}$
\\\hline
$(C_{3})^{(2)}$ & 2 & $7^{2},6^{2}$
\\\hline
$C_{3}$ & $2$ & $7^{2},6^{2}$\\
 & $3$ & $9,6^{2},3,2$\\
 & $7$ & $7^{2},6^{2}$\\
 & others & $9,6^{2},5$
\\\hline
$B_{3}$ & $2$ & $4^{6},1^{2}$\\
 & others & $7^{3},1^{5}$
\\\hline
$F_{4}(a_{3})$ & $2$ & $4^{6},1^{2}$\\
 & others & $5^{3},3^{3},1^{2}$
\\\hline
$C_{3}(a_{1})^{(2)}$ & 2 & $4^{4},3^{2},2^{2}$
\\\hline
$C_{3}(a_{1})$ & $2$ & $4^{4},3^{2},2^{2}$\\
 & others & $5^{2},4^{2},3,2^{2},1$
\\\hline
$(\tilde{A}_{2}+A_{1})^{(2)}$ & 2 & $4^{4},3^{2},2^{2}$
\\\hline
$\tilde{A}_{2}+A_{1}$ & $2$ & $4^{4},3^{2},2^{2}$\\
 & $3$ & $3^{8},2$\\
 & others & $5,4^{2},3^{3},2^{2}$
\\\hline
$(B_{2})^{(2)}$ & 2 & $4^{4},3^{2},1^{4}$
\\\hline
$B_{2}$ & $2$ & $4^{4},3^{2},1^{4}$\\
 & others & $5,4^{4},1^{5}$
\\\hline
$A_{2}+\tilde{A}_{1}$ & $2$ & $4^{2},3^{2},2^{6}$\\
 & $3$ & $3^{7},2^{2},1$\\
 & others & $4^{2},3^{3},2^{4},1$
\\\hline
$\tilde{A}_{2}$ & $2$ & $4^{2},3^{6}$\\
 & $3$ & $3^{8},2$\\
 & others & $5,3^{7}$
\\\hline
$(A_{2})^{(2)}$ & 2 & $4^{2},3^{6}$
\\\hline
$A_{2}$ & others & $3^{6},1^{8}$
\\\hline
$A_{1}+\tilde{A}_{1}$ & $2$ & $2^{12},1^{2}$\\
 & others & $3^{3},2^{6},1^{5}$
\\\hline
$(\tilde{A}_{1})^{(2)}$ & 2 & $2^{10},1^{6}$
\\\hline
$\tilde{A}_{1}$ & $2$ & $2^{10},1^{6}$\\
 & others & $3,2^{8},1^{7}$
\\\hline
$A_{1}$ & others & $2^{6},1^{14}$
\\\hline
\end{tabular}\\\newcolumn\\
\begin{tabular}{l l l}
$G_{2}$ & $2$ & $4,3$\\
 & others & $7$
\\\hline
$G_{2}(a_{1})$ & others & $3^{2},1$
\\\hline
$(\tilde{A}_{1})^{(3)}$ & 3 & $3,2^{2}$
\\\hline
$\tilde{A}_{1}$ & $2$ & $2^{3},1$\\
 & others & $3,2^{2}$
\\\hline
$A_{1}$ & others & $2^{2},1^{3}$
\\\hline
\end{tabular}\end{multicols}
\smallskip\begin{table}[!h]\caption{Jordan Blocks of nilpotent elements on the module $V_{27}$ for $E_6$, $V_{26}$ for $F_4$ and $V_7$ for $G_2$\label{t3}}\end{table}
 

\newpage

\section*{Appendix B: Representatives of nilpotent orbits}

 \let\oldarraystretch\arraystretch
\let\oldarraycolsep\arraycolsep

\def\arraystretch{1} \arraycolsep=0pt 

\FloatBarrier
\begin{table}[!htb]
\begin{minipage}{.3\linewidth}\begin{center}
\begin{tabular}{l | l}
Orbit & Representative\\\hline
$G_{2}$ & $e_{\tiny\begin{array}{c c c c c c c}1&0\end{array}}+
e_{\tiny\begin{array}{c c c c c c c}0&1\end{array}}$\\
$G_{2}(a_{1})$ & $e_{\tiny\begin{array}{c c c c c c c}0&1\end{array}}+
e_{\tiny\begin{array}{c c c c c c c}3&1\end{array}}$\\
$(\tilde{A}_{1})^{(3)}$ & $e_{\tiny\begin{array}{c c c c c c c}2&1\end{array}}+
e_{\tiny\begin{array}{c c c c c c c}3&2\end{array}}$\\
$\tilde{A}_{1}$ & $e_{\tiny\begin{array}{c c c c c c c}1&0\end{array}}$\\
$A_{1}$ & $e_{\tiny\begin{array}{c c c c c c c}0&1\end{array}}$\\\end{tabular}
\end{center}\end{minipage}%
\begin{minipage}{.3\linewidth}\begin{center}
\begin{tabular}{l|l}
Orbit & Representative\\\hline
$F_{4}$ & $e_{\tiny\begin{array}{c c c c c c c}0&0&0&1\end{array}}+
e_{\tiny\begin{array}{c c c c c c c}1&0&0&0\end{array}}+
e_{\tiny\begin{array}{c c c c c c c}0&0&1&0\end{array}}+
e_{\tiny\begin{array}{c c c c c c c}0&1&0&0\end{array}}$\\
$F_{4}(a_{1})$ & $e_{\tiny\begin{array}{c c c c c c c}1&0&0&0\end{array}}+
e_{\tiny\begin{array}{c c c c c c c}0&1&0&0\end{array}}+
e_{\tiny\begin{array}{c c c c c c c}0&0&1&1\end{array}}+
e_{\tiny\begin{array}{c c c c c c c}0&1&1&0\end{array}}$\\
$F_{4}(a_{2})$ & $e_{\tiny\begin{array}{c c c c c c c}0&0&0&1\end{array}}+
e_{\tiny\begin{array}{c c c c c c c}0&0&1&1\end{array}}+
e_{\tiny\begin{array}{c c c c c c c}1&1&0&0\end{array}}+
e_{\tiny\begin{array}{c c c c c c c}0&1&2&0\end{array}}$\\
$(C_{3})^{(2)}$ & $e_{\tiny\begin{array}{c c c c c c c}0&0&0&1\end{array}}+
e_{\tiny\begin{array}{c c c c c c c}1&1&1&0\end{array}}+
e_{\tiny\begin{array}{c c c c c c c}0&1&2&0\end{array}}+
e_{\tiny\begin{array}{c c c c c c c}1&2&2&2\end{array}}$\\
$C_{3}$ & $e_{\tiny\begin{array}{c c c c c c c}0&0&0&1\end{array}}+
e_{\tiny\begin{array}{c c c c c c c}0&0&1&0\end{array}}+
e_{\tiny\begin{array}{c c c c c c c}0&1&0&0\end{array}}$\\
$B_{3}$ & $e_{\tiny\begin{array}{c c c c c c c}1&0&0&0\end{array}}+
e_{\tiny\begin{array}{c c c c c c c}0&0&1&0\end{array}}+
e_{\tiny\begin{array}{c c c c c c c}0&1&0&0\end{array}}$\\
$F_{4}(a_{3})$ & $e_{\tiny\begin{array}{c c c c c c c}0&1&0&0\end{array}}+
e_{\tiny\begin{array}{c c c c c c c}1&1&0&0\end{array}}+
e_{\tiny\begin{array}{c c c c c c c}0&1&2&0\end{array}}+
e_{\tiny\begin{array}{c c c c c c c}1&1&2&2\end{array}}$\\
$C_{3}(a_{1})^{(2)}$ & $e_{\tiny\begin{array}{c c c c c c c}1&1&0&0\end{array}}+
e_{\tiny\begin{array}{c c c c c c c}0&1&2&1\end{array}}+
e_{\tiny\begin{array}{c c c c c c c}1&1&2&2\end{array}}+
e_{\tiny\begin{array}{c c c c c c c}1&2&4&2\end{array}}$\\
$C_{3}(a_{1})$ & $e_{\tiny\begin{array}{c c c c c c c}0&1&0&0\end{array}}+
e_{\tiny\begin{array}{c c c c c c c}0&0&1&1\end{array}}+
e_{\tiny\begin{array}{c c c c c c c}0&1&2&0\end{array}}$\\
$(\tilde{A}_{2}+A_{1})^{(2)}$ & $e_{\tiny\begin{array}{c c c c c c c}0&1&1&1\end{array}}+
e_{\tiny\begin{array}{c c c c c c c}1&1&2&1\end{array}}+
e_{\tiny\begin{array}{c c c c c c c}0&1&2&2\end{array}}+
e_{\tiny\begin{array}{c c c c c c c}1&2&2&0\end{array}}$\\
$\tilde{A}_{2}+A_{1}$ & $e_{\tiny\begin{array}{c c c c c c c}0&0&0&1\end{array}}+
e_{\tiny\begin{array}{c c c c c c c}1&0&0&0\end{array}}+
e_{\tiny\begin{array}{c c c c c c c}0&0&1&0\end{array}}$\\
$(B_{2})^{(2)}$ & $e_{\tiny\begin{array}{c c c c c c c}1&1&0&0\end{array}}+
e_{\tiny\begin{array}{c c c c c c c}1&1&2&0\end{array}}+
e_{\tiny\begin{array}{c c c c c c c}0&1&2&2\end{array}}$\\
$B_{2}$ & $e_{\tiny\begin{array}{c c c c c c c}0&0&1&0\end{array}}+
e_{\tiny\begin{array}{c c c c c c c}0&1&0&0\end{array}}$\\
$A_{2}+\tilde{A}_{1}$ & $e_{\tiny\begin{array}{c c c c c c c}0&0&0&1\end{array}}+
e_{\tiny\begin{array}{c c c c c c c}1&0&0&0\end{array}}+
e_{\tiny\begin{array}{c c c c c c c}0&1&0&0\end{array}}$\\
$\tilde{A}_{2}$ & $e_{\tiny\begin{array}{c c c c c c c}0&0&0&1\end{array}}+
e_{\tiny\begin{array}{c c c c c c c}0&0&1&0\end{array}}$\\
$(A_{2})^{(2)}$ & $e_{\tiny\begin{array}{c c c c c c c}0&1&2&1\end{array}}+
e_{\tiny\begin{array}{c c c c c c c}1&1&1&1\end{array}}+
e_{\tiny\begin{array}{c c c c c c c}2&3&4&2\end{array}}$\\
$A_{2}$ & $e_{\tiny\begin{array}{c c c c c c c}1&0&0&0\end{array}}+
e_{\tiny\begin{array}{c c c c c c c}0&1&0&0\end{array}}$\\
$A_{1}+\tilde{A}_{1}$ & $e_{\tiny\begin{array}{c c c c c c c}1&0&0&0\end{array}}+
e_{\tiny\begin{array}{c c c c c c c}0&0&1&0\end{array}}$\\
$(\tilde{A}_{1})^{(2)}$ & $e_{\tiny\begin{array}{c c c c c c c}1&2&3&2\end{array}}+
e_{\tiny\begin{array}{c c c c c c c}2&3&4&2\end{array}}$\\
$\tilde{A}_{1}$ & $e_{\tiny\begin{array}{c c c c c c c}0&0&1&0\end{array}}$\\
$A_{1}$ & $e_{\tiny\begin{array}{c c c c c c c}1&0&0&0\end{array}}$\\
 \end{tabular}\end{center}\end{minipage}
 \smallskip
 \caption
 {Representatives of nilpotent elements for $G_2$ and $F_4$}\label{r1}
 \end{table}
 
 \begin{table}[!htb]
\begin{center}
\begin{tabular}{l | l}
Orbit & Representative\\\hline
$E_{6}$ & $e_{\tiny\begin{array}{c c c c c c c}1&0&0&0&0\\&&0\end{array}}+
e_{\tiny\begin{array}{c c c c c c c}0&0&0&0&0\\&&1\end{array}}+
e_{\tiny\begin{array}{c c c c c c c}0&1&0&0&0\\&&0\end{array}}+
e_{\tiny\begin{array}{c c c c c c c}0&0&1&0&0\\&&0\end{array}}+
e_{\tiny\begin{array}{c c c c c c c}0&0&0&1&0\\&&0\end{array}}+
e_{\tiny\begin{array}{c c c c c c c}0&0&0&0&1\\&&0\end{array}}$\\
$E_{6}(a_{1})$ & $e_{\tiny\begin{array}{c c c c c c c}1&0&0&0&0\\&&0\end{array}}+
e_{\tiny\begin{array}{c c c c c c c}0&1&0&0&0\\&&0\end{array}}+
e_{\tiny\begin{array}{c c c c c c c}0&0&0&1&0\\&&0\end{array}}+
e_{\tiny\begin{array}{c c c c c c c}0&0&0&0&1\\&&0\end{array}}+
e_{\tiny\begin{array}{c c c c c c c}0&0&1&0&0\\&&1\end{array}}+
e_{\tiny\begin{array}{c c c c c c c}0&1&1&0&0\\&&0\end{array}}$\\
$D_{5}$ & $e_{\tiny\begin{array}{c c c c c c c}1&0&0&0&0\\&&0\end{array}}+
e_{\tiny\begin{array}{c c c c c c c}0&0&0&0&0\\&&1\end{array}}+
e_{\tiny\begin{array}{c c c c c c c}0&1&0&0&0\\&&0\end{array}}+
e_{\tiny\begin{array}{c c c c c c c}0&0&1&0&0\\&&0\end{array}}+
e_{\tiny\begin{array}{c c c c c c c}0&0&0&1&0\\&&0\end{array}}$\\
$E_{6}(a_{3})$ & $e_{\tiny\begin{array}{c c c c c c c}1&0&0&0&0\\&&0\end{array}}+
e_{\tiny\begin{array}{c c c c c c c}0&0&1&0&0\\&&1\end{array}}+
e_{\tiny\begin{array}{c c c c c c c}0&1&1&0&0\\&&0\end{array}}+
e_{\tiny\begin{array}{c c c c c c c}0&0&0&1&1\\&&0\end{array}}+
e_{\tiny\begin{array}{c c c c c c c}0&0&1&1&0\\&&1\end{array}}+
e_{\tiny\begin{array}{c c c c c c c}0&1&1&1&0\\&&1\end{array}}$\\
$D_{5}(a_{1})$ & $e_{\tiny\begin{array}{c c c c c c c}0&0&0&0&0\\&&1\end{array}}+
e_{\tiny\begin{array}{c c c c c c c}0&0&0&1&0\\&&0\end{array}}+
e_{\tiny\begin{array}{c c c c c c c}0&0&0&0&1\\&&0\end{array}}+
e_{\tiny\begin{array}{c c c c c c c}0&0&1&0&0\\&&1\end{array}}+
e_{\tiny\begin{array}{c c c c c c c}0&1&1&0&0\\&&0\end{array}}$\\
$A_{5}$ & $e_{\tiny\begin{array}{c c c c c c c}1&0&0&0&0\\&&0\end{array}}+
e_{\tiny\begin{array}{c c c c c c c}0&1&0&0&0\\&&0\end{array}}+
e_{\tiny\begin{array}{c c c c c c c}0&0&1&0&0\\&&0\end{array}}+
e_{\tiny\begin{array}{c c c c c c c}0&0&0&1&0\\&&0\end{array}}+
e_{\tiny\begin{array}{c c c c c c c}0&0&0&0&1\\&&0\end{array}}$\\
$A_{4}+A_{1}$ & $e_{\tiny\begin{array}{c c c c c c c}1&0&0&0&0\\&&0\end{array}}+
e_{\tiny\begin{array}{c c c c c c c}0&0&0&0&0\\&&1\end{array}}+
e_{\tiny\begin{array}{c c c c c c c}0&1&0&0&0\\&&0\end{array}}+
e_{\tiny\begin{array}{c c c c c c c}0&0&1&0&0\\&&0\end{array}}+
e_{\tiny\begin{array}{c c c c c c c}0&0&0&0&1\\&&0\end{array}}$\\
$D_{4}$ & $e_{\tiny\begin{array}{c c c c c c c}0&0&0&0&0\\&&1\end{array}}+
e_{\tiny\begin{array}{c c c c c c c}0&1&0&0&0\\&&0\end{array}}+
e_{\tiny\begin{array}{c c c c c c c}0&0&1&0&0\\&&0\end{array}}+
e_{\tiny\begin{array}{c c c c c c c}0&0&0&1&0\\&&0\end{array}}$\\
$A_{4}$ & $e_{\tiny\begin{array}{c c c c c c c}1&0&0&0&0\\&&0\end{array}}+
e_{\tiny\begin{array}{c c c c c c c}0&1&0&0&0\\&&0\end{array}}+
e_{\tiny\begin{array}{c c c c c c c}0&0&1&0&0\\&&0\end{array}}+
e_{\tiny\begin{array}{c c c c c c c}0&0&0&1&0\\&&0\end{array}}$\\
$D_{4}(a_{1})$ & $e_{\tiny\begin{array}{c c c c c c c}0&0&0&0&0\\&&1\end{array}}+
e_{\tiny\begin{array}{c c c c c c c}0&1&0&0&0\\&&0\end{array}}+
e_{\tiny\begin{array}{c c c c c c c}0&1&1&0&0\\&&0\end{array}}+
e_{\tiny\begin{array}{c c c c c c c}0&0&1&1&0\\&&0\end{array}}$\\
$A_{3}+A_{1}$ & $e_{\tiny\begin{array}{c c c c c c c}1&0&0&0&0\\&&0\end{array}}+
e_{\tiny\begin{array}{c c c c c c c}0&1&0&0&0\\&&0\end{array}}+
e_{\tiny\begin{array}{c c c c c c c}0&0&1&0&0\\&&0\end{array}}+
e_{\tiny\begin{array}{c c c c c c c}0&0&0&0&1\\&&0\end{array}}$\\
$2A_{2}+A_{1}$ & $e_{\tiny\begin{array}{c c c c c c c}1&0&0&0&0\\&&0\end{array}}+
e_{\tiny\begin{array}{c c c c c c c}0&0&0&0&0\\&&1\end{array}}+
e_{\tiny\begin{array}{c c c c c c c}0&1&0&0&0\\&&0\end{array}}+
e_{\tiny\begin{array}{c c c c c c c}0&0&0&1&0\\&&0\end{array}}+
e_{\tiny\begin{array}{c c c c c c c}0&0&0&0&1\\&&0\end{array}}$\\
$A_{3}$ & $e_{\tiny\begin{array}{c c c c c c c}1&0&0&0&0\\&&0\end{array}}+
e_{\tiny\begin{array}{c c c c c c c}0&1&0&0&0\\&&0\end{array}}+
e_{\tiny\begin{array}{c c c c c c c}0&0&1&0&0\\&&0\end{array}}$\\
$A_{2}+2A_{1}$ & $e_{\tiny\begin{array}{c c c c c c c}1&0&0&0&0\\&&0\end{array}}+
e_{\tiny\begin{array}{c c c c c c c}0&0&0&0&0\\&&1\end{array}}+
e_{\tiny\begin{array}{c c c c c c c}0&1&0&0&0\\&&0\end{array}}+
e_{\tiny\begin{array}{c c c c c c c}0&0&0&0&1\\&&0\end{array}}$\\
$2A_{2}$ & $e_{\tiny\begin{array}{c c c c c c c}1&0&0&0&0\\&&0\end{array}}+
e_{\tiny\begin{array}{c c c c c c c}0&1&0&0&0\\&&0\end{array}}+
e_{\tiny\begin{array}{c c c c c c c}0&0&0&1&0\\&&0\end{array}}+
e_{\tiny\begin{array}{c c c c c c c}0&0&0&0&1\\&&0\end{array}}$\\
$A_{2}+A_{1}$ & $e_{\tiny\begin{array}{c c c c c c c}1&0&0&0&0\\&&0\end{array}}+
e_{\tiny\begin{array}{c c c c c c c}0&1&0&0&0\\&&0\end{array}}+
e_{\tiny\begin{array}{c c c c c c c}0&0&0&1&0\\&&0\end{array}}$\\
$A_{2}$ & $e_{\tiny\begin{array}{c c c c c c c}1&0&0&0&0\\&&0\end{array}}+
e_{\tiny\begin{array}{c c c c c c c}0&1&0&0&0\\&&0\end{array}}$\\
$3A_{1}$ & $e_{\tiny\begin{array}{c c c c c c c}1&0&0&0&0\\&&0\end{array}}+
e_{\tiny\begin{array}{c c c c c c c}0&0&1&0&0\\&&0\end{array}}+
e_{\tiny\begin{array}{c c c c c c c}0&0&0&0&1\\&&0\end{array}}$\\
$2A_{1}$ & $e_{\tiny\begin{array}{c c c c c c c}1&0&0&0&0\\&&0\end{array}}+
e_{\tiny\begin{array}{c c c c c c c}0&0&1&0&0\\&&0\end{array}}$\\
$A_{1}$ & $e_{\tiny\begin{array}{c c c c c c c}1&0&0&0&0\\&&0\end{array}}$\\
\end{tabular}
\end{center}%
 \smallskip
 \caption
 {Representatives of nilpotent elements for $E_6$}\label{r2}
 \end{table}
\begin{longtable}{l|l}
Orbit & Representative\\\hline
$E_{7}$ & $e_{\tiny\begin{array}{c c c c c c c}1&0&0&0&0&0\\&&0\end{array}}+
e_{\tiny\begin{array}{c c c c c c c}0&0&0&0&0&0\\&&1\end{array}}+
e_{\tiny\begin{array}{c c c c c c c}0&1&0&0&0&0\\&&0\end{array}}+
e_{\tiny\begin{array}{c c c c c c c}0&0&1&0&0&0\\&&0\end{array}}+
e_{\tiny\begin{array}{c c c c c c c}0&0&0&1&0&0\\&&0\end{array}}+
e_{\tiny\begin{array}{c c c c c c c}0&0&0&0&1&0\\&&0\end{array}}+
e_{\tiny\begin{array}{c c c c c c c}0&0&0&0&0&1\\&&0\end{array}}$\\
$E_{7}(a_{1})$ & $e_{\tiny\begin{array}{c c c c c c c}1&0&0&0&0&0\\&&0\end{array}}+
e_{\tiny\begin{array}{c c c c c c c}0&1&0&0&0&0\\&&0\end{array}}+
e_{\tiny\begin{array}{c c c c c c c}0&0&0&1&0&0\\&&0\end{array}}+
e_{\tiny\begin{array}{c c c c c c c}0&0&0&0&1&0\\&&0\end{array}}+
e_{\tiny\begin{array}{c c c c c c c}0&0&0&0&0&1\\&&0\end{array}}+
e_{\tiny\begin{array}{c c c c c c c}0&0&1&0&0&0\\&&1\end{array}}+
e_{\tiny\begin{array}{c c c c c c c}0&1&1&0&0&0\\&&0\end{array}}$\\
$E_{7}(a_{2})$ & $e_{\tiny\begin{array}{c c c c c c c}1&0&0&0&0&0\\&&0\end{array}}+
e_{\tiny\begin{array}{c c c c c c c}0&0&0&0&0&0\\&&1\end{array}}+
e_{\tiny\begin{array}{c c c c c c c}0&1&0&0&0&0\\&&0\end{array}}+
e_{\tiny\begin{array}{c c c c c c c}0&0&1&0&0&0\\&&1\end{array}}+
e_{\tiny\begin{array}{c c c c c c c}0&0&1&1&0&0\\&&0\end{array}}+
e_{\tiny\begin{array}{c c c c c c c}0&0&0&1&1&0\\&&0\end{array}}+
e_{\tiny\begin{array}{c c c c c c c}0&0&0&0&1&1\\&&0\end{array}}$\\
$E_{7}(a_{3})$ & $e_{\tiny\begin{array}{c c c c c c c}1&0&0&0&0&0\\&&0\end{array}}+
e_{\tiny\begin{array}{c c c c c c c}0&0&0&0&0&1\\&&0\end{array}}+
e_{\tiny\begin{array}{c c c c c c c}0&0&1&0&0&0\\&&1\end{array}}+
e_{\tiny\begin{array}{c c c c c c c}0&1&1&0&0&0\\&&0\end{array}}+
e_{\tiny\begin{array}{c c c c c c c}0&0&0&1&1&0\\&&0\end{array}}+
e_{\tiny\begin{array}{c c c c c c c}0&0&1&1&0&0\\&&1\end{array}}+
e_{\tiny\begin{array}{c c c c c c c}0&1&1&1&0&0\\&&1\end{array}}$\\
$E_{6}$ & $e_{\tiny\begin{array}{c c c c c c c}1&0&0&0&0&0\\&&0\end{array}}+
e_{\tiny\begin{array}{c c c c c c c}0&0&0&0&0&0\\&&1\end{array}}+
e_{\tiny\begin{array}{c c c c c c c}0&1&0&0&0&0\\&&0\end{array}}+
e_{\tiny\begin{array}{c c c c c c c}0&0&1&0&0&0\\&&0\end{array}}+
e_{\tiny\begin{array}{c c c c c c c}0&0&0&1&0&0\\&&0\end{array}}+
e_{\tiny\begin{array}{c c c c c c c}0&0&0&0&1&0\\&&0\end{array}}$\\
$E_{6}(a_{1})$ & $e_{\tiny\begin{array}{c c c c c c c}1&0&0&0&0&0\\&&0\end{array}}+
e_{\tiny\begin{array}{c c c c c c c}0&1&0&0&0&0\\&&0\end{array}}+
e_{\tiny\begin{array}{c c c c c c c}0&0&0&1&0&0\\&&0\end{array}}+
e_{\tiny\begin{array}{c c c c c c c}0&0&0&0&1&0\\&&0\end{array}}+
e_{\tiny\begin{array}{c c c c c c c}0&0&1&0&0&0\\&&1\end{array}}+
e_{\tiny\begin{array}{c c c c c c c}0&1&1&0&0&0\\&&0\end{array}}$\\
$D_{6}$ & $e_{\tiny\begin{array}{c c c c c c c}0&0&0&0&0&0\\&&1\end{array}}+
e_{\tiny\begin{array}{c c c c c c c}0&1&0&0&0&0\\&&0\end{array}}+
e_{\tiny\begin{array}{c c c c c c c}0&0&1&0&0&0\\&&0\end{array}}+
e_{\tiny\begin{array}{c c c c c c c}0&0&0&1&0&0\\&&0\end{array}}+
e_{\tiny\begin{array}{c c c c c c c}0&0&0&0&1&0\\&&0\end{array}}+
e_{\tiny\begin{array}{c c c c c c c}0&0&0&0&0&1\\&&0\end{array}}$\\
$E_{7}(a_{4})$ & $e_{\tiny\begin{array}{c c c c c c c}1&0&0&0&0&0\\&&0\end{array}}+
e_{\tiny\begin{array}{c c c c c c c}0&0&0&0&1&1\\&&0\end{array}}+
e_{\tiny\begin{array}{c c c c c c c}0&1&1&0&0&0\\&&1\end{array}}+
e_{\tiny\begin{array}{c c c c c c c}0&0&1&1&0&0\\&&1\end{array}}+
e_{\tiny\begin{array}{c c c c c c c}0&1&1&1&0&0\\&&0\end{array}}+
e_{\tiny\begin{array}{c c c c c c c}0&0&1&1&1&0\\&&0\end{array}}+
e_{\tiny\begin{array}{c c c c c c c}0&1&1&1&1&0\\&&0\end{array}}$\\
$D_{6}(a_{1})$ & $e_{\tiny\begin{array}{c c c c c c c}0&0&0&0&0&0\\&&1\end{array}}+
e_{\tiny\begin{array}{c c c c c c c}0&0&0&1&0&0\\&&0\end{array}}+
e_{\tiny\begin{array}{c c c c c c c}0&0&0&0&1&0\\&&0\end{array}}+
e_{\tiny\begin{array}{c c c c c c c}0&0&0&0&0&1\\&&0\end{array}}+
e_{\tiny\begin{array}{c c c c c c c}0&0&1&0&0&0\\&&1\end{array}}+
e_{\tiny\begin{array}{c c c c c c c}0&1&1&0&0&0\\&&0\end{array}}$\\
$D_{5}+A_{1}$ & $e_{\tiny\begin{array}{c c c c c c c}1&0&0&0&0&0\\&&0\end{array}}+
e_{\tiny\begin{array}{c c c c c c c}0&0&0&0&0&0\\&&1\end{array}}+
e_{\tiny\begin{array}{c c c c c c c}0&1&0&0&0&0\\&&0\end{array}}+
e_{\tiny\begin{array}{c c c c c c c}0&0&1&0&0&0\\&&0\end{array}}+
e_{\tiny\begin{array}{c c c c c c c}0&0&0&1&0&0\\&&0\end{array}}+
e_{\tiny\begin{array}{c c c c c c c}0&0&0&0&0&1\\&&0\end{array}}$\\
$(A_{6})^{(2)}$ & $e_{\tiny\begin{array}{c c c c c c c}0&0&0&1&1&0\\&&0\end{array}}+
e_{\tiny\begin{array}{c c c c c c c}0&0&0&0&1&1\\&&0\end{array}}+
e_{\tiny\begin{array}{c c c c c c c}1&1&1&0&0&0\\&&0\end{array}}+
e_{\tiny\begin{array}{c c c c c c c}0&1&1&0&0&0\\&&1\end{array}}+
e_{\tiny\begin{array}{c c c c c c c}0&0&1&1&0&0\\&&1\end{array}}+
e_{\tiny\begin{array}{c c c c c c c}0&1&1&1&0&0\\&&0\end{array}}+
e_{\tiny\begin{array}{c c c c c c c}1&2&2&1&0&0\\&&1\end{array}}$\\
$A_{6}$ & $e_{\tiny\begin{array}{c c c c c c c}1&0&0&0&0&0\\&&0\end{array}}+
e_{\tiny\begin{array}{c c c c c c c}0&1&0&0&0&0\\&&0\end{array}}+
e_{\tiny\begin{array}{c c c c c c c}0&0&1&0&0&0\\&&0\end{array}}+
e_{\tiny\begin{array}{c c c c c c c}0&0&0&1&0&0\\&&0\end{array}}+
e_{\tiny\begin{array}{c c c c c c c}0&0&0&0&1&0\\&&0\end{array}}+
e_{\tiny\begin{array}{c c c c c c c}0&0&0&0&0&1\\&&0\end{array}}$\\
$E_{7}(a_{5})$ & $e_{\tiny\begin{array}{c c c c c c c}1&1&1&0&0&0\\&&0\end{array}}+
e_{\tiny\begin{array}{c c c c c c c}0&1&1&0&0&0\\&&1\end{array}}+
e_{\tiny\begin{array}{c c c c c c c}0&0&1&1&0&0\\&&1\end{array}}+
e_{\tiny\begin{array}{c c c c c c c}0&0&1&1&1&0\\&&0\end{array}}+
e_{\tiny\begin{array}{c c c c c c c}0&0&0&1&1&1\\&&0\end{array}}+
e_{\tiny\begin{array}{c c c c c c c}0&1&1&1&1&0\\&&0\end{array}}+
e_{\tiny\begin{array}{c c c c c c c}1&1&1&1&1&0\\&&0\end{array}}$\\
$D_{5}$ & $e_{\tiny\begin{array}{c c c c c c c}1&0&0&0&0&0\\&&0\end{array}}+
e_{\tiny\begin{array}{c c c c c c c}0&0&0&0&0&0\\&&1\end{array}}+
e_{\tiny\begin{array}{c c c c c c c}0&1&0&0&0&0\\&&0\end{array}}+
e_{\tiny\begin{array}{c c c c c c c}0&0&1&0&0&0\\&&0\end{array}}+
e_{\tiny\begin{array}{c c c c c c c}0&0&0&1&0&0\\&&0\end{array}}$\\
$E_{6}(a_{3})$ & $e_{\tiny\begin{array}{c c c c c c c}1&0&0&0&0&0\\&&0\end{array}}+
e_{\tiny\begin{array}{c c c c c c c}0&0&0&0&1&0\\&&0\end{array}}+
e_{\tiny\begin{array}{c c c c c c c}1&1&0&0&0&0\\&&0\end{array}}+
e_{\tiny\begin{array}{c c c c c c c}0&0&1&0&0&0\\&&1\end{array}}+
e_{\tiny\begin{array}{c c c c c c c}0&0&0&1&1&0\\&&0\end{array}}+
e_{\tiny\begin{array}{c c c c c c c}0&1&1&1&0&0\\&&0\end{array}}$\\
$D_{6}(a_{2})$ & $e_{\tiny\begin{array}{c c c c c c c}0&0&0&0&0&0\\&&1\end{array}}+
e_{\tiny\begin{array}{c c c c c c c}0&1&0&0&0&0\\&&0\end{array}}+
e_{\tiny\begin{array}{c c c c c c c}0&1&1&0&0&0\\&&0\end{array}}+
e_{\tiny\begin{array}{c c c c c c c}0&0&1&1&0&0\\&&0\end{array}}+
e_{\tiny\begin{array}{c c c c c c c}0&0&0&1&1&0\\&&0\end{array}}+
e_{\tiny\begin{array}{c c c c c c c}0&0&0&0&1&1\\&&0\end{array}}$\\
$D_{5}(a_{1})+A_{1}$ & $e_{\tiny\begin{array}{c c c c c c c}1&0&0&0&0&0\\&&0\end{array}}+
e_{\tiny\begin{array}{c c c c c c c}0&1&0&0&0&0\\&&0\end{array}}+
e_{\tiny\begin{array}{c c c c c c c}0&0&0&1&0&0\\&&0\end{array}}+
e_{\tiny\begin{array}{c c c c c c c}0&0&0&0&0&1\\&&0\end{array}}+
e_{\tiny\begin{array}{c c c c c c c}0&0&1&0&0&0\\&&1\end{array}}+
e_{\tiny\begin{array}{c c c c c c c}0&0&1&1&0&0\\&&0\end{array}}$\\
$A_{5}+A_{1}$ & $e_{\tiny\begin{array}{c c c c c c c}1&0&0&0&0&0\\&&0\end{array}}+
e_{\tiny\begin{array}{c c c c c c c}0&0&0&0&0&0\\&&1\end{array}}+
e_{\tiny\begin{array}{c c c c c c c}0&0&1&0&0&0\\&&0\end{array}}+
e_{\tiny\begin{array}{c c c c c c c}0&0&0&1&0&0\\&&0\end{array}}+
e_{\tiny\begin{array}{c c c c c c c}0&0&0&0&1&0\\&&0\end{array}}+
e_{\tiny\begin{array}{c c c c c c c}0&0&0&0&0&1\\&&0\end{array}}$\\
$(A_{5})'$ & $e_{\tiny\begin{array}{c c c c c c c}1&0&0&0&0&0\\&&0\end{array}}+
e_{\tiny\begin{array}{c c c c c c c}0&1&0&0&0&0\\&&0\end{array}}+
e_{\tiny\begin{array}{c c c c c c c}0&0&1&0&0&0\\&&0\end{array}}+
e_{\tiny\begin{array}{c c c c c c c}0&0&0&1&0&0\\&&0\end{array}}+
e_{\tiny\begin{array}{c c c c c c c}0&0&0&0&1&0\\&&0\end{array}}$\\
$A_{4}+A_{2}$ & $e_{\tiny\begin{array}{c c c c c c c}1&0&0&0&0&0\\&&0\end{array}}+
e_{\tiny\begin{array}{c c c c c c c}0&0&0&0&0&0\\&&1\end{array}}+
e_{\tiny\begin{array}{c c c c c c c}0&1&0&0&0&0\\&&0\end{array}}+
e_{\tiny\begin{array}{c c c c c c c}0&0&1&0&0&0\\&&0\end{array}}+
e_{\tiny\begin{array}{c c c c c c c}0&0&0&0&1&0\\&&0\end{array}}+
e_{\tiny\begin{array}{c c c c c c c}0&0&0&0&0&1\\&&0\end{array}}$\\
$D_{5}(a_{1})$ & $e_{\tiny\begin{array}{c c c c c c c}1&0&0&0&0&0\\&&0\end{array}}+
e_{\tiny\begin{array}{c c c c c c c}0&1&0&0&0&0\\&&0\end{array}}+
e_{\tiny\begin{array}{c c c c c c c}0&0&0&1&0&0\\&&0\end{array}}+
e_{\tiny\begin{array}{c c c c c c c}0&0&1&0&0&0\\&&1\end{array}}+
e_{\tiny\begin{array}{c c c c c c c}0&0&1&1&0&0\\&&0\end{array}}$\\
$A_{4}+A_{1}$ & $e_{\tiny\begin{array}{c c c c c c c}1&0&0&0&0&0\\&&0\end{array}}+
e_{\tiny\begin{array}{c c c c c c c}0&1&0&0&0&0\\&&0\end{array}}+
e_{\tiny\begin{array}{c c c c c c c}0&0&1&0&0&0\\&&0\end{array}}+
e_{\tiny\begin{array}{c c c c c c c}0&0&0&1&0&0\\&&0\end{array}}+
e_{\tiny\begin{array}{c c c c c c c}0&0&0&0&0&1\\&&0\end{array}}$\\
$D_{4}+A_{1}$ & $e_{\tiny\begin{array}{c c c c c c c}0&0&0&0&0&0\\&&1\end{array}}+
e_{\tiny\begin{array}{c c c c c c c}0&1&0&0&0&0\\&&0\end{array}}+
e_{\tiny\begin{array}{c c c c c c c}0&0&1&0&0&0\\&&0\end{array}}+
e_{\tiny\begin{array}{c c c c c c c}0&0&0&1&0&0\\&&0\end{array}}+
e_{\tiny\begin{array}{c c c c c c c}0&0&0&0&0&1\\&&0\end{array}}$\\
$(A_{5})''$ & $e_{\tiny\begin{array}{c c c c c c c}0&0&0&0&0&0\\&&1\end{array}}+
e_{\tiny\begin{array}{c c c c c c c}0&0&1&0&0&0\\&&0\end{array}}+
e_{\tiny\begin{array}{c c c c c c c}0&0&0&1&0&0\\&&0\end{array}}+
e_{\tiny\begin{array}{c c c c c c c}0&0&0&0&1&0\\&&0\end{array}}+
e_{\tiny\begin{array}{c c c c c c c}0&0&0&0&0&1\\&&0\end{array}}$\\
$A_{3}+A_{2}+A_{1}$ & $e_{\tiny\begin{array}{c c c c c c c}1&0&0&0&0&0\\&&0\end{array}}+
e_{\tiny\begin{array}{c c c c c c c}0&0&0&0&0&0\\&&1\end{array}}+
e_{\tiny\begin{array}{c c c c c c c}0&1&0&0&0&0\\&&0\end{array}}+
e_{\tiny\begin{array}{c c c c c c c}0&0&0&1&0&0\\&&0\end{array}}+
e_{\tiny\begin{array}{c c c c c c c}0&0&0&0&1&0\\&&0\end{array}}+
e_{\tiny\begin{array}{c c c c c c c}0&0&0&0&0&1\\&&0\end{array}}$\\
$A_{4}$ & $e_{\tiny\begin{array}{c c c c c c c}1&0&0&0&0&0\\&&0\end{array}}+
e_{\tiny\begin{array}{c c c c c c c}0&1&0&0&0&0\\&&0\end{array}}+
e_{\tiny\begin{array}{c c c c c c c}0&0&1&0&0&0\\&&0\end{array}}+
e_{\tiny\begin{array}{c c c c c c c}0&0&0&1&0&0\\&&0\end{array}}$\\
$(A_{3}+A_{2})^{(2)}$ & $e_{\tiny\begin{array}{c c c c c c c}0&0&0&0&0&1\\&&0\end{array}}+
e_{\tiny\begin{array}{c c c c c c c}0&0&1&0&0&0\\&&1\end{array}}+
e_{\tiny\begin{array}{c c c c c c c}0&0&0&0&1&1\\&&0\end{array}}+
e_{\tiny\begin{array}{c c c c c c c}0&1&1&1&0&0\\&&0\end{array}}+
e_{\tiny\begin{array}{c c c c c c c}0&0&1&1&1&0\\&&0\end{array}}+
e_{\tiny\begin{array}{c c c c c c c}0&1&1&1&1&0\\&&1\end{array}}$\\
$A_{3}+A_{2}$ & $e_{\tiny\begin{array}{c c c c c c c}1&0&0&0&0&0\\&&0\end{array}}+
e_{\tiny\begin{array}{c c c c c c c}0&1&0&0&0&0\\&&0\end{array}}+
e_{\tiny\begin{array}{c c c c c c c}0&0&1&0&0&0\\&&0\end{array}}+
e_{\tiny\begin{array}{c c c c c c c}0&0&0&0&1&0\\&&0\end{array}}+
e_{\tiny\begin{array}{c c c c c c c}0&0&0&0&0&1\\&&0\end{array}}$\\
$D_{4}(a_{1})+A_{1}$ & $e_{\tiny\begin{array}{c c c c c c c}0&0&0&0&0&0\\&&1\end{array}}+
e_{\tiny\begin{array}{c c c c c c c}0&1&0&0&0&0\\&&0\end{array}}+
e_{\tiny\begin{array}{c c c c c c c}0&0&0&0&0&1\\&&0\end{array}}+
e_{\tiny\begin{array}{c c c c c c c}0&0&1&0&0&0\\&&1\end{array}}+
e_{\tiny\begin{array}{c c c c c c c}0&0&1&1&0&0\\&&0\end{array}}$\\
$D_{4}$ & $e_{\tiny\begin{array}{c c c c c c c}0&0&0&0&0&0\\&&1\end{array}}+
e_{\tiny\begin{array}{c c c c c c c}0&1&0&0&0&0\\&&0\end{array}}+
e_{\tiny\begin{array}{c c c c c c c}0&0&1&0&0&0\\&&0\end{array}}+
e_{\tiny\begin{array}{c c c c c c c}0&0&0&1&0&0\\&&0\end{array}}$\\
$A_{3}+2A_{1}$ & $e_{\tiny\begin{array}{c c c c c c c}1&0&0&0&0&0\\&&0\end{array}}+
e_{\tiny\begin{array}{c c c c c c c}0&0&0&0&0&0\\&&1\end{array}}+
e_{\tiny\begin{array}{c c c c c c c}0&0&1&0&0&0\\&&0\end{array}}+
e_{\tiny\begin{array}{c c c c c c c}0&0&0&1&0&0\\&&0\end{array}}+
e_{\tiny\begin{array}{c c c c c c c}0&0&0&0&0&1\\&&0\end{array}}$\\
$D_{4}(a_{1})$ & $e_{\tiny\begin{array}{c c c c c c c}0&0&0&0&0&0\\&&1\end{array}}+
e_{\tiny\begin{array}{c c c c c c c}0&1&0&0&0&0\\&&0\end{array}}+
e_{\tiny\begin{array}{c c c c c c c}0&0&1&0&0&0\\&&1\end{array}}+
e_{\tiny\begin{array}{c c c c c c c}0&0&1&1&0&0\\&&0\end{array}}$\\
$(A_{3}+A_{1})'$ & $e_{\tiny\begin{array}{c c c c c c c}1&0&0&0&0&0\\&&0\end{array}}+
e_{\tiny\begin{array}{c c c c c c c}0&1&0&0&0&0\\&&0\end{array}}+
e_{\tiny\begin{array}{c c c c c c c}0&0&1&0&0&0\\&&0\end{array}}+
e_{\tiny\begin{array}{c c c c c c c}0&0&0&0&1&0\\&&0\end{array}}$\\
$2A_{2}+A_{1}$ & $e_{\tiny\begin{array}{c c c c c c c}1&0&0&0&0&0\\&&0\end{array}}+
e_{\tiny\begin{array}{c c c c c c c}0&0&0&0&0&0\\&&1\end{array}}+
e_{\tiny\begin{array}{c c c c c c c}0&1&0&0&0&0\\&&0\end{array}}+
e_{\tiny\begin{array}{c c c c c c c}0&0&0&1&0&0\\&&0\end{array}}+
e_{\tiny\begin{array}{c c c c c c c}0&0&0&0&1&0\\&&0\end{array}}$\\
$(A_{3}+A_{1})''$ & $e_{\tiny\begin{array}{c c c c c c c}0&0&0&0&0&0\\&&1\end{array}}+
e_{\tiny\begin{array}{c c c c c c c}0&0&0&1&0&0\\&&0\end{array}}+
e_{\tiny\begin{array}{c c c c c c c}0&0&0&0&1&0\\&&0\end{array}}+
e_{\tiny\begin{array}{c c c c c c c}0&0&0&0&0&1\\&&0\end{array}}$\\
$A_{2}+3A_{1}$ & $e_{\tiny\begin{array}{c c c c c c c}1&0&0&0&0&0\\&&0\end{array}}+
e_{\tiny\begin{array}{c c c c c c c}0&0&0&0&0&0\\&&1\end{array}}+
e_{\tiny\begin{array}{c c c c c c c}0&1&0&0&0&0\\&&0\end{array}}+
e_{\tiny\begin{array}{c c c c c c c}0&0&0&1&0&0\\&&0\end{array}}+
e_{\tiny\begin{array}{c c c c c c c}0&0&0&0&0&1\\&&0\end{array}}$\\
$2A_{2}$ & $e_{\tiny\begin{array}{c c c c c c c}1&0&0&0&0&0\\&&0\end{array}}+
e_{\tiny\begin{array}{c c c c c c c}0&1&0&0&0&0\\&&0\end{array}}+
e_{\tiny\begin{array}{c c c c c c c}0&0&0&0&1&0\\&&0\end{array}}+
e_{\tiny\begin{array}{c c c c c c c}0&0&0&0&0&1\\&&0\end{array}}$\\
$A_{3}$ & $e_{\tiny\begin{array}{c c c c c c c}1&0&0&0&0&0\\&&0\end{array}}+
e_{\tiny\begin{array}{c c c c c c c}0&1&0&0&0&0\\&&0\end{array}}+
e_{\tiny\begin{array}{c c c c c c c}0&0&1&0&0&0\\&&0\end{array}}$\\
$A_{2}+2A_{1}$ & $e_{\tiny\begin{array}{c c c c c c c}1&0&0&0&0&0\\&&0\end{array}}+
e_{\tiny\begin{array}{c c c c c c c}0&1&0&0&0&0\\&&0\end{array}}+
e_{\tiny\begin{array}{c c c c c c c}0&0&0&1&0&0\\&&0\end{array}}+
e_{\tiny\begin{array}{c c c c c c c}0&0&0&0&0&1\\&&0\end{array}}$\\
$A_{2}+A_{1}$ & $e_{\tiny\begin{array}{c c c c c c c}1&0&0&0&0&0\\&&0\end{array}}+
e_{\tiny\begin{array}{c c c c c c c}0&1&0&0&0&0\\&&0\end{array}}+
e_{\tiny\begin{array}{c c c c c c c}0&0&0&1&0&0\\&&0\end{array}}$\\
$4A_{1}$ & $e_{\tiny\begin{array}{c c c c c c c}1&0&0&0&0&0\\&&0\end{array}}+
e_{\tiny\begin{array}{c c c c c c c}0&0&0&0&0&0\\&&1\end{array}}+
e_{\tiny\begin{array}{c c c c c c c}0&0&0&1&0&0\\&&0\end{array}}+
e_{\tiny\begin{array}{c c c c c c c}0&0&0&0&0&1\\&&0\end{array}}$\\
$A_{2}$ & $e_{\tiny\begin{array}{c c c c c c c}1&0&0&0&0&0\\&&0\end{array}}+
e_{\tiny\begin{array}{c c c c c c c}0&1&0&0&0&0\\&&0\end{array}}$\\
$(3A_{1})'$ & $e_{\tiny\begin{array}{c c c c c c c}1&0&0&0&0&0\\&&0\end{array}}+
e_{\tiny\begin{array}{c c c c c c c}0&0&1&0&0&0\\&&0\end{array}}+
e_{\tiny\begin{array}{c c c c c c c}0&0&0&0&1&0\\&&0\end{array}}$\\
$(3A_{1})''$ & $e_{\tiny\begin{array}{c c c c c c c}0&0&0&0&0&0\\&&1\end{array}}+
e_{\tiny\begin{array}{c c c c c c c}0&0&0&1&0&0\\&&0\end{array}}+
e_{\tiny\begin{array}{c c c c c c c}0&0&0&0&0&1\\&&0\end{array}}$\\
$2A_{1}$ & $e_{\tiny\begin{array}{c c c c c c c}1&0&0&0&0&0\\&&0\end{array}}+
e_{\tiny\begin{array}{c c c c c c c}0&0&1&0&0&0\\&&0\end{array}}$\\
$A_{1}$ & $e_{\tiny\begin{array}{c c c c c c c}1&0&0&0&0&0\\&&0\end{array}}$\\\\
\caption{Representatives of nilpotent elements for $E_7$}\label{r3}
\end{longtable}

\newpage

\begin{longtable}{l|l}
Orbit & Representative\\\hline
$E_{8}$ & $e_{\tiny\begin{array}{c c c c c c c c}1&0&0&0&0&0&0\\&&0\end{array}}+
e_{\tiny\begin{array}{c c c c c c c c}0&0&0&0&0&0&0\\&&1\end{array}}+
e_{\tiny\begin{array}{c c c c c c c c}0&1&0&0&0&0&0\\&&0\end{array}}+
e_{\tiny\begin{array}{c c c c c c c c}0&0&1&0&0&0&0\\&&0\end{array}}+
e_{\tiny\begin{array}{c c c c c c c c}0&0&0&1&0&0&0\\&&0\end{array}}+
e_{\tiny\begin{array}{c c c c c c c c}0&0&0&0&1&0&0\\&&0\end{array}}+
e_{\tiny\begin{array}{c c c c c c c c}0&0&0&0&0&1&0\\&&0\end{array}}+
e_{\tiny\begin{array}{c c c c c c c c}0&0&0&0&0&0&1\\&&0\end{array}}$\\
$E_{8}(a_{1})$ & $e_{\tiny\begin{array}{c c c c c c c c}1&0&0&0&0&0&0\\&&0\end{array}}+
e_{\tiny\begin{array}{c c c c c c c c}0&0&0&0&0&0&0\\&&1\end{array}}+
e_{\tiny\begin{array}{c c c c c c c c}0&0&0&1&0&0&0\\&&0\end{array}}+
e_{\tiny\begin{array}{c c c c c c c c}0&0&0&0&1&0&0\\&&0\end{array}}+
e_{\tiny\begin{array}{c c c c c c c c}0&0&0&0&0&1&0\\&&0\end{array}}+
e_{\tiny\begin{array}{c c c c c c c c}0&0&0&0&0&0&1\\&&0\end{array}}+
e_{\tiny\begin{array}{c c c c c c c c}0&0&1&0&0&0&0\\&&1\end{array}}+
e_{\tiny\begin{array}{c c c c c c c c}0&1&1&0&0&0&0\\&&0\end{array}}$\\
$E_{8}(a_{2})$ & $e_{\tiny\begin{array}{c c c c c c c c}1&0&0&0&0&0&0\\&&0\end{array}}+
e_{\tiny\begin{array}{c c c c c c c c}0&0&0&0&0&0&0\\&&1\end{array}}+
e_{\tiny\begin{array}{c c c c c c c c}0&1&0&0&0&0&0\\&&0\end{array}}+
e_{\tiny\begin{array}{c c c c c c c c}0&0&0&0&0&0&1\\&&0\end{array}}+
e_{\tiny\begin{array}{c c c c c c c c}0&0&1&0&0&0&0\\&&1\end{array}}+
e_{\tiny\begin{array}{c c c c c c c c}0&0&1&1&0&0&0\\&&0\end{array}}+
e_{\tiny\begin{array}{c c c c c c c c}0&0&0&1&1&0&0\\&&0\end{array}}+
e_{\tiny\begin{array}{c c c c c c c c}0&0&0&0&1&1&0\\&&0\end{array}}$\\
$E_{8}(a_{3})$ & $e_{\tiny\begin{array}{c c c c c c c c}0&0&0&0&0&1&0\\&&0\end{array}}+
e_{\tiny\begin{array}{c c c c c c c c}0&0&0&0&0&0&1\\&&0\end{array}}+
e_{\tiny\begin{array}{c c c c c c c c}1&1&0&0&0&0&0\\&&0\end{array}}+
e_{\tiny\begin{array}{c c c c c c c c}0&0&1&0&0&0&0\\&&1\end{array}}+
e_{\tiny\begin{array}{c c c c c c c c}0&1&1&0&0&0&0\\&&0\end{array}}+
e_{\tiny\begin{array}{c c c c c c c c}0&0&1&1&0&0&0\\&&0\end{array}}+
e_{\tiny\begin{array}{c c c c c c c c}0&0&0&1&1&0&0\\&&0\end{array}}+
e_{\tiny\begin{array}{c c c c c c c c}0&1&1&1&0&0&0\\&&0\end{array}}$\\
$E_{8}(a_{4})$ & $e_{\tiny\begin{array}{c c c c c c c c}1&1&0&0&0&0&0\\&&0\end{array}}+
e_{\tiny\begin{array}{c c c c c c c c}0&0&1&0&0&0&0\\&&1\end{array}}+
e_{\tiny\begin{array}{c c c c c c c c}0&1&1&0&0&0&0\\&&0\end{array}}+
e_{\tiny\begin{array}{c c c c c c c c}0&0&1&1&0&0&0\\&&0\end{array}}+
e_{\tiny\begin{array}{c c c c c c c c}0&0&0&1&1&0&0\\&&0\end{array}}+
e_{\tiny\begin{array}{c c c c c c c c}0&0&0&0&1&1&0\\&&0\end{array}}+
e_{\tiny\begin{array}{c c c c c c c c}0&0&0&0&0&1&1\\&&0\end{array}}+
e_{\tiny\begin{array}{c c c c c c c c}0&1&1&1&0&0&0\\&&0\end{array}}$\\
$E_{7}$ & $e_{\tiny\begin{array}{c c c c c c c c}1&0&0&0&0&0&0\\&&0\end{array}}+
e_{\tiny\begin{array}{c c c c c c c c}0&0&0&0&0&0&0\\&&1\end{array}}+
e_{\tiny\begin{array}{c c c c c c c c}0&1&0&0&0&0&0\\&&0\end{array}}+
e_{\tiny\begin{array}{c c c c c c c c}0&0&1&0&0&0&0\\&&0\end{array}}+
e_{\tiny\begin{array}{c c c c c c c c}0&0&0&1&0&0&0\\&&0\end{array}}+
e_{\tiny\begin{array}{c c c c c c c c}0&0&0&0&1&0&0\\&&0\end{array}}+
e_{\tiny\begin{array}{c c c c c c c c}0&0&0&0&0&1&0\\&&0\end{array}}$\\
$E_{8}(b_{4})$ & $e_{\tiny\begin{array}{c c c c c c c c}0&0&0&0&0&1&0\\&&0\end{array}}+
e_{\tiny\begin{array}{c c c c c c c c}0&0&0&0&0&0&1\\&&0\end{array}}+
e_{\tiny\begin{array}{c c c c c c c c}1&1&0&0&0&0&0\\&&0\end{array}}+
e_{\tiny\begin{array}{c c c c c c c c}0&1&1&0&0&0&0\\&&1\end{array}}+
e_{\tiny\begin{array}{c c c c c c c c}0&0&1&1&0&0&0\\&&1\end{array}}+
e_{\tiny\begin{array}{c c c c c c c c}0&1&1&1&0&0&0\\&&0\end{array}}+
e_{\tiny\begin{array}{c c c c c c c c}0&0&1&1&1&0&0\\&&0\end{array}}+
e_{\tiny\begin{array}{c c c c c c c c}0&0&0&1&1&1&0\\&&0\end{array}}$\\
$E_{8}(a_{5})$ & $e_{\tiny\begin{array}{c c c c c c c c}1&1&0&0&0&0&0\\&&0\end{array}}+
e_{\tiny\begin{array}{c c c c c c c c}0&0&0&0&1&1&0\\&&0\end{array}}+
e_{\tiny\begin{array}{c c c c c c c c}0&0&0&0&0&1&1\\&&0\end{array}}+
e_{\tiny\begin{array}{c c c c c c c c}0&1&1&0&0&0&0\\&&1\end{array}}+
e_{\tiny\begin{array}{c c c c c c c c}0&0&1&1&0&0&0\\&&1\end{array}}+
e_{\tiny\begin{array}{c c c c c c c c}0&1&1&1&0&0&0\\&&0\end{array}}+
e_{\tiny\begin{array}{c c c c c c c c}0&0&1&1&1&0&0\\&&0\end{array}}+
e_{\tiny\begin{array}{c c c c c c c c}0&0&1&1&1&0&0\\&&1\end{array}}$\\
$E_{7}(a_{1})$ & $e_{\tiny\begin{array}{c c c c c c c c}1&0&0&0&0&0&0\\&&0\end{array}}+
e_{\tiny\begin{array}{c c c c c c c c}0&1&0&0&0&0&0\\&&0\end{array}}+
e_{\tiny\begin{array}{c c c c c c c c}0&0&0&1&0&0&0\\&&0\end{array}}+
e_{\tiny\begin{array}{c c c c c c c c}0&0&0&0&1&0&0\\&&0\end{array}}+
e_{\tiny\begin{array}{c c c c c c c c}0&0&0&0&0&1&0\\&&0\end{array}}+
e_{\tiny\begin{array}{c c c c c c c c}0&0&1&0&0&0&0\\&&1\end{array}}+
e_{\tiny\begin{array}{c c c c c c c c}0&1&1&0&0&0&0\\&&0\end{array}}$\\
$E_{8}(b_{5})$ & $e_{\tiny\begin{array}{c c c c c c c c}0&0&0&0&0&0&1\\&&0\end{array}}+
e_{\tiny\begin{array}{c c c c c c c c}1&1&1&0&0&0&0\\&&0\end{array}}+
e_{\tiny\begin{array}{c c c c c c c c}0&1&1&0&0&0&0\\&&1\end{array}}+
e_{\tiny\begin{array}{c c c c c c c c}0&0&1&1&0&0&0\\&&1\end{array}}+
e_{\tiny\begin{array}{c c c c c c c c}0&0&1&1&1&0&0\\&&0\end{array}}+
e_{\tiny\begin{array}{c c c c c c c c}0&0&0&1&1&1&0\\&&0\end{array}}+
e_{\tiny\begin{array}{c c c c c c c c}1&1&1&1&0&0&0\\&&0\end{array}}+
e_{\tiny\begin{array}{c c c c c c c c}0&1&1&1&1&0&0\\&&1\end{array}}$\\
$(D_{7})^{(2)}$ & $e_{\tiny\begin{array}{c c c c c c c c}1&0&0&0&0&0&0\\&&0\end{array}}+
e_{\tiny\begin{array}{c c c c c c c c}0&1&1&0&0&0&0\\&&1\end{array}}+
e_{\tiny\begin{array}{c c c c c c c c}0&0&1&1&0&0&0\\&&1\end{array}}+
e_{\tiny\begin{array}{c c c c c c c c}0&1&1&1&0&0&0\\&&0\end{array}}+
e_{\tiny\begin{array}{c c c c c c c c}0&0&1&1&1&0&0\\&&0\end{array}}+
e_{\tiny\begin{array}{c c c c c c c c}0&0&0&1&1&1&0\\&&0\end{array}}+
e_{\tiny\begin{array}{c c c c c c c c}0&0&0&0&1&1&1\\&&0\end{array}}+
e_{\tiny\begin{array}{c c c c c c c c}1&1&1&1&1&1&1\\&&1\end{array}}$\\
$D_{7}$ & $e_{\tiny\begin{array}{c c c c c c c c}0&0&0&0&0&0&0\\&&1\end{array}}+
e_{\tiny\begin{array}{c c c c c c c c}0&1&0&0&0&0&0\\&&0\end{array}}+
e_{\tiny\begin{array}{c c c c c c c c}0&0&1&0&0&0&0\\&&0\end{array}}+
e_{\tiny\begin{array}{c c c c c c c c}0&0&0&1&0&0&0\\&&0\end{array}}+
e_{\tiny\begin{array}{c c c c c c c c}0&0&0&0&1&0&0\\&&0\end{array}}+
e_{\tiny\begin{array}{c c c c c c c c}0&0&0&0&0&1&0\\&&0\end{array}}+
e_{\tiny\begin{array}{c c c c c c c c}0&0&0&0&0&0&1\\&&0\end{array}}$\\
$E_{8}(a_{6})$ & $e_{\tiny\begin{array}{c c c c c c c c}0&0&0&0&1&1&0\\&&0\end{array}}+
e_{\tiny\begin{array}{c c c c c c c c}0&0&0&0&0&1&1\\&&0\end{array}}+
e_{\tiny\begin{array}{c c c c c c c c}1&1&1&0&0&0&0\\&&0\end{array}}+
e_{\tiny\begin{array}{c c c c c c c c}0&0&1&1&1&0&0\\&&0\end{array}}+
e_{\tiny\begin{array}{c c c c c c c c}1&1&1&0&0&0&0\\&&1\end{array}}+
e_{\tiny\begin{array}{c c c c c c c c}1&1&1&1&0&0&0\\&&0\end{array}}+
e_{\tiny\begin{array}{c c c c c c c c}0&1&1&1&0&0&0\\&&1\end{array}}+
e_{\tiny\begin{array}{c c c c c c c c}1&1&1&1&1&0&0\\&&1\end{array}}$\\
$E_{7}(a_{2})$ & $e_{\tiny\begin{array}{c c c c c c c c}1&0&0&0&0&0&0\\&&0\end{array}}+
e_{\tiny\begin{array}{c c c c c c c c}0&0&0&0&0&0&0\\&&1\end{array}}+
e_{\tiny\begin{array}{c c c c c c c c}0&1&0&0&0&0&0\\&&0\end{array}}+
e_{\tiny\begin{array}{c c c c c c c c}0&0&1&0&0&0&0\\&&1\end{array}}+
e_{\tiny\begin{array}{c c c c c c c c}0&0&1&1&0&0&0\\&&0\end{array}}+
e_{\tiny\begin{array}{c c c c c c c c}0&0&0&1&1&0&0\\&&0\end{array}}+
e_{\tiny\begin{array}{c c c c c c c c}0&0&0&0&1&1&0\\&&0\end{array}}$\\
$E_{6}+A_{1}$ & $e_{\tiny\begin{array}{c c c c c c c c}1&0&0&0&0&0&0\\&&0\end{array}}+
e_{\tiny\begin{array}{c c c c c c c c}0&0&0&0&0&0&0\\&&1\end{array}}+
e_{\tiny\begin{array}{c c c c c c c c}0&1&0&0&0&0&0\\&&0\end{array}}+
e_{\tiny\begin{array}{c c c c c c c c}0&0&1&0&0&0&0\\&&0\end{array}}+
e_{\tiny\begin{array}{c c c c c c c c}0&0&0&1&0&0&0\\&&0\end{array}}+
e_{\tiny\begin{array}{c c c c c c c c}0&0&0&0&1&0&0\\&&0\end{array}}+
e_{\tiny\begin{array}{c c c c c c c c}0&0&0&0&0&0&1\\&&0\end{array}}$\\
$(D_{7}(a_{1}))^{(2)}$ & $e_{\tiny\begin{array}{c c c c c c c c}0&0&0&1&0&0&0\\&&0\end{array}}+
e_{\tiny\begin{array}{c c c c c c c c}0&0&0&0&0&0&1\\&&0\end{array}}+
e_{\tiny\begin{array}{c c c c c c c c}1&1&0&0&0&0&0\\&&0\end{array}}+
e_{\tiny\begin{array}{c c c c c c c c}0&0&1&1&0&0&0\\&&0\end{array}}+
e_{\tiny\begin{array}{c c c c c c c c}0&0&0&0&0&1&1\\&&0\end{array}}+
e_{\tiny\begin{array}{c c c c c c c c}0&0&1&1&1&0&0\\&&1\end{array}}+
e_{\tiny\begin{array}{c c c c c c c c}0&1&1&1&1&0&0\\&&0\end{array}}+
e_{\tiny\begin{array}{c c c c c c c c}0&1&2&1&1&1&0\\&&1\end{array}}$\\
$D_{7}(a_{1})$ & $e_{\tiny\begin{array}{c c c c c c c c}0&0&0&0&0&0&0\\&&1\end{array}}+
e_{\tiny\begin{array}{c c c c c c c c}0&0&0&1&0&0&0\\&&0\end{array}}+
e_{\tiny\begin{array}{c c c c c c c c}0&0&0&0&1&0&0\\&&0\end{array}}+
e_{\tiny\begin{array}{c c c c c c c c}0&0&0&0&0&1&0\\&&0\end{array}}+
e_{\tiny\begin{array}{c c c c c c c c}0&0&0&0&0&0&1\\&&0\end{array}}+
e_{\tiny\begin{array}{c c c c c c c c}0&0&1&0&0&0&0\\&&1\end{array}}+
e_{\tiny\begin{array}{c c c c c c c c}0&1&1&0&0&0&0\\&&0\end{array}}$\\
$E_{8}(b_{6}), p>2$ & $e_{\tiny\begin{array}{c c c c c c c c}0&0&0&0&0&1&1\\&&0\end{array}}+
e_{\tiny\begin{array}{c c c c c c c c}1&1&1&0&0&0&0\\&&1\end{array}}+
e_{\tiny\begin{array}{c c c c c c c c}1&1&1&1&0&0&0\\&&0\end{array}}+
e_{\tiny\begin{array}{c c c c c c c c}0&1&1&1&0&0&0\\&&1\end{array}}+
e_{\tiny\begin{array}{c c c c c c c c}0&0&1&1&1&0&0\\&&1\end{array}}+
e_{\tiny\begin{array}{c c c c c c c c}0&1&1&1&1&0&0\\&&0\end{array}}+
e_{\tiny\begin{array}{c c c c c c c c}0&0&1&1&1&1&0 \\&&0\end{array}}+
e_{\tiny\begin{array}{c c c c c c c c}0&0&0&1&1&1&1\\&&0\end{array}}$\\
$E_{8}(b_{6}), p=2$ & $e_{\tiny\begin{array}{c c c c c c c c}0&0&0&0&0&1&1\\&&0\end{array}}+
e_{\tiny\begin{array}{c c c c c c c c}0&0&0&0&1&1&1\\&&0\end{array}}+
e_{\tiny\begin{array}{c c c c c c c c}1&1&1&0&0&0&0\\&&1\end{array}}+
e_{\tiny\begin{array}{c c c c c c c c}1&1&1&1&0&0&0\\&&0\end{array}}+
e_{\tiny\begin{array}{c c c c c c c c}0&1&1&1&0&0&0\\&&1\end{array}}+
e_{\tiny\begin{array}{c c c c c c c c}0&0&1&1&1&0&0\\&&1\end{array}}+
e_{\tiny\begin{array}{c c c c c c c c}0&1&1&1&1&0&0\\&&0\end{array}}+
e_{\tiny\begin{array}{c c c c c c c c}0&0&1&1&1&1&0\\&&0\end{array}}$\\
$E_{7}(a_{3})$ & $e_{\tiny\begin{array}{c c c c c c c c}1&0&0&0&0&0&0\\&&0\end{array}}+
e_{\tiny\begin{array}{c c c c c c c c}0&0&0&0&0&1&0\\&&0\end{array}}+
e_{\tiny\begin{array}{c c c c c c c c}0&0&1&0&0&0&0\\&&1\end{array}}+
e_{\tiny\begin{array}{c c c c c c c c}0&1&1&0&0&0&0\\&&0\end{array}}+
e_{\tiny\begin{array}{c c c c c c c c}0&0&0&1&1&0&0\\&&0\end{array}}+
e_{\tiny\begin{array}{c c c c c c c c}0&0&1&1&0&0&0\\&&1\end{array}}+
e_{\tiny\begin{array}{c c c c c c c c}0&1&1&1&0&0&0\\&&1\end{array}}$\\
$E_{6}(a_{1})+A_{1}$ & $e_{\tiny\begin{array}{c c c c c c c c}1&0&0&0&0&0&0\\&&0\end{array}}+
e_{\tiny\begin{array}{c c c c c c c c}0&1&0&0&0&0&0\\&&0\end{array}}+
e_{\tiny\begin{array}{c c c c c c c c}0&0&0&1&0&0&0\\&&0\end{array}}+
e_{\tiny\begin{array}{c c c c c c c c}0&0&0&0&1&0&0\\&&0\end{array}}+
e_{\tiny\begin{array}{c c c c c c c c}0&0&0&0&0&0&1\\&&0\end{array}}+
e_{\tiny\begin{array}{c c c c c c c c}0&0&1&0&0&0&0\\&&1\end{array}}+
e_{\tiny\begin{array}{c c c c c c c c}0&1&1&0&0&0&0\\&&0\end{array}}$\\
$A_7^{(3)}$&$e_{\tiny\begin{array}{c c c c c c c c}0&0&0&1&1&1&0\\&&0\end{array}}+
e_{\tiny\begin{array}{c c c c c c c c}0&0&0&0&1&1&1\\&&0\end{array}}+
e_{\tiny\begin{array}{c c c c c c c c}1&1&1&0&0&0&0\\&&1\end{array}}+
e_{\tiny\begin{array}{c c c c c c c c}1&1&1&1&0&0&0\\&&0\end{array}}+
e_{\tiny\begin{array}{c c c c c c c c}0&0&1&1&1&0&0\\&&1\end{array}}+
e_{\tiny\begin{array}{c c c c c c c c}0&1&1&1&1&0&0\\&&0\end{array}}+
e_{\tiny\begin{array}{c c c c c c c c}0&1&2&1&0&0&0\\&&1\end{array}}+
e_{\tiny\begin{array}{c c c c c c c c}0&0&1&1&1&1&1\\&&0\end{array}}$\\
$A_{7}$ & $e_{\tiny\begin{array}{c c c c c c c c}1&0&0&0&0&0&0\\&&0\end{array}}+
e_{\tiny\begin{array}{c c c c c c c c}0&1&0&0&0&0&0\\&&0\end{array}}+
e_{\tiny\begin{array}{c c c c c c c c}0&0&1&0&0&0&0\\&&0\end{array}}+
e_{\tiny\begin{array}{c c c c c c c c}0&0&0&1&0&0&0\\&&0\end{array}}+
e_{\tiny\begin{array}{c c c c c c c c}0&0&0&0&1&0&0\\&&0\end{array}}+
e_{\tiny\begin{array}{c c c c c c c c}0&0&0&0&0&1&0\\&&0\end{array}}+
e_{\tiny\begin{array}{c c c c c c c c}0&0&0&0&0&0&1\\&&0\end{array}}$\\
$D_{7}(a_{2})$ & $e_{\tiny\begin{array}{c c c c c c c c}0&0&0&0&0&0&0\\&&1\end{array}}+
e_{\tiny\begin{array}{c c c c c c c c}0&1&0&0&0&0&0\\&&0\end{array}}+
e_{\tiny\begin{array}{c c c c c c c c}0&0&0&0&0&0&1\\&&0\end{array}}+
e_{\tiny\begin{array}{c c c c c c c c}0&1&1&0&0&0&0\\&&0\end{array}}+
e_{\tiny\begin{array}{c c c c c c c c}0&0&1&1&0&0&0\\&&0\end{array}}+
e_{\tiny\begin{array}{c c c c c c c c}0&0&0&1&1&0&0\\&&0\end{array}}+
e_{\tiny\begin{array}{c c c c c c c c}0&0&0&0&1&1&0\\&&0\end{array}}$\\
$E_{6}$ & $e_{\tiny\begin{array}{c c c c c c c c}1&0&0&0&0&0&0\\&&0\end{array}}+
e_{\tiny\begin{array}{c c c c c c c c}0&0&0&0&0&0&0\\&&1\end{array}}+
e_{\tiny\begin{array}{c c c c c c c c}0&1&0&0&0&0&0\\&&0\end{array}}+
e_{\tiny\begin{array}{c c c c c c c c}0&0&1&0&0&0&0\\&&0\end{array}}+
e_{\tiny\begin{array}{c c c c c c c c}0&0&0&1&0&0&0\\&&0\end{array}}+
e_{\tiny\begin{array}{c c c c c c c c}0&0&0&0&1&0&0\\&&0\end{array}}$\\
$D_{6}$ & $e_{\tiny\begin{array}{c c c c c c c c}0&0&0&0&0&0&0\\&&1\end{array}}+
e_{\tiny\begin{array}{c c c c c c c c}0&1&0&0&0&0&0\\&&0\end{array}}+
e_{\tiny\begin{array}{c c c c c c c c}0&0&1&0&0&0&0\\&&0\end{array}}+
e_{\tiny\begin{array}{c c c c c c c c}0&0&0&1&0&0&0\\&&0\end{array}}+
e_{\tiny\begin{array}{c c c c c c c c}0&0&0&0&1&0&0\\&&0\end{array}}+
e_{\tiny\begin{array}{c c c c c c c c}0&0&0&0&0&1&0\\&&0\end{array}}$\\
$(D_{5}+A_{2})^{(2)}$ & $e_{\tiny\begin{array}{c c c c c c c c}0&0&0&0&0&1&1\\&&0\end{array}}+
e_{\tiny\begin{array}{c c c c c c c c}0&0&0&0&1&1&1\\&&0\end{array}}+
e_{\tiny\begin{array}{c c c c c c c c}1&1&1&1&0&0&0\\&&1\end{array}}+
e_{\tiny\begin{array}{c c c c c c c c}1&1&1&1&1&0&0\\&&0\end{array}}+
e_{\tiny\begin{array}{c c c c c c c c}0&1&2&1&0&0&0\\&&1\end{array}}+
e_{\tiny\begin{array}{c c c c c c c c}0&1&1&1&1&0&0\\&&1\end{array}}+
e_{\tiny\begin{array}{c c c c c c c c}0&0&1&1&1&1&0\\&&1\end{array}}+
e_{\tiny\begin{array}{c c c c c c c c}0&1&1&1&1&1&0\\&&0\end{array}}$\\
$D_{5}+A_{2}$ & $e_{\tiny\begin{array}{c c c c c c c c}1&0&0&0&0&0&0\\&&0\end{array}}+
e_{\tiny\begin{array}{c c c c c c c c}0&0&0&0&0&0&0\\&&1\end{array}}+
e_{\tiny\begin{array}{c c c c c c c c}0&1&0&0&0&0&0\\&&0\end{array}}+
e_{\tiny\begin{array}{c c c c c c c c}0&0&1&0&0&0&0\\&&0\end{array}}+
e_{\tiny\begin{array}{c c c c c c c c}0&0&0&1&0&0&0\\&&0\end{array}}+
e_{\tiny\begin{array}{c c c c c c c c}0&0&0&0&0&1&0\\&&0\end{array}}+
e_{\tiny\begin{array}{c c c c c c c c}0&0&0&0&0&0&1\\&&0\end{array}}$\\
$E_{6}(a_{1})$ & $e_{\tiny\begin{array}{c c c c c c c c}1&0&0&0&0&0&0\\&&0\end{array}}+
e_{\tiny\begin{array}{c c c c c c c c}0&1&0&0&0&0&0\\&&0\end{array}}+
e_{\tiny\begin{array}{c c c c c c c c}0&0&0&1&0&0&0\\&&0\end{array}}+
e_{\tiny\begin{array}{c c c c c c c c}0&0&0&0&1&0&0\\&&0\end{array}}+
e_{\tiny\begin{array}{c c c c c c c c}0&0&1&0&0&0&0\\&&1\end{array}}+
e_{\tiny\begin{array}{c c c c c c c c}0&1&1&0&0&0&0\\&&0\end{array}}$\\
$E_{7}(a_{4})$ & $e_{\tiny\begin{array}{c c c c c c c c}1&0&0&0&0&0&0\\&&0\end{array}}+
e_{\tiny\begin{array}{c c c c c c c c}0&0&0&0&1&1&0\\&&0\end{array}}+
e_{\tiny\begin{array}{c c c c c c c c}0&1&1&0&0&0&0\\&&1\end{array}}+
e_{\tiny\begin{array}{c c c c c c c c}0&0&1&1&0&0&0\\&&1\end{array}}+
e_{\tiny\begin{array}{c c c c c c c c}0&1&1&1&0&0&0\\&&0\end{array}}+
e_{\tiny\begin{array}{c c c c c c c c}0&0&1&1&1&0&0\\&&0\end{array}}+
e_{\tiny\begin{array}{c c c c c c c c}0&1&1&1&1&0&0\\&&0\end{array}}$\\
$A_{6}+A_{1}$ & $e_{\tiny\begin{array}{c c c c c c c c}1&0&0&0&0&0&0\\&&0\end{array}}+
e_{\tiny\begin{array}{c c c c c c c c}0&0&0&0&0&0&0\\&&1\end{array}}+
e_{\tiny\begin{array}{c c c c c c c c}0&0&1&0&0&0&0\\&&0\end{array}}+
e_{\tiny\begin{array}{c c c c c c c c}0&0&0&1&0&0&0\\&&0\end{array}}+
e_{\tiny\begin{array}{c c c c c c c c}0&0&0&0&1&0&0\\&&0\end{array}}+
e_{\tiny\begin{array}{c c c c c c c c}0&0&0&0&0&1&0\\&&0\end{array}}+
e_{\tiny\begin{array}{c c c c c c c c}0&0&0&0&0&0&1\\&&0\end{array}}$\\
$D_{6}(a_{1})$ & $e_{\tiny\begin{array}{c c c c c c c c}0&0&0&0&0&0&0\\&&1\end{array}}+
e_{\tiny\begin{array}{c c c c c c c c}0&0&0&1&0&0&0\\&&0\end{array}}+
e_{\tiny\begin{array}{c c c c c c c c}0&0&0&0&1&0&0\\&&0\end{array}}+
e_{\tiny\begin{array}{c c c c c c c c}0&0&0&0&0&1&0\\&&0\end{array}}+
e_{\tiny\begin{array}{c c c c c c c c}0&0&1&0&0&0&0\\&&1\end{array}}+
e_{\tiny\begin{array}{c c c c c c c c}0&1&1&0&0&0&0\\&&0\end{array}}$\\
$(A_{6})^{(2)}$ & $e_{\tiny\begin{array}{c c c c c c c c}0&0&0&1&1&0&0\\&&0\end{array}}+
e_{\tiny\begin{array}{c c c c c c c c}0&0&0&0&1&1&0\\&&0\end{array}}+
e_{\tiny\begin{array}{c c c c c c c c}1&1&1&0&0&0&0\\&&0\end{array}}+
e_{\tiny\begin{array}{c c c c c c c c}0&1&1&0&0&0&0\\&&1\end{array}}+
e_{\tiny\begin{array}{c c c c c c c c}0&0&1&1&0&0&0\\&&1\end{array}}+
e_{\tiny\begin{array}{c c c c c c c c}0&1&1&1&0&0&0\\&&0\end{array}}+
e_{\tiny\begin{array}{c c c c c c c c}1&2&2&1&0&0&0\\&&1\end{array}}$\\
$A_{6}$ & $e_{\tiny\begin{array}{c c c c c c c c}1&0&0&0&0&0&0\\&&0\end{array}}+
e_{\tiny\begin{array}{c c c c c c c c}0&1&0&0&0&0&0\\&&0\end{array}}+
e_{\tiny\begin{array}{c c c c c c c c}0&0&1&0&0&0&0\\&&0\end{array}}+
e_{\tiny\begin{array}{c c c c c c c c}0&0&0&1&0&0&0\\&&0\end{array}}+
e_{\tiny\begin{array}{c c c c c c c c}0&0&0&0&1&0&0\\&&0\end{array}}+
e_{\tiny\begin{array}{c c c c c c c c}0&0&0&0&0&1&0\\&&0\end{array}}$\\
$E_{8}(a_{7})$ & $e_{\tiny\begin{array}{c c c c c c c c}1&1&1&1&0&0&0\\&&1\end{array}}+
e_{\tiny\begin{array}{c c c c c c c c}0&1&1&1&1&0&0\\&&1\end{array}}+
e_{\tiny\begin{array}{c c c c c c c c}0&1&1&1&1&1&0\\&&0\end{array}}+
e_{\tiny\begin{array}{c c c c c c c c}0&0&1&1&1&1&1\\&&0\end{array}}+
e_{\tiny\begin{array}{c c c c c c c c}0&1&2&1&1&0&0\\&&1\end{array}}+
e_{\tiny\begin{array}{c c c c c c c c}1&1&2&1&1&0&0\\&&1\end{array}}+
e_{\tiny\begin{array}{c c c c c c c c}1&1&2&1&1&1&0\\&&1\end{array}}+
e_{\tiny\begin{array}{c c c c c c c c}1&1&2&1&1&1&1\\&&1\end{array}}$\\
$D_{5}+A_{1}$ & $e_{\tiny\begin{array}{c c c c c c c c}1&0&0&0&0&0&0\\&&0\end{array}}+
e_{\tiny\begin{array}{c c c c c c c c}0&0&0&0&0&0&0\\&&1\end{array}}+
e_{\tiny\begin{array}{c c c c c c c c}0&1&0&0&0&0&0\\&&0\end{array}}+
e_{\tiny\begin{array}{c c c c c c c c}0&0&1&0&0&0&0\\&&0\end{array}}+
e_{\tiny\begin{array}{c c c c c c c c}0&0&0&1&0&0&0\\&&0\end{array}}+
e_{\tiny\begin{array}{c c c c c c c c}0&0&0&0&0&1&0\\&&0\end{array}}$\\
$E_{7}(a_{5})$ & $e_{\tiny\begin{array}{c c c c c c c c}1&1&1&0&0&0&0\\&&0\end{array}}+
e_{\tiny\begin{array}{c c c c c c c c}0&1&1&0&0&0&0\\&&1\end{array}}+
e_{\tiny\begin{array}{c c c c c c c c}0&0&1&1&0&0&0\\&&1\end{array}}+
e_{\tiny\begin{array}{c c c c c c c c}0&0&1&1&1&0&0\\&&0\end{array}}+
e_{\tiny\begin{array}{c c c c c c c c}0&0&0&1&1&1&0\\&&0\end{array}}+
e_{\tiny\begin{array}{c c c c c c c c}0&1&1&1&1&0&0\\&&0\end{array}}+
e_{\tiny\begin{array}{c c c c c c c c}1&1&1&1&1&0&0\\&&0\end{array}}$\\
$E_{6}(a_{3})+A_{1}$ & $e_{\tiny\begin{array}{c c c c c c c c}1&0&0&0&0&0&0\\&&0\end{array}}+
e_{\tiny\begin{array}{c c c c c c c c}0&0&0&0&1&0&0\\&&0\end{array}}+
e_{\tiny\begin{array}{c c c c c c c c}0&0&0&0&0&0&1\\&&0\end{array}}+
e_{\tiny\begin{array}{c c c c c c c c}1&1&0&0&0&0&0\\&&0\end{array}}+
e_{\tiny\begin{array}{c c c c c c c c}0&0&1&0&0&0&0\\&&1\end{array}}+
e_{\tiny\begin{array}{c c c c c c c c}0&0&0&1&1&0&0\\&&0\end{array}}+
e_{\tiny\begin{array}{c c c c c c c c}0&1&1&1&0&0&0\\&&0\end{array}}$\\
$D_{6}(a_{2})$ & $e_{\tiny\begin{array}{c c c c c c c c}0&0&0&0&0&0&0\\&&1\end{array}}+
e_{\tiny\begin{array}{c c c c c c c c}0&1&0&0&0&0&0\\&&0\end{array}}+
e_{\tiny\begin{array}{c c c c c c c c}0&1&1&0&0&0&0\\&&0\end{array}}+
e_{\tiny\begin{array}{c c c c c c c c}0&0&1&1&0&0&0\\&&0\end{array}}+
e_{\tiny\begin{array}{c c c c c c c c}0&0&0&1&1&0&0\\&&0\end{array}}+
e_{\tiny\begin{array}{c c c c c c c c}0&0&0&0&1&1&0\\&&0\end{array}}$\\
$D_{5}(a_{1})+A_{2}$ & $e_{\tiny\begin{array}{c c c c c c c c}1&0&0&0&0&0&0\\&&0\end{array}}+
e_{\tiny\begin{array}{c c c c c c c c}0&1&0&0&0&0&0\\&&0\end{array}}+
e_{\tiny\begin{array}{c c c c c c c c}0&0&0&1&0&0&0\\&&0\end{array}}+
e_{\tiny\begin{array}{c c c c c c c c}0&0&0&0&0&1&0\\&&0\end{array}}+
e_{\tiny\begin{array}{c c c c c c c c}0&0&0&0&0&0&1\\&&0\end{array}}+
e_{\tiny\begin{array}{c c c c c c c c}0&0&1&0&0&0&0\\&&1\end{array}}+
e_{\tiny\begin{array}{c c c c c c c c}0&0&1&1&0&0&0\\&&0\end{array}}$\\
$A_{5}+A_{1}$ & $e_{\tiny\begin{array}{c c c c c c c c}1&0&0&0&0&0&0\\&&0\end{array}}+
e_{\tiny\begin{array}{c c c c c c c c}0&1&0&0&0&0&0\\&&0\end{array}}+
e_{\tiny\begin{array}{c c c c c c c c}0&0&1&0&0&0&0\\&&0\end{array}}+
e_{\tiny\begin{array}{c c c c c c c c}0&0&0&1&0&0&0\\&&0\end{array}}+
e_{\tiny\begin{array}{c c c c c c c c}0&0&0&0&1&0&0\\&&0\end{array}}+
e_{\tiny\begin{array}{c c c c c c c c}0&0&0&0&0&0&1\\&&0\end{array}}$\\
$A_{4}+A_{3}$ & $e_{\tiny\begin{array}{c c c c c c c c}1&0&0&0&0&0&0\\&&0\end{array}}+
e_{\tiny\begin{array}{c c c c c c c c}0&0&0&0&0&0&0\\&&1\end{array}}+
e_{\tiny\begin{array}{c c c c c c c c}0&1&0&0&0&0&0\\&&0\end{array}}+
e_{\tiny\begin{array}{c c c c c c c c}0&0&1&0&0&0&0\\&&0\end{array}}+
e_{\tiny\begin{array}{c c c c c c c c}0&0&0&0&1&0&0\\&&0\end{array}}+
e_{\tiny\begin{array}{c c c c c c c c}0&0&0&0&0&1&0\\&&0\end{array}}+
e_{\tiny\begin{array}{c c c c c c c c}0&0&0&0&0&0&1\\&&0\end{array}}$\\
$D_{5}$ & $e_{\tiny\begin{array}{c c c c c c c c}1&0&0&0&0&0&0\\&&0\end{array}}+
e_{\tiny\begin{array}{c c c c c c c c}0&0&0&0&0&0&0\\&&1\end{array}}+
e_{\tiny\begin{array}{c c c c c c c c}0&1&0&0&0&0&0\\&&0\end{array}}+
e_{\tiny\begin{array}{c c c c c c c c}0&0&1&0&0&0&0\\&&0\end{array}}+
e_{\tiny\begin{array}{c c c c c c c c}0&0&0&1&0&0&0\\&&0\end{array}}$\\
$E_{6}(a_{3})$ & $e_{\tiny\begin{array}{c c c c c c c c}1&0&0&0&0&0&0\\&&0\end{array}}+
e_{\tiny\begin{array}{c c c c c c c c}0&0&0&0&1&0&0\\&&0\end{array}}+
e_{\tiny\begin{array}{c c c c c c c c}1&1&0&0&0&0&0\\&&0\end{array}}+
e_{\tiny\begin{array}{c c c c c c c c}0&0&1&0&0&0&0\\&&1\end{array}}+
e_{\tiny\begin{array}{c c c c c c c c}0&0&0&1&1&0&0\\&&0\end{array}}+
e_{\tiny\begin{array}{c c c c c c c c}0&1&1&1&0&0&0\\&&0\end{array}}$\\
$(D_{4}+A_{2})^{(2)}$ & $e_{\tiny\begin{array}{c c c c c c c c}0&0&1&0&0&0&0\\&&0\end{array}}+
e_{\tiny\begin{array}{c c c c c c c c}0&0&0&0&0&0&1\\&&0\end{array}}+
e_{\tiny\begin{array}{c c c c c c c c}0&0&0&0&1&1&0\\&&0\end{array}}+
e_{\tiny\begin{array}{c c c c c c c c}0&1&1&0&0&0&0\\&&1\end{array}}+
e_{\tiny\begin{array}{c c c c c c c c}0&0&1&1&0&0&0\\&&1\end{array}}+
e_{\tiny\begin{array}{c c c c c c c c}0&0&0&1&1&1&0\\&&0\end{array}}+
e_{\tiny\begin{array}{c c c c c c c c}0&1&1&1&1&0&0\\&&0\end{array}}$\\
$D_{4}+A_{2}$ & $e_{\tiny\begin{array}{c c c c c c c c}0&0&0&0&0&0&0\\&&1\end{array}}+
e_{\tiny\begin{array}{c c c c c c c c}0&1&0&0&0&0&0\\&&0\end{array}}+
e_{\tiny\begin{array}{c c c c c c c c}0&0&1&0&0&0&0\\&&0\end{array}}+
e_{\tiny\begin{array}{c c c c c c c c}0&0&0&1&0&0&0\\&&0\end{array}}+
e_{\tiny\begin{array}{c c c c c c c c}0&0&0&0&0&1&0\\&&0\end{array}}+
e_{\tiny\begin{array}{c c c c c c c c}0&0&0&0&0&0&1\\&&0\end{array}}$\\
$A_{4}+A_{2}+A_{1}$ & $e_{\tiny\begin{array}{c c c c c c c c}1&0&0&0&0&0&0\\&&0\end{array}}+
e_{\tiny\begin{array}{c c c c c c c c}0&0&0&0&0&0&0\\&&1\end{array}}+
e_{\tiny\begin{array}{c c c c c c c c}0&1&0&0&0&0&0\\&&0\end{array}}+
e_{\tiny\begin{array}{c c c c c c c c}0&0&0&1&0&0&0\\&&0\end{array}}+
e_{\tiny\begin{array}{c c c c c c c c}0&0&0&0&1&0&0\\&&0\end{array}}+
e_{\tiny\begin{array}{c c c c c c c c}0&0&0&0&0&1&0\\&&0\end{array}}+
e_{\tiny\begin{array}{c c c c c c c c}0&0&0&0&0&0&1\\&&0\end{array}}$\\
$D_{5}(a_{1})+A_{1}$ & $e_{\tiny\begin{array}{c c c c c c c c}1&0&0&0&0&0&0\\&&0\end{array}}+
e_{\tiny\begin{array}{c c c c c c c c}0&1&0&0&0&0&0\\&&0\end{array}}+
e_{\tiny\begin{array}{c c c c c c c c}0&0&0&1&0&0&0\\&&0\end{array}}+
e_{\tiny\begin{array}{c c c c c c c c}0&0&0&0&0&1&0\\&&0\end{array}}+
e_{\tiny\begin{array}{c c c c c c c c}0&0&1&0&0&0&0\\&&1\end{array}}+
e_{\tiny\begin{array}{c c c c c c c c}0&0&1&1&0&0&0\\&&0\end{array}}$\\
$A_{5}$ & $e_{\tiny\begin{array}{c c c c c c c c}1&0&0&0&0&0&0\\&&0\end{array}}+
e_{\tiny\begin{array}{c c c c c c c c}0&1&0&0&0&0&0\\&&0\end{array}}+
e_{\tiny\begin{array}{c c c c c c c c}0&0&1&0&0&0&0\\&&0\end{array}}+
e_{\tiny\begin{array}{c c c c c c c c}0&0&0&1&0&0&0\\&&0\end{array}}+
e_{\tiny\begin{array}{c c c c c c c c}0&0&0&0&1&0&0\\&&0\end{array}}$\\
$A_{4}+A_{2}$ & $e_{\tiny\begin{array}{c c c c c c c c}1&0&0&0&0&0&0\\&&0\end{array}}+
e_{\tiny\begin{array}{c c c c c c c c}0&1&0&0&0&0&0\\&&0\end{array}}+
e_{\tiny\begin{array}{c c c c c c c c}0&0&0&1&0&0&0\\&&0\end{array}}+
e_{\tiny\begin{array}{c c c c c c c c}0&0&0&0&1&0&0\\&&0\end{array}}+
e_{\tiny\begin{array}{c c c c c c c c}0&0&0&0&0&1&0\\&&0\end{array}}+
e_{\tiny\begin{array}{c c c c c c c c}0&0&0&0&0&0&1\\&&0\end{array}}$\\
$A_{4}+2A_{1}$ & $e_{\tiny\begin{array}{c c c c c c c c}1&0&0&0&0&0&0\\&&0\end{array}}+
e_{\tiny\begin{array}{c c c c c c c c}0&0&0&0&0&0&0\\&&1\end{array}}+
e_{\tiny\begin{array}{c c c c c c c c}0&0&0&1&0&0&0\\&&0\end{array}}+
e_{\tiny\begin{array}{c c c c c c c c}0&0&0&0&1&0&0\\&&0\end{array}}+
e_{\tiny\begin{array}{c c c c c c c c}0&0&0&0&0&1&0\\&&0\end{array}}+
e_{\tiny\begin{array}{c c c c c c c c}0&0&0&0&0&0&1\\&&0\end{array}}$\\
$D_{5}(a_{1})$ & $e_{\tiny\begin{array}{c c c c c c c c}1&0&0&0&0&0&0\\&&0\end{array}}+
e_{\tiny\begin{array}{c c c c c c c c}0&1&0&0&0&0&0\\&&0\end{array}}+
e_{\tiny\begin{array}{c c c c c c c c}0&0&0&1&0&0&0\\&&0\end{array}}+
e_{\tiny\begin{array}{c c c c c c c c}0&0&1&0&0&0&0\\&&1\end{array}}+
e_{\tiny\begin{array}{c c c c c c c c}0&0&1&1&0&0&0\\&&0\end{array}}$\\
$2A_{3}$ & $e_{\tiny\begin{array}{c c c c c c c c}1&0&0&0&0&0&0\\&&0\end{array}}+
e_{\tiny\begin{array}{c c c c c c c c}0&1&0&0&0&0&0\\&&0\end{array}}+
e_{\tiny\begin{array}{c c c c c c c c}0&0&1&0&0&0&0\\&&0\end{array}}+
e_{\tiny\begin{array}{c c c c c c c c}0&0&0&0&1&0&0\\&&0\end{array}}+
e_{\tiny\begin{array}{c c c c c c c c}0&0&0&0&0&1&0\\&&0\end{array}}+
e_{\tiny\begin{array}{c c c c c c c c}0&0&0&0&0&0&1\\&&0\end{array}}$\\
$A_{4}+A_{1}$ & $e_{\tiny\begin{array}{c c c c c c c c}1&0&0&0&0&0&0\\&&0\end{array}}+
e_{\tiny\begin{array}{c c c c c c c c}0&1&0&0&0&0&0\\&&0\end{array}}+
e_{\tiny\begin{array}{c c c c c c c c}0&0&1&0&0&0&0\\&&0\end{array}}+
e_{\tiny\begin{array}{c c c c c c c c}0&0&0&1&0&0&0\\&&0\end{array}}+
e_{\tiny\begin{array}{c c c c c c c c}0&0&0&0&0&1&0\\&&0\end{array}}$\\
$D_{4}(a_{1})+A_{2}$ & $e_{\tiny\begin{array}{c c c c c c c c}0&1&0&0&0&0&0\\&&0\end{array}}+
e_{\tiny\begin{array}{c c c c c c c c}0&0&0&1&0&0&0\\&&0\end{array}}+
e_{\tiny\begin{array}{c c c c c c c c}0&0&0&0&0&1&0\\&&0\end{array}}+
e_{\tiny\begin{array}{c c c c c c c c}0&0&0&0&0&0&1\\&&0\end{array}}+
e_{\tiny\begin{array}{c c c c c c c c}0&0&1&0&0&0&0\\&&1\end{array}}+
e_{\tiny\begin{array}{c c c c c c c c}0&0&1&1&0&0&0\\&&0\end{array}}$\\
$D_{4}+A_{1}$ & $e_{\tiny\begin{array}{c c c c c c c c}0&0&0&0&0&0&0\\&&1\end{array}}+
e_{\tiny\begin{array}{c c c c c c c c}0&1&0&0&0&0&0\\&&0\end{array}}+
e_{\tiny\begin{array}{c c c c c c c c}0&0&1&0&0&0&0\\&&0\end{array}}+
e_{\tiny\begin{array}{c c c c c c c c}0&0&0&1&0&0&0\\&&0\end{array}}+
e_{\tiny\begin{array}{c c c c c c c c}0&0&0&0&0&1&0\\&&0\end{array}}$\\
$A_{3}+A_{2}+A_{1}$ & $e_{\tiny\begin{array}{c c c c c c c c}1&0&0&0&0&0&0\\&&0\end{array}}+
e_{\tiny\begin{array}{c c c c c c c c}0&0&0&0&0&0&0\\&&1\end{array}}+
e_{\tiny\begin{array}{c c c c c c c c}0&1&0&0&0&0&0\\&&0\end{array}}+
e_{\tiny\begin{array}{c c c c c c c c}0&0&0&1&0&0&0\\&&0\end{array}}+
e_{\tiny\begin{array}{c c c c c c c c}0&0&0&0&1&0&0\\&&0\end{array}}+
e_{\tiny\begin{array}{c c c c c c c c}0&0&0&0&0&1&0\\&&0\end{array}}$\\
$A_{4}$ & $e_{\tiny\begin{array}{c c c c c c c c}1&0&0&0&0&0&0\\&&0\end{array}}+
e_{\tiny\begin{array}{c c c c c c c c}0&1&0&0&0&0&0\\&&0\end{array}}+
e_{\tiny\begin{array}{c c c c c c c c}0&0&1&0&0&0&0\\&&0\end{array}}+
e_{\tiny\begin{array}{c c c c c c c c}0&0&0&1&0&0&0\\&&0\end{array}}$\\
$(A_{3}+A_{2})^{(2)}$ & $e_{\tiny\begin{array}{c c c c c c c c}0&0&0&0&0&1&0\\&&0\end{array}}+
e_{\tiny\begin{array}{c c c c c c c c}0&0&1&0&0&0&0\\&&1\end{array}}+
e_{\tiny\begin{array}{c c c c c c c c}0&0&0&0&1&1&0\\&&0\end{array}}+
e_{\tiny\begin{array}{c c c c c c c c}0&1&1&1&0&0&0\\&&0\end{array}}+
e_{\tiny\begin{array}{c c c c c c c c}0&0&1&1&1&0&0\\&&0\end{array}}+
e_{\tiny\begin{array}{c c c c c c c c}0&1&1&1&1&0&0\\&&1\end{array}}$\\
$A_{3}+A_{2}$ & $e_{\tiny\begin{array}{c c c c c c c c}1&0&0&0&0&0&0\\&&0\end{array}}+
e_{\tiny\begin{array}{c c c c c c c c}0&1&0&0&0&0&0\\&&0\end{array}}+
e_{\tiny\begin{array}{c c c c c c c c}0&0&1&0&0&0&0\\&&0\end{array}}+
e_{\tiny\begin{array}{c c c c c c c c}0&0&0&0&1&0&0\\&&0\end{array}}+
e_{\tiny\begin{array}{c c c c c c c c}0&0&0&0&0&1&0\\&&0\end{array}}$\\
$D_{4}(a_{1})+A_{1}$ & $e_{\tiny\begin{array}{c c c c c c c c}0&1&0&0&0&0&0\\&&0\end{array}}+
e_{\tiny\begin{array}{c c c c c c c c}0&0&0&1&0&0&0\\&&0\end{array}}+
e_{\tiny\begin{array}{c c c c c c c c}0&0&0&0&0&1&0\\&&0\end{array}}+
e_{\tiny\begin{array}{c c c c c c c c}0&0&1&0&0&0&0\\&&1\end{array}}+
e_{\tiny\begin{array}{c c c c c c c c}0&0&1&1&0&0&0\\&&0\end{array}}$\\
$A_{3}+2A_{1}$ & $e_{\tiny\begin{array}{c c c c c c c c}1&0&0&0&0&0&0\\&&0\end{array}}+
e_{\tiny\begin{array}{c c c c c c c c}0&1&0&0&0&0&0\\&&0\end{array}}+
e_{\tiny\begin{array}{c c c c c c c c}0&0&1&0&0&0&0\\&&0\end{array}}+
e_{\tiny\begin{array}{c c c c c c c c}0&0&0&0&1&0&0\\&&0\end{array}}+
e_{\tiny\begin{array}{c c c c c c c c}0&0&0&0&0&0&1\\&&0\end{array}}$\\
$2A_{2}+2A_{1}$ & $e_{\tiny\begin{array}{c c c c c c c c}1&0&0&0&0&0&0\\&&0\end{array}}+
e_{\tiny\begin{array}{c c c c c c c c}0&0&0&0&0&0&0\\&&1\end{array}}+
e_{\tiny\begin{array}{c c c c c c c c}0&1&0&0&0&0&0\\&&0\end{array}}+
e_{\tiny\begin{array}{c c c c c c c c}0&0&0&1&0&0&0\\&&0\end{array}}+
e_{\tiny\begin{array}{c c c c c c c c}0&0&0&0&1&0&0\\&&0\end{array}}+
e_{\tiny\begin{array}{c c c c c c c c}0&0&0&0&0&0&1\\&&0\end{array}}$\\
$D_{4}$ & $e_{\tiny\begin{array}{c c c c c c c c}0&0&0&0&0&0&0\\&&1\end{array}}+
e_{\tiny\begin{array}{c c c c c c c c}0&1&0&0&0&0&0\\&&0\end{array}}+
e_{\tiny\begin{array}{c c c c c c c c}0&0&1&0&0&0&0\\&&0\end{array}}+
e_{\tiny\begin{array}{c c c c c c c c}0&0&0&1&0&0&0\\&&0\end{array}}$\\
$D_{4}(a_{1})$ & $e_{\tiny\begin{array}{c c c c c c c c}0&1&0&0&0&0&0\\&&0\end{array}}+
e_{\tiny\begin{array}{c c c c c c c c}0&0&0&1&0&0&0\\&&0\end{array}}+
e_{\tiny\begin{array}{c c c c c c c c}0&0&1&0&0&0&0\\&&1\end{array}}+
e_{\tiny\begin{array}{c c c c c c c c}0&0&1&1&0&0&0\\&&0\end{array}}$\\
$A_{3}+A_{1}$ & $e_{\tiny\begin{array}{c c c c c c c c}1&0&0&0&0&0&0\\&&0\end{array}}+
e_{\tiny\begin{array}{c c c c c c c c}0&1&0&0&0&0&0\\&&0\end{array}}+
e_{\tiny\begin{array}{c c c c c c c c}0&0&1&0&0&0&0\\&&0\end{array}}+
e_{\tiny\begin{array}{c c c c c c c c}0&0&0&0&1&0&0\\&&0\end{array}}$\\
$2A_{2}+A_{1}$ & $e_{\tiny\begin{array}{c c c c c c c c}1&0&0&0&0&0&0\\&&0\end{array}}+
e_{\tiny\begin{array}{c c c c c c c c}0&1&0&0&0&0&0\\&&0\end{array}}+
e_{\tiny\begin{array}{c c c c c c c c}0&0&0&1&0&0&0\\&&0\end{array}}+
e_{\tiny\begin{array}{c c c c c c c c}0&0&0&0&1&0&0\\&&0\end{array}}+
e_{\tiny\begin{array}{c c c c c c c c}0&0&0&0&0&0&1\\&&0\end{array}}$\\
$2A_{2}$ & $e_{\tiny\begin{array}{c c c c c c c c}1&0&0&0&0&0&0\\&&0\end{array}}+
e_{\tiny\begin{array}{c c c c c c c c}0&1&0&0&0&0&0\\&&0\end{array}}+
e_{\tiny\begin{array}{c c c c c c c c}0&0&0&1&0&0&0\\&&0\end{array}}+
e_{\tiny\begin{array}{c c c c c c c c}0&0&0&0&1&0&0\\&&0\end{array}}$\\
$A_{2}+3A_{1}$ & $e_{\tiny\begin{array}{c c c c c c c c}1&0&0&0&0&0&0\\&&0\end{array}}+
e_{\tiny\begin{array}{c c c c c c c c}0&0&0&0&0&0&0\\&&1\end{array}}+
e_{\tiny\begin{array}{c c c c c c c c}0&1&0&0&0&0&0\\&&0\end{array}}+
e_{\tiny\begin{array}{c c c c c c c c}0&0&0&1&0&0&0\\&&0\end{array}}+
e_{\tiny\begin{array}{c c c c c c c c}0&0&0&0&0&1&0\\&&0\end{array}}$\\
$A_{3}$ & $e_{\tiny\begin{array}{c c c c c c c c}1&0&0&0&0&0&0\\&&0\end{array}}+
e_{\tiny\begin{array}{c c c c c c c c}0&1&0&0&0&0&0\\&&0\end{array}}+
e_{\tiny\begin{array}{c c c c c c c c}0&0&1&0&0&0&0\\&&0\end{array}}$\\
$A_{2}+2A_{1}$ & $e_{\tiny\begin{array}{c c c c c c c c}1&0&0&0&0&0&0\\&&0\end{array}}+
e_{\tiny\begin{array}{c c c c c c c c}0&1&0&0&0&0&0\\&&0\end{array}}+
e_{\tiny\begin{array}{c c c c c c c c}0&0&0&1&0&0&0\\&&0\end{array}}+
e_{\tiny\begin{array}{c c c c c c c c}0&0&0&0&0&1&0\\&&0\end{array}}$\\
$A_{2}+A_{1}$ & $e_{\tiny\begin{array}{c c c c c c c c}1&0&0&0&0&0&0\\&&0\end{array}}+
e_{\tiny\begin{array}{c c c c c c c c}0&1&0&0&0&0&0\\&&0\end{array}}+
e_{\tiny\begin{array}{c c c c c c c c}0&0&0&1&0&0&0\\&&0\end{array}}$\\
$4A_{1}$ & $e_{\tiny\begin{array}{c c c c c c c c}1&0&0&0&0&0&0\\&&0\end{array}}+
e_{\tiny\begin{array}{c c c c c c c c}0&0&1&0&0&0&0\\&&0\end{array}}+
e_{\tiny\begin{array}{c c c c c c c c}0&0&0&0&1&0&0\\&&0\end{array}}+
e_{\tiny\begin{array}{c c c c c c c c}0&0&0&0&0&0&1\\&&0\end{array}}$\\
$A_{2}$ & $e_{\tiny\begin{array}{c c c c c c c c}1&0&0&0&0&0&0\\&&0\end{array}}+
e_{\tiny\begin{array}{c c c c c c c c}0&1&0&0&0&0&0\\&&0\end{array}}$\\
$3A_{1}$ & $e_{\tiny\begin{array}{c c c c c c c c}1&0&0&0&0&0&0\\&&0\end{array}}+
e_{\tiny\begin{array}{c c c c c c c c}0&0&1&0&0&0&0\\&&0\end{array}}+
e_{\tiny\begin{array}{c c c c c c c c}0&0&0&0&1&0&0\\&&0\end{array}}$\\
$2A_{1}$ & $e_{\tiny\begin{array}{c c c c c c c c}1&0&0&0&0&0&0\\&&0\end{array}}+
e_{\tiny\begin{array}{c c c c c c c c}0&0&1&0&0&0&0\\&&0\end{array}}$\\
$A_{1}$ & $e_{\tiny\begin{array}{c c c c c c c c}1&0&0&0&0&0&0\\&&0\end{array}}$\\
\\
\caption{Representatives of nilpotent elements for $E_8$}\label{r4}
\end{longtable}
\FloatBarrier
\def\arraystretch{\oldarraystretch} \arraycolsep=\oldarraycolsep 
\section*{Appendix C: Jordan block structures on the adjoint modules}
\begin{multicols}{2}\setbox\ltmcbox\vbox{\makeatletter\col@number\@ne\tiny\begin{longtable}{l l l}$E_{8}$ & $2$ & $16^{8},15^{8}$\\
 & $3$ & $27^{7},26^{2},3,2^{2}$\\
 & $5$ & $25^{8},24^{2}$\\
 & $7$ & $49^{2},47,35,27,23,15,3$\\
 & $19$ & $59,47,39,35,27,19^{2},3$\\
 & $29$ & $59,47,39,29^{2},27,15,3$\\
 & $31$ & $31^{8}$\\
 & $37$ & $37^{6},23,3$\\
 & $41$ & $41^{4},39,27,15,3$\\
 & $43$ & $43^{4},35,23,15,3$\\
 & $47$ & $47^{3},39,27,23,15,3$\\
 & $53$ & $53^{2},39,35,27,23,15,3$\\
 & others & $59,47,39,35,27,23,15,3$\\
\hline
$E_{8}(a_{1})$ & $2$ & $16^{8},15^{8}$\\
 & $3$ & $27^{7},26^{2},3,2^{2}$\\
 & $5$ & $25^{8},24^{2}$\\
 & $7$ & $47,39,35,28^{2},21^{2},15,11,3$\\
 & $13$ & $47,39,35,27,26^{2},19,13^{2},3$\\
 & $17$ & $47,35,34^{2},27,17^{4},3$\\
 & $19$ & $39,38^{2},19^{7}$\\
 & $23$ & $47,39,35,29,23^{3},15,11,3$\\
 & $29$ & $29^{7},27,15,3$\\
 & $31$ & $31^{6},29,19,11,3$\\
 & $37$ & $37^{4},29,23,19,15,11,3$\\
 & $41$ & $41^{2},39,29,27,23,19,15,11,3$\\
 & $43$ & $43^{2},35,29,27,23,19,15,11,3$\\
 & others & $47,39,35,29,27,23,19,15,11,3$\\
\hline
$E_{8}(a_{2})$ & $2$ & $16^{8},15^{8}$\\
 & $3$ & $27^{5},26^{2},20^{2},9^{2},3$\\
 & $5$ & $25^{6},24^{2},20^{2},5^{2}$\\
 & $7$ & $39,35,28^{2},23,21^{2},15,14^{2},7,3$\\
 & $13$ & $39,27,26^{4},13^{6}$\\
 & $17$ & $35,34^{2},23,17^{7},3$\\
 & $19$ & $39,35,29,27,23,19^{3},17,11,7,3$\\
 & $23$ & $23^{10},15,3$\\
 & $29$ & $29^{5},27,23,17,15,11,7,3$\\
 & $31$ & $31^{4},29,23,19,17,15,11,7,3$\\
 & $37$ & $37^{2},29,27,23^{2},19,17,15,11,7,3$\\
 & others & $39,35,29,27,23^{2},19,17,15,11,7,3$\\
\hline
$E_{8}(a_{3})$ & $2$ & $16^{8},15^{8}$\\
 & $3$ & $27^{3},20^{2},18^{4},9^{4},8^{2},3$\\
 & $5$ & $25^{4},24^{2},22^{2},15^{2},10^{2},3^{2}$\\
 & $7$ & $35,28^{2},21^{4},14^{4},7^{2},3$\\
 & $13$ & $35,29,27^{2},23,19^{2},17,13^{2},11,9,3^{2}$\\
 & $17$ & $35,29,27^{2},23,19,17^{3},11^{2},9,3^{2}$\\
 & $19$ & $19^{12},17,3$\\
 & $23$ & $23^{9},15,11,9,3^{2}$\\
 & $29$ & $29^{3},27^{2},19^{2},17,15,11^{2},9,3^{2}$\\
 & $31$ & $31^{2},29,27,23,19^{2},17,15,11^{2},9,3^{2}$\\
 & others & $35,29,27^{2},23,19^{2},17,15,11^{2},9,3^{2}$\\
\hline
$E_{8}(a_{4})$ & $2$ & $16^{8},15^{8}$\\
 & $3$ & $27,26^{2},21^{2},18^{2},15^{2},12^{2},9^{2},8^{2},3$\\
 & $5$ & $25^{2},24^{2},20^{2},15^{4},10^{4},5^{2}$\\
 & $7$ & $29,27,21^{4},15^{3},14^{2},7^{5}$\\
 & $13$ & $27,26^{2},23,19,15,13^{8},5,3$\\
 & $17$ & $17^{13},15,9,3$\\
 & $19$ & $19^{10},17,15,11,7,5,3$\\
 & $23$ & $23^{6},19,15^{3},11^{2},9,7,5,3$\\
 & others & $29,27,23^{2},19^{2},17,15^{3},11^{2},9,7,5,3$\\
\hline
$E_{7}$ & $2$ & $16^{6},15^{2},14^{8},2^{2},1^{6}$\\
 & $3$ & $27^{3},19^{4},18^{2},9^{4},8^{2},1^{3}$\\
 & $5$ & $25^{4},23^{4},15^{2},10^{2},3,1^{3}$\\
 & $7$ & $35,28^{2},21^{4},14^{4},7^{2},1^{3}$\\
 & $13$ & $35,28^{2},27,23,19,18^{2},13^{2},10^{2},3,1^{3}$\\
 & $17$ & $35,28^{2},27,23,18^{2},17^{2},11,10^{2},3,1^{3}$\\
 & $19$ & $19^{11},18^{2},1^{3}$\\
 & $23$ & $23^{9},15,10^{2},3,1^{3}$\\
 & $29$ & $29^{2},28^{2},27,19,18^{2},15,11,10^{2},3,1^{3}$\\
 & $31$ & $31^{2},28^{2},23,19,18^{2},15,11,10^{2},3,1^{3}$\\
 & others & $35,28^{2},27,23,19,18^{2},15,11,10^{2},3,1^{3}$\\
\hline
$E_{8}(b_{4})$ & $2$ & $16^{6},15^{2},14^{8},2^{2},1^{6}$\\
 & $3$ & $27,20^{2},19,18^{2},17^{3},9^{8},3$\\
 & $5$ & $25^{2},23,21,19,17^{2},15^{2},13,11^{3},7^{2},5,3^{2}$\\
 & $7$ & $27,23,21^{3},15,14^{6},7^{4},5,3$\\
 & $11$ & $23,22^{4},13,11^{11},3$\\
 & $13$ & $27,23^{2},21,19,17^{2},15,13^{3},11^{2},7^{2},5,3^{2}$\\
 & $17$ & $17^{12},15,11,7,5,3^{2}$\\
 & $19$ & $19^{9},17,13,11^{2},7^{2},5,3^{2}$\\
 & $23$ & $23^{4},21,17^{2},15^{2},13,11^{3},7^{2},5,3^{2}$\\
 & others & $27,23^{2},21,19,17^{2},15^{2},13,11^{3},7^{2},5,3^{2}$\\
\hline
$E_{8}(a_{5})$ & $2$ & $16^{4},15^{4},12^{8},4^{4},3^{4}$\\
 & $3$ & $23^{2},21,19,17,15^{3},12^{4},11^{2},9,7,5,3^{3}$\\
 & $5$ & $23^{2},21,19,17,15^{3},13^{2},11^{3},10^{2},7,5,3^{3}$\\
 & $7$ & $23,21^{3},15,14^{6},11,7^{7},3$\\
 & $11$ & $23,22^{2},19,17,15^{2},13,11^{8},5,3^{3}$\\
 & $13$ & $13^{18},11,3$\\
 & $17$ & $17^{9},15^{2},13,11^{2},9,7,5,3^{3}$\\
 & $19$ & $19^{7},15,13^{2},11^{4},9,7,5,3^{3}$\\
 & others & $23^{2},21,19,17,15^{3},13^{2},11^{4},9,7,5,3^{3}$\\
\hline
$E_{7}(a_{1})$ & $2$ & $16^{6},15^{2},14^{8},2^{2},1^{6}$\\
 & $3$ & $27,19^{5},17^{3},9^{8},1^{3}$\\
 & $5$ & $25^{2},22^{2},19,17,16^{2},15,12^{2},11^{2},7,6^{2},3,1^{3}$\\
 & $7$ & $27,22^{2},21^{2},15,14^{6},7^{3},6^{2},1^{3}$\\
 & $11$ & $23,22^{4},12^{2},11^{10},1^{3}$\\
 & $13$ & $27,23,22^{2},19,17,16^{2},13^{2},12^{2},11,7,6^{2},3,1^{3}$\\
 & $17$ & $17^{11},16^{2},11,6^{2},3,1^{3}$\\
 & $19$ & $19^{9},17,12^{2},11,7,6^{2},3,1^{3}$\\
 & $23$ & $23^{3},22^{2},17,16^{2},15,12^{2},11^{2},7,6^{2},3,1^{3}$\\
 & others & $27,23,22^{2},19,17,16^{2},15,12^{2},11^{2},7,6^{2},3,1^{3}$\\
\hline
$E_{8}(b_{5})$ & $2$ & $16^{4},15^{4},12^{8},4^{4},3^{4}$\\
 & $3$ & $19,18^{2},17^{3},11^{3},9^{11},7,3$\\
 & $5$ & $23,19^{2},17^{3},15^{3},11,10^{4},9,7^{2},5,3^{4}$\\
 & $13$ & $13^{18},5,3^{3}$\\
 & $17$ & $17^{9},15,11^{2},9^{3},7^{2},5,3^{4}$\\
 & $19$ & $19^{4},17^{3},15^{2},11^{3},9^{3},7^{2},5,3^{4}$\\
 & others & $23,19^{2},17^{3},15^{3},11^{3},9^{3},7^{2},5,3^{4}$\\
\hline
$(D_{7})^{(2)}$ & $2$ & $16^{2},15^{6},8^{14},7^{2}$\\
\hline
$D_{7}$ & $2$ & $8^{24},7^{8}$\\
 & $7$ & $22^{2},21^{2},15,14^{4},13^{3},7^{7},1^{3}$\\
 & $11$ & $23,22^{2},19,16^{2},13^{3},11^{7},4^{2},3,1^{3}$\\
 & $13$ & $13^{17},12^{2},1^{3}$\\
 & $17$ & $17^{8},16^{2},13^{3},10^{2},7,4^{2},3,1^{3}$\\
 & $19$ & $19^{7},13^{3},12^{2},11,10^{2},7,4^{2},3,1^{3}$\\
 & others & $23,22^{2},19,16^{2},15,13^{3},12^{2},11,10^{2},7,4^{2},3,1^{3}$\\
\hline
$E_{8}(a_{6})$ & $2$ & $16^{2},15^{6},8^{14},7^{2}$\\
 & $3$ & $19,18^{2},15^{3},13^{3},9^{9},7,6^{2},3^{3}$\\
 & $5$ & $19^{2},15^{5},11,10^{8},9,5^{7}$\\
 & $7$ & $19^{2},17,15,14^{4},13,11^{3},9^{2},7^{7},3^{3}$\\
 & $11$ & $11^{21},7^{2},3$\\
 & $13$ & $13^{15},9^{2},7^{3},5,3^{3}$\\
 & $17$ & $17^{5},15,13^{3},11^{3},9^{3},7^{5},5,3^{3}$\\
 & others & $19^{2},17,15^{3},13^{3},11^{3},9^{3},7^{5},5,3^{3}$\\
\hline
$E_{7}(a_{2})$ & $2$ & $16^{4},14^{4},13^{4},12^{4},4^{2},3^{2},2^{6},1^{2}$\\
 & $3$ & $19,18^{2},17^{3},10^{6},9^{7},8^{2},1^{3}$\\
 & $5$ & $23,19,18^{2},17,16^{2},15^{2},11,10^{4},8^{2},7,4^{2},3^{2},1^{3}$\\
 & $13$ & $13^{18},4^{2},3,1^{3}$\\
 & $17$ & $17^{9},15,11,10^{2},9,8^{2},7,4^{2},3^{2},1^{3}$\\
 & $19$ & $19^{3},18^{2},17,16^{2},15,11^{2},10^{2},9,8^{2},7,4^{2},3^{2},1^{3}$\\
 & others & $23,19,18^{2},17,16^{2},15^{2},11^{2},10^{2},9,8^{2},7,4^{2},3^{2},1^{3}$\\
\hline
$E_{6}+A_{1}$ & $2$ & $16^{4},14^{4},13^{4},12^{4},4^{2},3^{2},2^{6},1^{2}$\\
 & $3$ & $9^{25},8^{2},3,2^{2}$\\
 & $13$ & $13^{18},3,2^{4},1^{3}$\\
 & $17$ & $17^{9},15,10^{2},9^{3},8^{2},3^{2},2^{4},1^{3}$\\
 & $19$ & $19^{2},18^{2},17^{3},16^{2},11,10^{2},9^{3},8^{2},3^{2},2^{4},1^{3}$\\
 & others & $23,18^{2},17^{3},16^{2},15,11,10^{2},9^{3},8^{2},3^{2},2^{4},1^{3}$\\
\hline
$(D_{7}(a_{1}))^{(2)}$ & $2$ & $16^{2},15^{2},14^{4},9^{4},8^{8},6^{4},2^{2},1^{2}$\\
\hline
$D_{7}(a_{1})$ & $2$ & $8^{24},7^{8}$\\
 & $3$ & $19,17^{2},15^{3},12^{2},11^{4},9^{5},6^{4},3^{4},1$\\
 & $5$ & $19,17^{2},15^{3},13,11^{3},10^{6},7^{2},5^{4},3^{3},1$\\
 & $11$ & $11^{20},9^{2},3^{3},1$\\
 & $13$ & $13^{13},11^{3},9,7^{2},5^{2},3^{4},1$\\
 & $17$ & $17^{4},15^{2},13,11^{6},9^{3},7^{3},5^{2},3^{4},1$\\
 & others & $19,17^{2},15^{3},13,11^{6},9^{3},7^{3},5^{2},3^{4},1$\\
\hline
$E_{8}(b_{6})$ & $2$ & $16^{2},15^{2},12^{4},11^{4},8^{6},7^{2},4^{8}$\\
 & $3$ & $9^{25},8^{2},3,2^{2}$\\
 & $5$ & $17,15^{3},13^{2},11^{3},10^{6},7^{3},5^{8},3^{2}$\\
 & $7$ & $15,14^{6},11^{2},7^{17},5,3$\\
 & $11$ & $11^{18},9,7^{2},5^{3},3^{4}$\\
 & $13$ & $13^{10},11^{3},9^{2},7^{5},5^{4},3^{4}$\\
 & others & $17,15^{3},13^{2},11^{6},9^{3},7^{5},5^{4},3^{4}$\\
\hline
$E_{7}(a_{3})$ & $2$ & $16^{2},15^{2},14^{4},9^{4},8^{8},6^{4},2^{2},1^{2}$\\
 & $3$ & $19,17,16^{2},15^{2},12^{2},11^{2},10^{4},9^{3},6^{4},3^{2},2^{2},1^{3}$\\
 & $5$ & $19,17,16^{2},15^{2},12^{2},11^{2},10^{6},7,6^{2},5^{3},3,2^{2},1^{3}$\\
 & $11$ & $11^{19},10^{2},9,3,2^{2},1^{3}$\\
 & $13$ & $13^{12},12^{2},11,10^{2},7,6^{2},5,3^{2},2^{2},1^{3}$\\
 & $17$ & $17^{3},16^{2},15,12^{2},11^{3},10^{4},9,7^{2},6^{2},5,3^{2},2^{2},1^{3}$\\
 & others & $19,17,16^{2},15^{2},12^{2},11^{3},10^{4},9,7^{2},6^{2},5,3^{2},2^{2},1^{3}$\\
\hline
$E_{6}(a_{1})+A_{1}$ & $2$ & $16^{2},14^{2},13^{2},12^{2},11^{2},10^{2},9^{2},8^{4},6^{2},5^{2},4^{4},2^{4}$\\
 & $3$ & $9^{25},8^{2},3,2^{2}$\\
 & $5$ & $17,15,14^{2},13^{2},12^{2},11,10^{4},9^{2},8^{2},7,5^{7},3^{2},2^{2},1$\\
 & $7$ & $15,14^{4},13^{2},12^{2},7^{17},3,2^{2},1$\\
 & $11$ & $11^{18},9,6^{2},5^{2},4^{2},3^{2},2^{2},1$\\
 & $13$ & $13^{10},11,10^{2},9^{2},8^{2},7,6^{2},5^{3},4^{2},3^{2},2^{2},1$\\
 & others & $17,15,14^{2},13^{2},12^{2},11^{2},10^{2},9^{3},8^{2},7,6^{2},5^{3},4^{2},3^{2},2^{2},1$\\
\hline
$(A_{7})^{(3)}$ & $3$ & $9^{25},8^{2},3,2^{2}$\\
\hline
$A_{7}$ & $2$ & $8^{24},7^{8}$\\
 & $3$ & $9^{25},8^{2},1^{7}$\\
 & $5$ & $16^{2},15,13^{3},12^{2},11,10^{2},9^{3},8^{4},7,5^{7},3,1^{3}$\\
 & $7$ & $15,14^{4},13^{3},8^{4},7^{13},6^{2},1^{3}$\\
 & $11$ & $11^{17},8^{4},5^{3},4^{2},3,1^{3}$\\
 & $13$ & $13^{9},12^{2},9^{3},8^{4},7,6^{2},5^{3},4^{2},3,1^{3}$\\
 & others & $16^{2},15,13^{3},12^{2},11,10^{2},9^{3},8^{4},7,6^{2},5^{3},4^{2},3,1^{3}$\\
\hline
$D_{7}(a_{2})$ & $2$ & $8^{24},7^{8}$\\
 & $3$ & $9^{23},7,5^{6},3,1$\\
 & $5$ & $15,14^{2},11,10^{10},9^{2},5^{15},1$\\
 & $7$ & $15,14^{2},13,12^{2},11^{2},10^{2},9^{2},8^{2},7^{9},5^{2},4^{2},3^{2},2^{2},1$\\
 & $11$ & $11^{14},9^{2},8^{2},7^{2},6^{2},5^{3},4^{2},3^{2},2^{2},1$\\
 & $13$ & $13^{7},11,10^{2},9^{3},8^{4},7^{3},6^{2},5^{3},4^{2},3^{2},2^{2},1$\\
 & others & $15,14^{2},13,12^{2},11^{2},10^{2},9^{3},8^{4},7^{3},6^{2},5^{3},4^{2},3^{2},2^{2},1$\\
\hline
$E_{6}$ & $2$ & $16^{4},13^{12},4^{2},3^{2},1^{14}$\\
 & $3$ & $9^{25},8^{2},1^{7}$\\
 & $13$ & $13^{18},1^{14}$\\
 & $17$ & $17^{9},15,9^{7},3,1^{14}$\\
 & $19$ & $19^{2},17^{7},11,9^{7},3,1^{14}$\\
 & others & $23,17^{7},15,11,9^{7},3,1^{14}$\\
\hline
$D_{6}$ & $2$ & $8^{22},7^{2},6^{8},2^{2},1^{6}$\\
 & $3$ & $19,16^{4},15,11^{5},10^{4},9^{2},6^{4},3,1^{10}$\\
 & $5$ & $19,16^{4},15,11^{6},10^{4},6^{4},5^{2},1^{10}$\\
 & $11$ & $11^{18},10^{4},1^{10}$\\
 & $13$ & $13^{12},11^{5},6^{4},3,1^{10}$\\
 & $17$ & $17^{2},16^{4},11^{6},10^{4},7,6^{4},3,1^{10}$\\
 & others & $19,16^{4},15,11^{6},10^{4},7,6^{4},3,1^{10}$\\
\hline
$(D_{5}+A_{2})^{(2)}$ & $2$ & $8^{22},7^{2},6^{8},2^{2},1^{6}$\\
\hline
$D_{5}+A_{2}$ & $2$ & $8^{20},7^{4},4^{12},3^{4}$\\
 & $3$ & $9^{19},8^{4},6^{4},3^{7}$\\
 & $5$ & $15,13^{2},11^{3},10^{8},9,7^{2},5^{11},3^{5},1$\\
 & $11$ & $11^{13},9^{3},7^{4},5^{5},3^{8},1$\\
 & $13$ & $13^{4},11^{6},9^{5},7^{5},5^{5},3^{8},1$\\
 & others & $15,13^{2},11^{7},9^{5},7^{5},5^{5},3^{8},1$\\
\hline
$E_{6}(a_{1})$ & $2$ & $16^{2},13^{6},11^{2},9^{6},8^{2},5^{6},4^{2},1^{8}$\\
 & $3$ & $9^{25},8^{2},1^{7}$\\
 & $5$ & $17,15,13^{6},11,10^{2},9^{6},7,5^{7},3,1^{8}$\\
 & $7$ & $15,14^{2},13^{6},7^{17},1^{8}$\\
 & $11$ & $11^{18},9,5^{6},3,1^{8}$\\
 & $13$ & $13^{10},11,9^{6},7,5^{7},3,1^{8}$\\
 & others & $17,15,13^{6},11^{2},9^{7},7,5^{7},3,1^{8}$\\
\hline
$E_{7}(a_{4})$ & $2$ & $8^{22},7^{2},6^{8},2^{2},1^{6}$\\
 & $3$ & $9^{19},8^{2},7^{4},6^{2},3^{6},1^{3}$\\
 & $5$ & $15,13,12^{2},11^{2},10^{8},8^{2},7,5^{10},4^{2},3^{2},2^{2},1^{3}$\\
 & $11$ & $11^{12},10^{2},9,8^{2},7^{2},6^{2},5^{2},4^{4},3^{4},2^{2},1^{3}$\\
 & $13$ & $13^{3},12^{2},11^{3},10^{4},9^{2},8^{2},7^{3},6^{2},5^{2},4^{4},3^{4},2^{2},1^{3}$\\
 & others & $15,13,12^{2},11^{4},10^{4},9^{2},8^{2},7^{3},6^{2},5^{2},4^{4},3^{4},2^{2},1^{3}$\\
\hline
$A_{6}+A_{1}$ & $2$ & $8^{22},7^{2},6^{8},2^{2},1^{6}$\\
 & $3$ & $9^{19},8^{2},7^{4},6^{2},3^{6},1^{3}$\\
 & $7$ & $7^{35},3$\\
 & $11$ & $11^{11},8^{4},7^{5},6^{4},5^{3},4^{2},3^{2},2^{2},1^{3}$\\
 & others & $13^{3},12^{2},11,10^{2},9^{3},8^{4},7^{5},6^{4},5^{3},4^{2},3^{2},2^{2},1^{3}$\\
\hline
$D_{6}(a_{1})$ & $2$ & $8^{22},7^{2},6^{8},2^{2},1^{6}$\\
 & $3$ & $9^{19},7^{8},3^{5},1^{6}$\\
 & $5$ & $15,12^{4},11,10^{6},9^{4},7,5^{10},3^{5},1^{6}$\\
 & $11$ & $11^{12},9^{5},7,6^{4},4^{4},3^{6},1^{6}$\\
 & $13$ & $13^{2},12^{4},11,10^{4},9^{5},7^{2},6^{4},4^{4},3^{6},1^{6}$\\
 & others & $15,12^{4},11^{2},10^{4},9^{5},7^{2},6^{4},4^{4},3^{6},1^{6}$\\
\hline
$(A_{6})^{(2)}$ & $2$ & $8^{22},7^{2},6^{8},2^{2},1^{6}$\\
\hline
$A_{6}$ & $2$ & $8^{14},7^{18},1^{10}$\\
 & $3$ & $9^{19},7^{8},3^{5},1^{6}$\\
 & $7$ & $7^{35},1^{3}$\\
 & $11$ & $11^{11},7^{13},5^{3},3^{5},1^{6}$\\
 & others & $13^{3},11^{5},9^{3},7^{13},5^{3},3^{5},1^{6}$\\
\hline
$E_{8}(a_{7})$ & $2$ & $8^{20},7^{4},4^{12},3^{4}$\\
 & $3$ & $9^{14},7^{2},6^{8},5^{6},3^{10}$\\
 & $5$ & $11,10^{6},9^{3},7^{5},5^{20},3^{5}$\\
 & $7$ & $7^{30},5^{4},3^{6}$\\
 & others & $11^{4},9^{6},7^{10},5^{10},3^{10}$\\
\hline
$D_{5}+A_{1}$ & $2$ & $8^{20},6^{4},5^{4},4^{6},3^{2},2^{6},1^{2}$\\
 & $3$ & $9^{15},8^{8},7^{4},3,2^{6},1^{6}$\\
 & $5$ & $15,12^{2},11^{5},10^{4},9^{3},8^{2},5^{10},3,2^{6},1^{6}$\\
 & $11$ & $11^{11},10^{2},9^{3},8^{2},6^{2},5^{4},4^{2},3^{2},2^{6},1^{6}$\\
 & $13$ & $13^{2},12^{2},11^{4},10^{4},9^{3},8^{2},7,6^{2},5^{4},4^{2},3^{2},2^{6},1^{6}$\\
 & others & $15,12^{2},11^{5},10^{4},9^{3},8^{2},7,6^{2},5^{4},4^{2},3^{2},2^{6},1^{6}$\\
\hline
$E_{7}(a_{5})$ & $2$ & $8^{20},6^{4},5^{4},4^{6},3^{2},2^{6},1^{2}$\\
 & $3$ & $9^{13},8^{2},7,6^{8},5^{3},4^{6},3^{6},1^{3}$\\
 & $5$ & $11,10^{6},9,8^{4},7^{2},6^{2},5^{18},4^{2},3^{3},1^{3}$\\
 & $7$ & $7^{29},6^{2},5,4^{4},3^{3},1^{3}$\\
 & others & $11^{3},10^{2},9^{3},8^{4},7^{5},6^{6},5^{4},4^{6},3^{6},1^{3}$\\
\hline
$E_{6}(a_{3})+A_{1}$ & $2$ & $8^{20},6^{4},5^{4},4^{6},3^{2},2^{6},1^{2}$\\
 & $3$ & $9^{9},8^{2},6^{16},3^{17},2^{2}$\\
 & $5$ & $11,10^{4},9^{3},8^{4},7^{3},6^{2},5^{17},3^{3},2^{4},1^{3}$\\
 & $7$ & $7^{28},6^{2},5^{3},4^{2},3^{2},2^{4},1^{3}$\\
 & others & $11^{2},10^{2},9^{4},8^{4},7^{4},6^{6},5^{7},4^{4},3^{4},2^{4},1^{3}$\\
\hline
$D_{6}(a_{2})$ & $2$ & $8^{12},7^{4},6^{16},2^{12},1^{4}$\\
 & $3$ & $9^{13},7^{5},6^{4},5^{5},4^{8},3^{3},1^{6}$\\
 & $5$ & $11,10^{6},8^{4},7^{5},5^{17},4^{4},3,1^{6}$\\
 & $7$ & $7^{29},5^{4},4^{4},3,1^{6}$\\
 & others & $11^{2},10^{4},9,8^{4},7^{7},6^{4},5^{5},4^{8},3^{3},1^{6}$\\
\hline
$D_{5}(a_{1})+A_{2}$ & $2$ & $8^{10},7^{6},4^{30},3^{2}$\\
 & $3$ & $9^{9},8^{2},6^{16},3^{17},2^{2}$\\
 & $5$ & $11,10^{2},9^{3},8^{6},7^{4},6^{4},5^{10},4^{2},3^{7},2^{4},1^{3}$\\
 & $7$ & $7^{28},5^{3},4^{2},3^{6},2^{4},1^{3}$\\
 & others & $11,10^{2},9^{3},8^{6},7^{4},6^{6},5^{6},4^{4},3^{7},2^{4},1^{3}$\\
\hline
$A_{5}+A_{1}$ & $2$ & $8^{12},7^{4},6^{16},2^{12},1^{4}$\\
 & $3$ & $9^{9},7^{4},6^{14},3^{16},2^{2},1^{3}$\\
 & $5$ & $11,10^{4},9^{3},8^{2},7^{4},6^{6},5^{13},4^{2},3,2^{4},1^{6}$\\
 & $7$ & $7^{25},6^{6},5^{4},3,2^{4},1^{6}$\\
 & others & $11,10^{4},9^{3},8^{2},7^{5},6^{8},5^{7},4^{4},3^{2},2^{4},1^{6}$\\
\hline
$A_{4}+A_{3}$ & $2$ & $8^{10},7^{6},4^{30},3^{2}$\\
 & $3$ & $9^{7},8^{2},7^{6},6^{6},5^{6},4^{6},3^{10},2^{2},1^{3}$\\
 & $5$ & $5^{48},4^{2}$\\
 & $7$ & $7^{24},6^{2},5^{3},4^{6},3^{6},2^{4},1^{3}$\\
 & others & $10^{2},9^{3},8^{4},7^{6},6^{6},5^{6},4^{8},3^{6},2^{4},1^{3}$\\
\hline
$D_{5}$ & $2$ & $8^{20},5^{12},4^{2},3^{2},1^{14}$\\
 & $3$ & $9^{11},8^{16},1^{21}$\\
 & $5$ & $15,11^{9},9^{7},5^{10},1^{21}$\\
 & $11$ & $11^{11},9^{7},5^{8},3,1^{21}$\\
 & $13$ & $13^{2},11^{8},9^{7},7,5^{8},3,1^{21}$\\
 & others & $15,11^{9},9^{7},7,5^{8},3,1^{21}$\\
\hline
$E_{6}(a_{3})$ & $2$ & $8^{20},5^{12},4^{2},3^{2},1^{14}$\\
 & $3$ & $9^{9},8^{2},6^{16},3^{16},1^{7}$\\
 & $5$ & $11,10^{2},9^{7},7^{7},5^{17},3^{2},1^{14}$\\
 & $7$ & $7^{28},5^{7},3,1^{14}$\\
 & others & $11^{2},9^{8},7^{8},5^{15},3^{3},1^{14}$\\
\hline
$(D_{4}+A_{2})^{(2)}$ & $2$ & $8^{10},7^{2},6^{4},5^{4},4^{24},2^{6},1^{2}$\\
\hline
$D_{4}+A_{2}$ & $2$ & $4^{56},3^{8}$\\
 & $3$ & $9^{8},7^{7},6^{12},3^{16},1^{7}$\\
 & $7$ & $7^{28},5,3^{13},1^{8}$\\
 & others & $11,9^{6},7^{14},5^{7},3^{14},1^{8}$\\
\hline
$A_{4}+A_{2}+A_{1}$ & $2$ & $8^{10},7^{2},6^{4},5^{4},4^{24},2^{6},1^{2}$\\
 & $3$ & $9^{3},8^{2},7,6^{18},5^{3},3^{25}$\\
 & $5$ & $5^{46},4^{2},3^{2},2^{2}$\\
 & $7$ & $7^{19},6^{2},5^{7},4^{8},3^{7},2^{6},1^{3}$\\
 & others & $9^{3},8^{4},7^{5},6^{6},5^{10},4^{8},3^{7},2^{6},1^{3}$\\
\hline
$D_{5}(a_{1})+A_{1}$ & $2$ & $8^{10},7^{2},6^{4},5^{4},4^{24},2^{6},1^{2}$\\
 & $3$ & $9^{5},8^{6},7^{4},6^{12},3^{11},2^{8},1^{6}$\\
 & $7$ & $7^{28},4^{2},3^{6},2^{10},1^{6}$\\
 & others & $11,9^{3},8^{6},7^{8},6^{6},5^{3},4^{2},3^{7},2^{10},1^{6}$\\
\hline
$A_{5}$ & $2$ & $8^{8},7^{12},6^{12},2^{6},1^{16}$\\
 & $3$ & $9^{9},7^{4},6^{14},3^{15},1^{10}$\\
 & $5$ & $11,10^{2},9^{7},6^{14},5^{9},4^{2},1^{17}$\\
 & $7$ & $7^{21},6^{14},1^{17}$\\
 & others & $11,10^{2},9^{7},7,6^{14},5^{7},4^{2},3,1^{17}$\\
\hline
$A_{4}+A_{2}$ & $2$ & $8^{6},7^{10},5^{4},4^{20},3^{8},1^{6}$\\
 & $3$ & $9^{3},7^{5},6^{16},5^{3},3^{24},1^{3}$\\
 & $5$ & $5^{46},3^{5},1^{3}$\\
 & $7$ & $7^{19},5^{11},3^{18},1^{6}$\\
 & others & $9^{3},7^{13},5^{14},3^{18},1^{6}$\\
\hline
$A_{4}+2A_{1}$ & $2$ & $8^{6},7^{2},6^{12},5^{4},4^{14},3^{2},2^{16}$\\
 & $3$ & $9,8^{4},7^{3},6^{12},5^{7},4^{4},3^{17},2^{4},1^{4}$\\
 & $5$ & $5^{45},3^{4},2^{4},1^{3}$\\
 & $7$ & $7^{15},6^{4},5^{8},4^{8},3^{9},2^{8},1^{4}$\\
 & others & $9,8^{4},7^{5},6^{8},5^{9},4^{8},3^{9},2^{8},1^{4}$\\
\hline
$D_{5}(a_{1})$ & $2$ & $8^{10},7^{2},5^{12},4^{20},1^{14}$\\
 & $3$ & $9^{3},8^{8},7^{6},6^{10},3^{8},2^{8},1^{15}$\\
 & $7$ & $7^{28},3^{7},2^{8},1^{15}$\\
 & others & $11,9,8^{8},7^{8},6^{8},5,3^{8},2^{8},1^{15}$\\
\hline
$2A_{3}$ & $2$ & $4^{56},3^{8}$\\
 & $5$ & $5^{38},4^{12},1^{10}$\\
 & $7$ & $7^{14},5^{10},4^{16},3^{6},2^{4},1^{10}$\\
 & others & $8^{4},7^{6},6^{4},5^{10},4^{16},3^{6},2^{4},1^{10}$\\
\hline
$A_{4}+A_{1}$ & $2$ & $8^{4},7^{6},6^{8},5^{8},4^{10},3^{6},2^{10},1^{8}$\\
 & $3$ & $9,8^{2},7^{7},6^{8},5^{9},4^{6},3^{12},2^{6},1^{9}$\\
 & $5$ & $5^{45},3,2^{6},1^{8}$\\
 & $7$ & $7^{13},6^{6},5^{8},4^{8},3^{8},2^{8},1^{9}$\\
 & others & $9,8^{2},7^{7},6^{8},5^{9},4^{8},3^{8},2^{8},1^{9}$\\
\hline
$D_{4}(a_{1})+A_{2}$ & $2$ & $4^{56},3^{8}$\\
 & $3$ & $7,6^{14},5^{6},3^{42},1$\\
 & $5$ & $5^{36},3^{20},1^{8}$\\
 & others & $7^{8},5^{20},3^{28},1^{8}$\\
\hline
$D_{4}+A_{1}$ & $2$ & $4^{54},3^{2},2^{10},1^{6}$\\
 & $3$ & $9^{2},8^{6},7^{13},6^{6},3^{2},2^{14},1^{21}$\\
 & $7$ & $7^{28},3,2^{14},1^{21}$\\
 & others & $11,8^{6},7^{14},6^{6},3^{2},2^{14},1^{21}$\\
\hline
$A_{3}+A_{2}+A_{1}$ & $2$ & $4^{54},3^{2},2^{10},1^{6}$\\
 & $3$ & $7,6^{14},5^{3},4^{6},3^{37},2^{2},1^{3}$\\
 & $5$ & $5^{32},4^{8},3^{10},2^{10},1^{6}$\\
 & others & $7^{5},6^{6},5^{10},4^{14},3^{15},2^{10},1^{6}$\\
\hline
$A_{4}$ & $2$ & $8^{2},7^{10},5^{20},4^{2},3^{10},1^{24}$\\
 & $5$ & $5^{45},1^{23}$\\
 & $7$ & $7^{13},5^{20},3^{11},1^{24}$\\
 & others & $9,7^{11},5^{21},3^{11},1^{24}$\\
\hline
$(A_{3}+A_{2})^{(2)}$ & $2$ & $4^{54},3^{2},2^{10},1^{6}$\\
\hline
$A_{3}+A_{2}$ & $2$ & $4^{46},3^{18},1^{10}$\\
 & $3$ & $7,6^{12},5^{5},4^{8},3^{34},1^{10}$\\
 & $5$ & $5^{30},4^{8},3^{13},2^{8},1^{11}$\\
 & others & $7^{3},6^{8},5^{8},4^{16},3^{16},2^{8},1^{11}$\\
\hline
$D_{4}(a_{1})+A_{1}$ & $2$ & $4^{54},3^{2},2^{10},1^{6}$\\
 & $3$ & $7,6^{8},5^{12},4^{6},3^{28},2^{8},1^{9}$\\
 & $5$ & $5^{29},4^{6},3^{14},2^{14},1^{9}$\\
 & others & $7^{2},6^{6},5^{13},4^{12},3^{16},2^{14},1^{9}$\\
\hline
$A_{3}+2A_{1}$ & $2$ & $4^{44},3^{4},2^{28},1^{4}$\\
 & $3$ & $7,6^{6},5^{11},4^{14},3^{19},2^{12},1^{13}$\\
 & $5$ & $5^{25},4^{10},3^{14},2^{14},1^{13}$\\
 & others & $7,6^{6},5^{11},4^{16},3^{15},2^{14},1^{13}$\\
\hline
$2A_{2}+2A_{1}$ & $2$ & $4^{44},3^{4},2^{28},1^{4}$\\
 & $3$ & $3^{80},2^{4}$\\
 & $5$ & $5^{18},4^{12},3^{20},2^{20},1^{10}$\\
 & others & $6^{4},5^{10},4^{16},3^{20},2^{20},1^{10}$\\
\hline
$D_{4}$ & $2$ & $4^{54},3^{2},1^{26}$\\
 & $3$ & $9^{2},7^{25},3,1^{52}$\\
 & $7$ & $7^{28},1^{52}$\\
 & others & $11,7^{26},3,1^{52}$\\
\hline
$D_{4}(a_{1})$ & $2$ & $4^{54},3^{2},1^{26}$\\
 & $3$ & $7,6^{2},5^{24},3^{27},1^{28}$\\
 & $5$ & $5^{29},3^{25},1^{28}$\\
 & others & $7^{2},5^{25},3^{27},1^{28}$\\
\hline
$A_{3}+A_{1}$ & $2$ & $4^{40},3^{12},2^{18},1^{16}$\\
 & $5$ & $5^{21},4^{16},3^{9},2^{14},1^{24}$\\
 & others & $7,6^{2},5^{15},4^{18},3^{10},2^{14},1^{24}$\\
\hline
$2A_{2}+A_{1}$ & $2$ & $4^{40},3^{12},2^{18},1^{16}$\\
 & $3$ & $3^{79},2^{2},1^{7}$\\
 & $5$ & $5^{14},4^{14},3^{23},2^{18},1^{17}$\\
 & others & $6^{2},5^{10},4^{16},3^{23},2^{18},1^{17}$\\
\hline
$2A_{2}$ & $2$ & $4^{28},3^{36},1^{28}$\\
 & $3$ & $3^{78},1^{14}$\\
 & others & $5^{14},3^{50},1^{28}$\\
\hline
$A_{2}+3A_{1}$ & $2$ & $4^{26},3^{6},2^{62},1^{2}$\\
 & $3$ & $3^{70},2^{14},1^{10}$\\
 & others & $5^{7},4^{14},3^{28},2^{28},1^{17}$\\
\hline
$A_{3}$ & $2$ & $4^{36},3^{20},1^{44}$\\
 & $5$ & $5^{13},4^{32},1^{55}$\\
 & others & $7,5^{11},4^{32},3,1^{55}$\\
\hline
$A_{2}+2A_{1}$ & $2$ & $4^{22},3^{14},2^{52},1^{14}$\\
 & $3$ & $3^{65},2^{16},1^{21}$\\
 & others & $5^{3},4^{16},3^{27},2^{32},1^{24}$\\
\hline
$A_{2}+A_{1}$ & $2$ & $4^{14},3^{30},2^{34},1^{34}$\\
 & $3$ & $3^{58},2^{20},1^{34}$\\
 & others & $5,4^{12},3^{32},2^{32},1^{35}$\\
\hline
$4A_{1}$ & $2$ & $2^{120},1^{8}$\\
 & $3$ & $3^{44},2^{40},1^{36}$\\
 & others & $4^{8},3^{28},2^{48},1^{36}$\\
\hline
$A_{2}$ & $2$ & $4^{2},3^{54},1^{78}$\\
 & $3$ & $3^{57},1^{77}$\\
 & others & $5,3^{55},1^{78}$\\
\hline
$3A_{1}$ & $2$ & $2^{110},1^{28}$\\
 & $3$ & $3^{31},2^{50},1^{55}$\\
 & others & $4^{2},3^{27},2^{52},1^{55}$\\
\hline
$2A_{1}$ & $2$ & $2^{92},1^{64}$\\
 & others & $3^{14},2^{64},1^{78}$\\
\hline
$A_{1}$ & $2$ & $2^{58},1^{132}$\\
 & others & $3,2^{56},1^{133}$\\
\hline
\end{longtable}\unskip\unpenalty\unpenalty}\unvbox\ltmcbox\end{multicols}
\smallskip\begin{table}[!h]\caption{Jordan Blocks of nilpotent elements on the adjoint module for $E_8$\label{e8a}}\end{table}\newpage

\begin{multicols}{2}\setbox\ltmcbox\vbox{\makeatletter\col@number\@ne\tiny\begin{longtable}{l l l}$E_{7}$ & $2$ & $16^{6},15^{2},2,1^{5}$\\
 & $3$ & $27^{3},9^{4},8^{2}$\\
 & $5$ & $25^{4},15^{2},3$\\
 & $7$ & $35,21^{4},7^{2}$\\
 & $13$ & $35,27,23,19,13^{2},3$\\
 & $17$ & $35,27,23,17^{2},11,3$\\
 & $19$ & $19^{7}$\\
 & $23$ & $23^{5},15,3$\\
 & $29$ & $29^{2},27,19,15,11,3$\\
 & $31$ & $31^{2},23,19,15,11,3$\\
 & others & $35,27,23,19,15,11,3$\\
\hline
$E_{7}(a_{1})$ & $2$ & $16^{6},15^{2},2,1^{5}$\\
 & $3$ & $27,19,17^{3},9^{4}$\\
 & $5$ & $25^{2},19,17,15,11^{2},7,3$\\
 & $7$ & $27,21^{2},15,14^{2},7^{3}$\\
 & $11$ & $23,22^{2},11^{6}$\\
 & $13$ & $27,23,19,17,13^{2},11,7,3$\\
 & $17$ & $17^{7},11,3$\\
 & $19$ & $19^{5},17,11,7,3$\\
 & $23$ & $23^{3},17,15,11^{2},7,3$\\
 & others & $27,23,19,17,15,11^{2},7,3$\\
\hline
$E_{7}(a_{2})$ & $2$ & $16^{4},13^{4},4^{2},3^{2},2,1$\\
 & $3$ & $19,17^{3},9^{7}$\\
 & $5$ & $23,19,17,15^{2},11,10^{2},7,3^{2}$\\
 & $13$ & $13^{10},3$\\
 & $17$ & $17^{5},15,11,9,7,3^{2}$\\
 & $19$ & $19^{3},17,15,11^{2},9,7,3^{2}$\\
 & others & $23,19,17,15^{2},11^{2},9,7,3^{2}$\\
\hline
$E_{7}(a_{3})$ & $2$ & $16^{2},15^{2},9^{4},8^{4},2,1$\\
 & $3$ & $19,17,15^{2},11^{2},9^{3},6^{2},3^{2}$\\
 & $5$ & $19,17,15^{2},11^{2},10^{2},7,5^{3},3$\\
 & $11$ & $11^{11},9,3$\\
 & $13$ & $13^{8},11,7,5,3^{2}$\\
 & $17$ & $17^{3},15,11^{3},9,7^{2},5,3^{2}$\\
 & others & $19,17,15^{2},11^{3},9,7^{2},5,3^{2}$\\
\hline
$E_{6}$ & $2$ & $16^{4},13^{4},4^{2},3^{2},1^{3}$\\
 & $3$ & $9^{13},8^{2}$\\
 & $13$ & $13^{10},1^{3}$\\
 & $17$ & $17^{5},15,9^{3},3,1^{3}$\\
 & $19$ & $19^{2},17^{3},11,9^{3},3,1^{3}$\\
 & others & $23,17^{3},15,11,9^{3},3,1^{3}$\\
\hline
$E_{6}(a_{1})$ & $2$ & $16^{2},13^{2},11^{2},9^{2},8^{2},5^{2},4^{2},1$\\
 & $3$ & $9^{13},8^{2}$\\
 & $5$ & $17,15,13^{2},11,10^{2},9^{2},7,5^{3},3,1$\\
 & $7$ & $15,14^{2},13^{2},7^{9},1$\\
 & $11$ & $11^{10},9,5^{2},3,1$\\
 & $13$ & $13^{6},11,9^{2},7,5^{3},3,1$\\
 & others & $17,15,13^{2},11^{2},9^{3},7,5^{3},3,1$\\
\hline
$D_{6}$ & $2$ & $8^{14},7^{2},2,1^{5}$\\
 & $3$ & $19,16^{2},15,11,10^{2},9^{2},6^{2},3,1^{3}$\\
 & $5$ & $19,16^{2},15,11^{2},10^{2},6^{2},5^{2},1^{3}$\\
 & $11$ & $11^{10},10^{2},1^{3}$\\
 & $13$ & $13^{8},11,6^{2},3,1^{3}$\\
 & $17$ & $17^{2},16^{2},11^{2},10^{2},7,6^{2},3,1^{3}$\\
 & others & $19,16^{2},15,11^{2},10^{2},7,6^{2},3,1^{3}$\\
\hline
$E_{7}(a_{4})$ & $2$ & $8^{14},7^{2},2,1^{5}$\\
 & $3$ & $9^{11},8^{2},6^{2},3^{2}$\\
 & $5$ & $15,13,11^{2},10^{4},7,5^{6},3^{2}$\\
 & $11$ & $11^{8},9,7^{2},5^{2},3^{4}$\\
 & $13$ & $13^{3},11^{3},9^{2},7^{3},5^{2},3^{4}$\\
 & others & $15,13,11^{4},9^{2},7^{3},5^{2},3^{4}$\\
\hline
$D_{6}(a_{1})$ & $2$ & $8^{14},7^{2},2,1^{5}$\\
 & $3$ & $9^{11},7^{4},3,1^{3}$\\
 & $5$ & $15,12^{2},11,10^{4},7,5^{6},3,1^{3}$\\
 & $11$ & $11^{8},9,7,6^{2},4^{2},3^{2},1^{3}$\\
 & $13$ & $13^{2},12^{2},11,10^{2},9,7^{2},6^{2},4^{2},3^{2},1^{3}$\\
 & others & $15,12^{2},11^{2},10^{2},9,7^{2},6^{2},4^{2},3^{2},1^{3}$\\
\hline
$D_{5}+A_{1}$ & $2$ & $8^{12},5^{4},4^{2},3^{2},2,1$\\
 & $3$ & $9^{11},7^{4},3,1^{3}$\\
 & $5$ & $15,12^{2},11,10^{2},9^{3},5^{6},3,1^{3}$\\
 & $11$ & $11^{7},9^{3},6^{2},4^{2},3^{2},1^{3}$\\
 & $13$ & $13^{2},12^{2},10^{2},9^{3},7,6^{2},4^{2},3^{2},1^{3}$\\
 & others & $15,12^{2},11,10^{2},9^{3},7,6^{2},4^{2},3^{2},1^{3}$\\
\hline
$(A_{6})^{(2)}$ & 2 & $8^{14},7^{2},2,1^{5}$\\
\hline
$A_{6}$ & $2$ & $8^{14},7^{2},1^{7}$\\
 & $3$ & $9^{11},7^{4},3,1^{3}$\\
 & $7$ & $7^{19}$\\
 & $11$ & $11^{7},7^{5},5^{3},3,1^{3}$\\
 & others & $13^{3},11,9^{3},7^{5},5^{3},3,1^{3}$\\
\hline
$E_{7}(a_{5})$ & $2$ & $8^{12},5^{4},4^{2},3^{2},2,1$\\
 & $3$ & $9^{9},7,6^{2},5^{3},3^{6}$\\
 & $5$ & $11,10^{4},9,7^{2},5^{10},3^{3}$\\
 & $7$ & $7^{17},5,3^{3}$\\
 & others & $11^{3},9^{3},7^{5},5^{4},3^{6}$\\
\hline
$D_{5}$ & $2$ & $8^{12},5^{4},4^{2},3^{2},1^{3}$\\
 & $3$ & $9^{7},8^{8},1^{6}$\\
 & $5$ & $15,11^{5},9^{3},5^{6},1^{6}$\\
 & $11$ & $11^{7},9^{3},5^{4},3,1^{6}$\\
 & $13$ & $13^{2},11^{4},9^{3},7,5^{4},3,1^{6}$\\
 & others & $15,11^{5},9^{3},7,5^{4},3,1^{6}$\\
\hline
$E_{6}(a_{3})$ & $2$ & $8^{12},5^{4},4^{2},3^{2},1^{3}$\\
 & $3$ & $9^{5},8^{2},6^{8},3^{8}$\\
 & $5$ & $11,10^{2},9^{3},7^{3},5^{9},3^{2},1^{3}$\\
 & $7$ & $7^{16},5^{3},3,1^{3}$\\
 & others & $11^{2},9^{4},7^{4},5^{7},3^{3},1^{3}$\\
\hline
$D_{6}(a_{2})$ & $2$ & $8^{8},7^{4},6^{4},2^{7},1^{3}$\\
 & $3$ & $9^{9},7,6^{2},5,4^{4},3^{3},1^{3}$\\
 & $5$ & $11,10^{4},8^{2},7,5^{9},4^{2},3,1^{3}$\\
 & $7$ & $7^{17},4^{2},3,1^{3}$\\
 & others & $11^{2},10^{2},9,8^{2},7^{3},6^{2},5,4^{4},3^{3},1^{3}$\\
\hline
$D_{5}(a_{1})+A_{1}$ & $2$ & $8^{6},7^{2},5^{4},4^{12},2,1$\\
 & $3$ & $9^{5},7^{4},6^{6},3^{7},1^{3}$\\
 & $7$ & $7^{16},3^{6},1^{3}$\\
 & others & $11,9^{3},7^{8},5^{3},3^{7},1^{3}$\\
\hline
$A_{5}+A_{1}$ & $2$ & $8^{8},7^{4},6^{4},2^{7},1^{3}$\\
 & $3$ & $9^{9},3^{16},2^{2}$\\
 & $5$ & $11,10^{2},9^{3},8^{2},5^{9},3,2^{4},1^{3}$\\
 & $7$ & $7^{17},3,2^{4},1^{3}$\\
 & others & $11,10^{2},9^{3},8^{2},7,6^{2},5^{3},4^{2},3^{2},2^{4},1^{3}$\\
\hline
$(A_{5})'$ & $2$ & $8^{8},7^{4},6^{4},2^{6},1^{5}$\\
 & $3$ & $9^{5},7^{4},6^{6},3^{7},1^{3}$\\
 & $5$ & $11,10^{2},9^{3},6^{6},5^{5},4^{2},1^{6}$\\
 & $7$ & $7^{13},6^{6},1^{6}$\\
 & others & $11,10^{2},9^{3},7,6^{6},5^{3},4^{2},3,1^{6}$\\
\hline
$A_{4}+A_{2}$ & $2$ & $8^{6},7^{2},5^{4},4^{12},1^{3}$\\
 & $3$ & $9^{3},7,6^{8},5^{3},3^{12}$\\
 & $5$ & $5^{26},3$\\
 & $7$ & $7^{11},5^{7},3^{6},1^{3}$\\
 & others & $9^{3},7^{5},5^{10},3^{6},1^{3}$\\
\hline
$D_{5}(a_{1})$ & $2$ & $8^{6},7^{2},5^{4},4^{12},1^{3}$\\
 & $3$ & $9^{3},8^{4},7^{2},6^{6},3^{4},2^{4},1^{4}$\\
 & $7$ & $7^{16},3^{3},2^{4},1^{4}$\\
 & others & $11,9,8^{4},7^{4},6^{4},5,3^{4},2^{4},1^{4}$\\
\hline
$A_{4}+A_{1}$ & $2$ & $8^{4},7^{2},6^{4},5^{4},4^{6},3^{2},2^{6},1$\\
 & $3$ & $9,8^{2},7^{3},6^{4},5^{5},4^{2},3^{8},2^{2},1^{2}$\\
 & $5$ & $5^{25},3,2^{2},1$\\
 & $7$ & $7^{9},6^{2},5^{4},4^{4},3^{4},2^{4},1^{2}$\\
 & others & $9,8^{2},7^{3},6^{4},5^{5},4^{4},3^{4},2^{4},1^{2}$\\
\hline
$D_{4}+A_{1}$ & $2$ & $4^{30},3^{2},2,1^{5}$\\
 & $3$ & $9^{2},8^{4},7^{5},6^{4},3^{2},2^{4},1^{10}$\\
 & $7$ & $7^{16},3,2^{4},1^{10}$\\
 & others & $11,8^{4},7^{6},6^{4},3^{2},2^{4},1^{10}$\\
\hline
$(A_{5})''$ & $2$ & $8^{4},7^{12},2,1^{15}$\\
 & $3$ & $9^{9},3^{15},1^{7}$\\
 & $5$ & $11,9^{7},5^{9},1^{14}$\\
 & $7$ & $7^{17},1^{14}$\\
 & others & $11,9^{7},7,5^{7},3,1^{14}$\\
\hline
$A_{3}+A_{2}+A_{1}$ & $2$ & $4^{30},3^{2},2,1^{5}$\\
 & $3$ & $7,6^{8},5^{3},3^{21}$\\
 & $5$ & $5^{20},3^{10},1^{3}$\\
 & others & $7^{5},5^{10},3^{15},1^{3}$\\
\hline
$A_{4}$ & $2$ & $8^{2},7^{6},5^{8},4^{2},3^{6},1^{9}$\\
 & $5$ & $5^{25},1^{8}$\\
 & $7$ & $7^{9},5^{8},3^{7},1^{9}$\\
 & others & $9,7^{7},5^{9},3^{7},1^{9}$\\
\hline
$(A_{3}+A_{2})^{(2)}$ & 2 & $4^{30},3^{2},2,1^{5}$\\
\hline
$A_{3}+A_{2}$ & $2$ & $4^{30},3^{2},1^{7}$\\
 & $3$ & $7,6^{8},5,4^{4},3^{18},1^{3}$\\
 & $5$ & $5^{18},4^{4},3^{5},2^{4},1^{4}$\\
 & others & $7^{3},6^{4},5^{4},4^{8},3^{8},2^{4},1^{4}$\\
\hline
$D_{4}(a_{1})+A_{1}$ & $2$ & $4^{30},3^{2},2,1^{5}$\\
 & $3$ & $7,6^{6},5^{4},4^{4},3^{16},1^{6}$\\
 & $5$ & $5^{17},4^{4},3^{6},2^{4},1^{6}$\\
 & others & $7^{2},6^{4},5^{5},4^{8},3^{8},2^{4},1^{6}$\\
\hline
$D_{4}$ & $2$ & $4^{30},3^{2},1^{7}$\\
 & $3$ & $9^{2},7^{13},3,1^{21}$\\
 & $7$ & $7^{16},1^{21}$\\
 & others & $11,7^{14},3,1^{21}$\\
\hline
$A_{3}+2A_{1}$ & $2$ & $4^{24},3^{4},2^{11},1^{3}$\\
 & $3$ & $7,6^{4},5^{7},4^{4},3^{11},2^{6},1^{6}$\\
 & $5$ & $5^{17},4^{2},3^{6},2^{8},1^{6}$\\
 & others & $7,6^{4},5^{7},4^{6},3^{7},2^{8},1^{6}$\\
\hline
$D_{4}(a_{1})$ & $2$ & $4^{30},3^{2},1^{7}$\\
 & $3$ & $7,6^{2},5^{12},3^{15},1^{9}$\\
 & $5$ & $5^{17},3^{13},1^{9}$\\
 & others & $7^{2},5^{13},3^{15},1^{9}$\\
\hline
$(A_{3}+A_{1})'$ & $2$ & $4^{24},3^{4},2^{10},1^{5}$\\
 & $5$ & $5^{13},4^{8},3^{5},2^{6},1^{9}$\\
 & others & $7,6^{2},5^{7},4^{10},3^{6},2^{6},1^{9}$\\
\hline
$2A_{2}+A_{1}$ & $2$ & $4^{24},3^{4},2^{10},1^{5}$\\
 & $3$ & $3^{43},2^{2}$\\
 & $5$ & $5^{10},4^{6},3^{11},2^{10},1^{6}$\\
 & others & $6^{2},5^{6},4^{8},3^{11},2^{10},1^{6}$\\
\hline
$(A_{3}+A_{1})''$ & $2$ & $4^{20},3^{12},2,1^{15}$\\
 & $5$ & $5^{17},3^{9},1^{21}$\\
 & others & $7,5^{15},3^{10},1^{21}$\\
\hline
$A_{2}+3A_{1}$ & $2$ & $4^{14},3^{6},2^{29},1$\\
 & $3$ & $3^{42},1^{7}$\\
 & others & $5^{7},3^{28},1^{14}$\\
\hline
$2A_{2}$ & $2$ & $4^{20},3^{12},1^{17}$\\
 & $3$ & $3^{42},1^{7}$\\
 & others & $5^{10},3^{22},1^{17}$\\
\hline
$A_{3}$ & $2$ & $4^{20},3^{12},1^{17}$\\
 & $5$ & $5^{9},4^{16},1^{24}$\\
 & others & $7,5^{7},4^{16},3,1^{24}$\\
\hline
$A_{2}+2A_{1}$ & $2$ & $4^{14},3^{6},2^{28},1^{3}$\\
 & $3$ & $3^{37},2^{8},1^{6}$\\
 & others & $5^{3},4^{8},3^{15},2^{16},1^{9}$\\
\hline
$A_{2}+A_{1}$ & $2$ & $4^{10},3^{14},2^{18},1^{15}$\\
 & $3$ & $3^{34},2^{8},1^{15}$\\
 & others & $5,4^{8},3^{16},2^{16},1^{16}$\\
\hline
$4A_{1}$ & $2$ & $2^{63},1^{7}$\\
 & $3$ & $3^{28},2^{14},1^{21}$\\
 & others & $4^{6},3^{16},2^{20},1^{21}$\\
\hline
$A_{2}$ & $2$ & $4^{2},3^{30},1^{35}$\\
 & $3$ & $3^{33},1^{34}$\\
 & others & $5,3^{31},1^{35}$\\
\hline
$(3A_{1})'$ & $2$ & $2^{62},1^{9}$\\
 & $3$ & $3^{19},2^{26},1^{24}$\\
 & others & $4^{2},3^{15},2^{28},1^{24}$\\
\hline
$(3A_{1})''$ & $2$ & $2^{53},1^{27}$\\
 & others & $3^{27},1^{52}$\\
\hline
$2A_{1}$ & $2$ & $2^{52},1^{29}$\\
 & others & $3^{10},2^{32},1^{39}$\\
\hline
$A_{1}$ & $2$ & $2^{34},1^{65}$\\
 & others & $3,2^{32},1^{66}$\\
\hline
\end{longtable}\unskip\unpenalty\unpenalty}\unvbox\ltmcbox\end{multicols}
\smallskip\begin{table}[!h]\caption{Jordan Blocks of nilpotent elements on the adjoint module for $E_7$\label{e7a}}\end{table}\newpage

\begin{multicols}{2}\setbox\ltmcbox\vbox{\makeatletter\col@number\@ne\tiny\begin{longtable}{l l l}$E_{6}$ & $2$ & $16^{4},4^{2},3^{2}$\\
 & $3$ & $9^{6},8^{3}$\\
 & $13$ & $13^{6}$\\
 & $17$ & $17^{3},15,9,3$\\
 & $19$ & $19^{2},17,11,9,3$\\
 & others & $23,17,15,11,9,3$\\
\hline
$E_{6}(a_{1})$ & $2$ & $16^{2},11^{2},8^{2},4^{2}$\\
 & $3$ & $9^{6},8^{3}$\\
 & $5$ & $17,15,11,10^{2},7,5,3$\\
 & $7$ & $15,14^{2},7^{5}$\\
 & $11$ & $11^{6},9,3$\\
 & $13$ & $13^{4},11,7,5,3$\\
 & others & $17,15,11^{2},9,7,5,3$\\
\hline
$D_{5}$ & $2$ & $8^{8},4^{2},3^{2}$\\
 & $3$ & $9^{5},8^{4},1$\\
 & $5$ & $15,11^{3},9,5^{4},1$\\
 & $11$ & $11^{5},9,5^{2},3,1$\\
 & $13$ & $13^{2},11^{2},9,7,5^{2},3,1$\\
 & others & $15,11^{3},9,7,5^{2},3,1$\\
\hline
$E_{6}(a_{3})$ & $2$ & $8^{8},4^{2},3^{2}$\\
 & $3$ & $9^{3},8^{2},6^{4},3^{3},2$\\
 & $5$ & $11,10^{2},9,7,5^{5},3^{2}$\\
 & $7$ & $7^{10},5,3$\\
 & others & $11^{2},9^{2},7^{2},5^{3},3^{3}$\\
\hline
$D_{5}(a_{1})$ & $2$ & $8^{4},7^{2},4^{8}$\\
 & $3$ & $9^{3},8^{2},6^{4},3^{2},2^{2},1$\\
 & $7$ & $7^{10},3,2^{2},1$\\
 & others & $11,9,8^{2},7^{2},6^{2},5,3^{2},2^{2},1$\\
\hline
$A_{5}$ & $2$ & $8^{8},2^{6},1^{2}$\\
 & $3$ & $9^{3},7^{4},6^{2},3^{2},2,1^{3}$\\
 & $5$ & $11,10^{2},9,6^{2},5^{3},4^{2},1^{3}$\\
 & $7$ & $7^{9},6^{2},1^{3}$\\
 & others & $11,10^{2},9,7,6^{2},5,4^{2},3,1^{3}$\\
\hline
$A_{4}+A_{1}$ & $2$ & $8^{4},6^{2},5^{2},4^{4},2^{4}$\\
 & $3$ & $9,8^{2},7,6^{2},5^{3},3^{6},1$\\
 & $5$ & $5^{15},3$\\
 & $7$ & $7^{7},5^{2},4^{2},3^{2},2^{2},1$\\
 & others & $9,8^{2},7,6^{2},5^{3},4^{2},3^{2},2^{2},1$\\
\hline
$D_{4}$ & $2$ & $4^{18},3^{2}$\\
 & $3$ & $9^{2},7^{7},3,1^{8}$\\
 & $7$ & $7^{10},1^{8}$\\
 & others & $11,7^{8},3,1^{8}$\\
\hline
$A_{4}$ & $2$ & $8^{2},7^{4},5^{2},4^{2},3^{4},1^{4}$\\
 & $5$ & $5^{15},1^{3}$\\
 & $7$ & $7^{7},5^{2},3^{5},1^{4}$\\
 & others & $9,7^{5},5^{3},3^{5},1^{4}$\\
\hline
$D_{4}(a_{1})$ & $2$ & $4^{18},3^{2}$\\
 & $3$ & $7,6^{2},5^{6},3^{9},1^{2}$\\
 & $5$ & $5^{11},3^{7},1^{2}$\\
 & others & $7^{2},5^{7},3^{9},1^{2}$\\
\hline
$A_{3}+A_{1}$ & $2$ & $4^{16},2^{6},1^{2}$\\
 & $5$ & $5^{9},4^{4},3^{3},2^{2},1^{4}$\\
 & others & $7,6^{2},5^{3},4^{6},3^{4},2^{2},1^{4}$\\
\hline
$2A_{2}+A_{1}$ & $2$ & $4^{16},2^{6},1^{2}$\\
 & $3$ & $3^{24},2^{3}$\\
 & $5$ & $5^{8},4^{2},3^{5},2^{6},1^{3}$\\
 & others & $6^{2},5^{4},4^{4},3^{5},2^{6},1^{3}$\\
\hline
$A_{3}$ & $2$ & $4^{12},3^{8},1^{6}$\\
 & $5$ & $5^{7},4^{8},1^{11}$\\
 & others & $7,5^{5},4^{8},3,1^{11}$\\
\hline
$A_{2}+2A_{1}$ & $2$ & $4^{10},3^{2},2^{16}$\\
 & $3$ & $3^{23},2^{4},1$\\
 & others & $5^{3},4^{4},3^{9},2^{8},1^{4}$\\
\hline
$2A_{2}$ & $2$ & $4^{16},1^{14}$\\
 & $3$ & $3^{23},2,1^{7}$\\
 & others & $5^{8},3^{8},1^{14}$\\
\hline
$A_{2}+A_{1}$ & $2$ & $4^{8},3^{6},2^{10},1^{8}$\\
 & $3$ & $3^{22},2^{2},1^{8}$\\
 & others & $5,4^{6},3^{8},2^{8},1^{9}$\\
\hline
$A_{2}$ & $2$ & $4^{2},3^{18},1^{16}$\\
 & $3$ & $3^{21},1^{15}$\\
 & others & $5,3^{19},1^{16}$\\
\hline
$3A_{1}$ & $2$ & $2^{38},1^{2}$\\
 & $3$ & $3^{13},2^{14},1^{11}$\\
 & others & $4^{2},3^{9},2^{16},1^{11}$\\
\hline
$2A_{1}$ & $2$ & $2^{32},1^{14}$\\
 & others & $3^{8},2^{16},1^{22}$\\
\hline
$A_{1}$ & $2$ & $2^{22},1^{34}$\\
 & others & $3,2^{20},1^{35}$\\
\hline
\end{longtable}\unskip\unpenalty\unpenalty}\unvbox\ltmcbox\end{multicols}
\smallskip\begin{table}[!h]\caption{Jordan Blocks of nilpotent elements on the adjoint module for $E_6$\label{e6a}}\end{table}\newpage

\begin{multicols}{3}
\setbox\ltmcbox\vbox{\makeatletter\col@number\@ne
\tiny\begin{longtable}{l l l}$F_{4}$ & $2$ & $16^{2},4^{2},3^{4}$\\
 & $3$ & $9^{4},8^{2}$\\
 & $13$ & $13^{4}$\\
 & $17$ & $17^{2},15,3$\\
 & $19$ & $19^{2},11,3$\\
 & others & $23,15,11,3$\\
\hline
$F_{4}(a_{1})$ & $2$ & $8^{4},4^{2},3^{4}$\\
 & $3$ & $9^{4},8^{2}$\\
 & $5$ & $15,11^{2},5^{3}$\\
 & $11$ & $11^{4},5,3$\\
 & $13$ & $13^{2},11,7,5,3$\\
 & others & $15,11^{2},7,5,3$\\
\hline
$F_{4}(a_{2})$ & $2$ & $8^{4},4^{2},3^{4}$\\
 & $3$ & $9^{2},8^{2},6^{2},3^{2}$\\
 & $5$ & $11,10^{2},5^{3},3^{2}$\\
 & $7$ & $7^{7},3$\\
 & others & $11^{2},9,7,5,3^{3}$\\
\hline
$(C_{3})^{(2)}$ & 2 & $8^{4},4,3^{3},2^{3},1$\\
\hline
$C_{3}$ & $2$ & $8^{4},2^{8},1^{4}$\\
 & $3$ & $9^{2},7^{4},3,1^{3}$\\
 & $5$ & $11,10^{2},5^{2},4^{2},1^{3}$\\
 & $7$ & $7^{7},1^{3}$\\
 & others & $11,10^{2},7,4^{2},3,1^{3}$\\
\hline
$B_{3}$ & $2$ & $4^{10},3^{4}$\\
 & $3$ & $9^{2},7^{4},3,1^{3}$\\
 & $7$ & $7^{7},1^{3}$\\
 & others & $11,7^{5},3,1^{3}$\\
\hline
$F_{4}(a_{3})$ & $2$ & $4^{10},3^{4}$\\
 & $3$ & $7,6^{2},5^{3},3^{6}$\\
 & $5$ & $5^{8},3^{4}$\\
 & others & $7^{2},5^{4},3^{6}$\\
\hline
$C_{3}(a_{1})^{(2)}$ & 2 & $4^{9},3^{3},2^{3},1$\\
\hline
$C_{3}(a_{1})$ & $2$ & $4^{8},2^{8},1^{4}$\\
 & $5$ & $5^{7},4^{2},3^{2},1^{3}$\\
 & others & $7,6^{2},5,4^{4},3^{3},1^{3}$\\
\hline
$(\tilde{A}_{2}+A_{1})^{(2)}$ & 2 & $4^{9},3^{3},2^{3},1$\\
\hline
$\tilde{A}_{2}+A_{1}$ & $2$ & $4^{8},2^{8},1^{4}$\\
 & $3$ & $3^{16},2^{2}$\\
 & $5$ & $5^{7},3^{2},2^{4},1^{3}$\\
 & others & $6^{2},5^{3},4^{2},3^{2},2^{4},1^{3}$\\
\hline
$(B_{2})^{(2)}$ & 2 & $4^{7},3^{7},1^{3}$\\
\hline
$B_{2}$ & $2$ & $4^{6},3^{4},2^{5},1^{6}$\\
 & $5$ & $5^{6},4^{4},1^{6}$\\
 & others & $7,5^{4},4^{4},3,1^{6}$\\
\hline
$A_{2}+\tilde{A}_{1}$ & $2$ & $4^{6},3^{4},2^{8}$\\
 & $3$ & $3^{16},2^{2}$\\
 & others & $5^{3},4^{2},3^{6},2^{4},1^{3}$\\
\hline
$\tilde{A}_{2}$ & $2$ & $4^{8},1^{20}$\\
 & $3$ & $3^{15},1^{7}$\\
 & others & $5^{7},3,1^{14}$\\
\hline
$(A_{2})^{(2)}$ & 2 & $4^{8},2^{5},1^{10}$\\
\hline
$A_{2}$ & $2$ & $4^{2},3^{12},1^{8}$\\
 & $3$ & $3^{15},1^{7}$\\
 & others & $5,3^{13},1^{8}$\\
\hline
$A_{1}+\tilde{A}_{1}$ & $2$ & $2^{24},1^{4}$\\
 & $3$ & $3^{10},2^{8},1^{6}$\\
 & others & $4^{2},3^{6},2^{10},1^{6}$\\
\hline
$(\tilde{A}_{1})^{(2)}$ & 2 & $2^{21},1^{10}$\\
\hline
$\tilde{A}_{1}$ & $2$ & $2^{16},1^{20}$\\
 & others & $3^{7},2^{8},1^{15}$\\
\hline
$A_{1}$ & $2$ & $2^{16},1^{20}$\\
 & others & $3,2^{14},1^{21}$\\
\hline
\end{longtable}\unskip\unpenalty\unpenalty}\unvbox\ltmcbox\newcolumn\ 
\tiny\begin{tabular}{l l l}$G_{2}$ & $2$ & $4^{2},3^{2}$\\
 & $3$ & $9,3,2$\\
 & $7$ & $7^{2}$\\
 & others & $11,3$\\
\hline
$G_{2}(a_{1})$ & $2$ & $4^{2},3^{2}$\\
 & $3$ & $3^{4},2$\\
 & others & $5,3^{3}$\\
\hline
$(\tilde{A}_{1})^{(3)}$ & 3 & $3^{3},2^{2},1$\\
\hline
$\tilde{A}_{1}$ & $2$ & $2^{6},1^{2}$\\
 & $3$ & $3^{3},1^{5}$\\
 & others & $4^{2},3,1^{3}$\\
\hline
$A_{1}$ & $2$ & $2^{6},1^{2}$\\
 & $3$ & $3,2^{4},1^{3}$\\
 & others & $4^{2},3,1^{3}$\\
\hline

\end{tabular}\end{multicols}\begin{table}[!h]\caption{Jordan Blocks of nilpotent elements on the adjoint module for $F_4$ and $G_2$\label{f4g2a}}\end{table}
\FloatBarrier
\newgeometry{margin=2cm}
{\footnotesize
\bibliographystyle{amsalpha}
\bibliography{bib}}

\end{document}